\documentclass[11pt,reqno]{amsart}
\usepackage{a4wide,amsfonts,graphicx,
amssymb,stmaryrd,amscd,amsmath,latexsym,amsbsy,amsthm}
\usepackage[all,knot,poly]{xy}
\usepackage[all,knot,poly]{xy}
\usepackage{color}
\usepackage{ifpdf}
\ifpdf 
  \usepackage{graphicx}
  \DeclareGraphicsExtensions{.pdf,.png,.mps}
  \usepackage{pgf}
  \usepackage{tikz}
\else 
  \usepackage{graphicx}
  \DeclareGraphicsExtensions{.eps,.bmp}
  \usepackage{pgf}
  \usepackage{tikz}
  \usepackage{pstricks}
\fi
\usepackage{epic,bez123}
\usepackage{floatflt}
\usepackage{wrapfig}

\makeindex

\newtheorem{thm}{Theorem}[section]

\newtheorem{ex}[thm]{Example}

\newtheorem{cor}[thm]{Corollary}
\newtheorem{lem}[thm]{Lemma}
\newtheorem{prop}[thm]{Proposition}
\newtheorem{defn}[thm]{Definition}
\newtheorem{conj}[thm]{Conjecture}
\newtheorem{rem}[thm]{Remark}
\providecommand{\norm}[1]{\left\| #1 \right\|}

\newcommand{\mb}{\mathbf}
\newcommand{\mh}{\mathbb}
\newcommand{\mr}{\mathrm}
\newcommand{\mc}{\mathcal}

\newcommand{\ts}{\textstyle}
\newcommand{\ds}{\displaystyle}

\begin{document}

\title[Discrete series and formal degrees]
{Discrete series characters for affine Hecke algebras
and their formal degrees}
\author{Eric Opdam}
\address{Korteweg de Vries Institute for Mathematics\\
University of Amsterdam\\
Plantage Muidergracht 24\\
1018TV Amsterdam\\
The Netherlands\\
email: opdam@science.uva.nl}
\author{Maarten Solleveld}
\address{Mathematisches Institut\\
Universit\"at G\"otttingen\\
Bunsenstra\ss e \ 3-5\\
37073 G\"ottingen\\
Deutschland\\
email: maarten.solleveld@mathematik.uni-goettingen.de}
\date{\today}
\thanks{We thank Gert Heckman, N. Christopher Phillips
and Mark Reeder for discussions and advice}
\keywords{Affine Hecke algebra, discrete series character, formal dimension}
\subjclass[2000]{Primary 20C08; Secondary 22D25, 43A30}

\begin{abstract}
We introduce the \emph{generic central character} of an irreducible discrete
series representation of an affine Hecke algebra. Using this invariant we give a
new classification of
the irreducible discrete series characters for all abstract affine Hecke algebras
(except for the types $E_{6,7,8}^{(1)}$) with \emph{arbitrary positive parameters}
and we prove an explicit product formula for their formal degrees (in all cases).
\end{abstract}

\maketitle

\tableofcontents

\section{Introduction}

Considering the role of affine Hecke algebras in representation theory
\cite{IwMa}, \cite{Bo}, \cite{BZ}, \cite{BM0}, \cite{BM}, \cite{Mo1}, \cite{Mo2},
\cite{Lu4}, \cite{Re}, \cite{BHK}, \cite{BK1}
or in the theory of integrable models \cite{Ch2}, \cite{HO}, \cite{Ma2}, \cite{EOS}
it is natural to ask for the description of their (algebraic) representation theory
and for the properties of their representations in relation to harmonic analysis
(e.g. unitarity, temperedness, formal degrees).
An analytic approach to such questions (based on the spectral theory of
$C^*$-algebras) was first proposed by Matsumoto \cite{Mat}.
This approach to affine Hecke algebras gives rise to a program in the spirit of
Harish-Chandra's work on the harmonic analysis on locally compact groups arising
from reductive groups (for a concise account of Harish-Chandra's work in the
$p$-adic case see \cite{W}).
The main challenges to surmount on this classical route
designed to describe the tempered spectrum and the Plancherel isomorphism
(the ``philosophy of cusp forms'')
are related to understanding the basic building
blocks, the so-called discrete series characters. The most fundamental problems are:
\begin{itemize}
\item[(i)] Classify the irreducible discrete series characters.
\item[(ii)] Calculate their formal degrees.
\end{itemize}
In the present paper we will
essentially\footnote{Our solution of (i) does not cover the cases
$E_n$ ($n=6,7,8$), hence in these cases we rely on \cite{KL}.
Our solution of (ii) is complete
only up to the determination of a rational constant factor
for each continuous family (in the sense to be explained
below) of discrete series characters.}
solve both these problems
for general abstract semisimple affine Hecke algebras with
\emph{arbitrary} positive parameters.

The study of harmonic analysis in this context requires
the introduction of classical notions borrowed from Harish-Chandra's seminal
work (e.g. the Schwartz completion, temperedness, parabolic induction) for
abstract affine Hecke algebras.
It was shown in \cite{DeOp1} that the above program can
indeed be carried out.
In view of \cite{DeOp1} (also see \cite{Opd2})
our solution of (i) can in fact be amplified
to yield the classification of all irreducible tempered characters
of the Hecke algebra. The explicit Plancherel isomorphism
can be reconstructed by (ii) and \cite[Theorem 4.43]{Opd1}.

Let us describe the methods used in this paper.
The new tool in this study of these questions for abstract affine Hecke
algebras is derived
from the presence of a space of continuous parameters with respect to which the
harmonic analysis naturally deforms. Observe that
this aspect is missing in the traditional context of the harmonic analysis on
reductive groups.
The main message of this paper is that parameter deformation is a powerful tool
for solving the questions (i) and (ii), especially (but not exclusively)
for non-simply laced root data. There are in fact two
other pillars on which our method rests, based on results from
\cite{Opd1} and \cite{OpdSol}. We will now give a more detailed account
of these matters.

An affine Hecke algebra $\mc{H}=\mc{H}(\mc{R},q)$
is defined in terms of a based root datum
\[
\mc{R}= (X, R_0 ,Y ,R_0^\vee ,F_0 )
\]
and a parameter function $q\in\mc{Q}=\mc{Q}(\mc{R})$. By this we mean
that $q$ is a (positive) function on the set
$S$ of simple affine reflections in the affine Weyl group $\mathbb{Z} R_0 \rtimes W_0$,
such that $q(s) = q(s')$ whenever $s$ and $s'$ are conjugate in the extended
Weyl group $W = X \rtimes W_0$. The deformation method is based on regarding
the affine Hecke algebras $\mc{H}(\mc{R},q)$
with fixed $\mc{R}$ as a continuous field of algebras, depending on the
parameter $q$. This enables us to transfer properties that hold for
$q \equiv 1$ or for \emph{generic} $q$ to \emph{arbitrary} positive parameters.

We will prove that every irreducible discrete series character
$\delta_{0}$ of $\mc{H}(\mc{R},q_0)$ is the evaluation at $q_0$ of a
\emph{unique} maximal continuous family $q\to\delta_q$ of discrete series
characters of $\mc{H}(\mc{R},q)$ defined in a suitable open neighborhood of $q_0$.
The continuity of the
family means that the corresponding family of primitive central idempotents
$q\to e_\delta(q)\in\mc{S}$ (the Schwartz completion of $\mc{H}(\mc{R},q)$, a Fr\'echet
algebra which is independent of $q$ as a Fr\'echet space) is continuous in $q$
with respect to the Fr\'echet topology of $\mc{S}$.
The maximal domain of definition of the family $q \to \delta_q$ is described in
terms of the zero locus of an explicit rational function on $\mc{Q}$.
This reduces the classification of the discrete series of
$\mc{H}(\mc{R},q)$ for \emph{arbitrary} (possibly special) positive parameters to
that for \emph{generic} positive parameters, a problem that is considerably
easier than the general case.

Let us take the discussion one step further to see how this idea
leads to a practical strategy for the classification of the discrete series characters.
For this it is crucial to understand how the ``central characters'' behave under
the unique continuous deformation $q\to\delta_q$ of an irreducible discrete series
character $\delta_0$. Since it is known that the set of discrete
series can be nonempty only if $R_0$ spans $X \otimes_{\mathbb{Z}}\mathbb{Q}$, we
assume
this throughout the paper. To enable the use of analytic techniques we need an
involution * and a positive trace $\tau$ on our affine Hecke algebras
$\mc{H}(\mc{R},q)$. A natural choice is available, provided that all parameters are
positive (another assumption we make throughout this paper).
Then $\mc{H}(\mc{R},q)$ is in fact a Hilbert
algebra with tracial state $\tau$. The spectral decomposition of $\tau$ defines
a positive measure $\mu_{Pl}$ (called the Plancherel measure) on the set of
irreducible representations of $\mc{H}(\mc{R},q)$, cf. \cite{Opd1, DeOp1}.
More or less by definition an irreducible representation $\pi$ belongs to the
discrete series if $\mu_{Pl}(\{ \pi \})>0$. It is known that this condition
is equivalent
to the statement that $\pi$ is an irreducible projective representation of
$\mc{S}(\mc{R},q)$, the Schwartz completion of $\mc{H}(\mc{R},q)$. In
particular $\pi$ is an irreducible discrete series representation iff $\pi$ is
afforded by a primitive central idempotent $e_\pi\in\mc{S}(\mc{R},q)$ of finite
rank. Thus the definition of continuity of a family of irreducible
characters in the preceding
paragraph makes sense for discrete series characters only.
We denote the finite set of irreducible discrete series characters
of $\mc{H}(\mc{R},q)$ by $\Delta (\mc{R},q)$.

A cornerstone in the spectral theory of the affine Hecke algebra is formed by
Bernstein's classical construction of a large commutative subalgebra
$\mc{A}\subset\mc{H}(\mc{R},q)$
isomorphic to the group algebra $\mathbb{C}[X]$. It follows from
this construction that the center of $\mc{H}(\mc{R},q)$
equals $\mc{A}^{W_0} \cong \mathbb{C}[X]^{W_0}$. Therefore we have a central
character map
\begin{equation}
cc_q : \operatorname{Irr}(\mc{H}(\mc{R},q))\to W_0\backslash T
\end{equation}
(where $T$ is complex torus $\textup{Hom}(X,\mathbb{C}^\times)$) which
is an invariant in the sense that this map is constant on equivalence classes
of irreducible representations.

It was shown by ``residue calculus'' \cite[Lemma 3.31]{Opd1} that a given orbit
$W_0t\in W_0\backslash T$ is the the central character of a discrete series
representation iff $W_0t$ is a $W_0$-orbit of
so-called \emph{residual points of $T$}. These residual points
are defined in terms of the poles and zeros of an explicit rational differential
form on $T$ (see Definition \ref{def:respt}), and they have been classified
completely. They depend on a pair $(\mc{R},q)$ consisting of a (semisimple) root
datum $\mc{R}$ and a parameter $q\in\mc{Q}$.
In fact, given a semisimple root datum $\mc{R}$ there exist finitely many
$\mc{Q}$-valued points $r$ of $T$, called \emph{generic residual points}, such that
on a Zariski-open set of the parameter space $\mc{Q}$
the evaluation $r(q)\in T$ is a residual point for $(\mc{R},q)$.
Moreover, for every $q_0\in\mc{Q}(\mc{R})$ and every residual point $r_0$ of
$(\mc{R},q_0)$ there exists at least one generic residual point $r$
such that $r_0=r(q_0)$.

For fixed $q_0\in\mc{Q}$ these techniques do in general not shine any further light
on the cardinality of $\Delta(\mc{R},q_0)$. The problem is a well known difficulty
in representation theory: the central character invariant $cc_{q_0}(\delta_0)$
is not strong enough to separate the equivalence classes of irreducible
(discrete series) representations.
But this is precisely the point where the deformation method is helpful.
The idea is that at generic parameters the
separation of the irreducible discrete series characters by their central character is much
better (almost perfect in fact, see below) than for special parameters. Therefore
we can improve the quality of the central character invariant for
$\delta_0\in\Delta(\mc{R},q_0)$ by considering the family of central
characters $q\to cc_q (\delta_q)$ of the unique continuous deformation $q\to\delta_q$
of $\delta_0$ as described above.
It turns out that this family of central characters is in fact
a $W_0$-orbit $W_0r$ of generic residual points.
We  call this the \emph{generic central character} $gcc(\delta_0)=W_0r$
of $\delta_0$.
\\[1.5mm]

Our proof of this fact requires various techniques.
First of all the existence and uniqueness of the germ of continuous deformations
of a discrete series character depends in an essential way on the continuous
field of pre-$C^*$-algebras
$\mc{S}(\mc R ,q)$, where $q$ runs through $\mc{Q}$ and $\mc{S}(\mc{R},q)$
is the Schwartz completion of $\mc{H}(\mc{R},q)$ (see \cite{DeOp1}).
Pick $\delta_0 \in \Delta(\mc{R},q_0)$ with central character $cc_{q_0}(\delta_0)=
W_0 r_0 \in W_0 \backslash T$. With analytic techniques we prove that there exists
an open neighborhood $U \times V \subset Q \times W_0 \backslash T$ of $(q_0,W_0 r_0)$
such that (see Lemma \ref{lem:est}, Theorem \ref{thm:def} and Theorem \ref{thm:cont}):
\begin{itemize}
\item there exists a unique continuous family
$U \ni q \to \delta_q \in \Delta (\mc R ,q)$ with $\delta_{q_0} = \delta_0$,
\item the cardinality of $\{\delta \in \Delta(\mc{R},q)\mid cc_q(\delta)\in V\}$
is independent of $q \in U$.
\end{itemize}
Next we consider the formal degree $\mu_{Pl} (\{ \delta_q \})$ of $\delta_q
\in \Delta (\mc R ,q)$. In \cite{OpdSol} we proved an ``index formula'' for the formal
degree, expressing $\mu_{Pl} (\{ \delta_q \})$ as alternating sum of formal degrees
of characters of certain finite dimensional involutive subalgebras of $\mc H(\mc{R},q)$.
It follows that $\mu_{Pl}(\{ \delta_q \})$
is a rational function of $q \in U$, with rational coefficients.
On the other hand using the residue calculus \cite{Opd1} we derive an explicit
factorization
\begin{equation}\label{eq:ii}
\mu_{Pl} (\{ \delta_q \}) = d_\delta m_{W_0r} (q) \qquad q \in U \,,
\end{equation}
with $d_\delta \in \mathbb{Q}^\times$ independent of $q$ and
$m_{W_0r} (q)$ depending only on $q$ and on the central character
$cc_q (\delta_q) = W_0 r(q)$ (for the definition of $m$ see (\ref{eq:m})).
Using the classification of generic residual
points this enables us to prove that $q \to cc_q (\delta_q )$ is not only continuous
but in fact (in a neighborhood of $q_0$) of the form $q\to W_0r(q)$ for a unique
orbit of generic residual points $gcc(\delta_0)=W_0r$, the generic
central character of $\delta_0$ alluded to above.
We can now write (\ref{eq:ii}) in the form (see Theorem \ref{thm:main2}):
\begin{equation}\label{eq:iibis}
\mu_{Pl} (\{ \delta_q \}) = d_\delta m_{gcc(\delta)}(q) \qquad q \in U \,,
\end{equation}
where $m_{gcc(\delta)}$ is an explicit rational function with rational
coefficients on $\mc{Q}$, which is regular on $\mc{Q}$ and whose zero locus
is a finite union of hyperplanes in $\mc{Q}$ (viewed as a vector space).

The incidence space $\mathcal{O}(\mc{R})$ consisting of pairs $(W_0r,q)$
with $W_0r$ an orbit of generic residual points and $q\in \mc{Q}$ such that
$r(q)$ is a residual point for $(\mc{R},q)$ can alternatively be described as
$\mathcal{O}(\mc{R})=\{(W_0r,q)\mid m_{W_0r}(q)\not=0\}$.
Thus $\mathcal{O}(\mc{R})$ is a disjoint union of copies of certain
convex open cones in $\mc{Q}$. The above deformation arguments culminate in
Theorem \ref{thm:main1} stating that the map
\begin{align}
GCC:\coprod_{q\in\mc{Q}(\mc{R})}\Delta(\mc{R},q)&\to\mathcal{O(\mc{R})}\\
\nonumber\Delta(\mc{R},q)\ni\delta&\to(gcc(\delta),q)
\end{align}
gives $\Delta(\mc{R}):=\coprod_{q\in\mc{Q}(\mc{R})}\Delta(\mc{R},q)$ the
structure of a locally constant sheaf of finite sets on $\mathcal{O}(\mc{R})$.
Since every component of $\mathcal{O}(\mc{R})$ is contractible this result
reduces the classification of the set $\Delta(\mc{R})$ to the computation of the
multiplicities of the various components of $\mathcal{O}(\mc{R})$
(i.e. the cardinalities of the fibers of the map $GCC$).

One more ingredient is of great technical importance.
Lusztig \cite{Lus2} proved fundamental reduction theorems
which reduce the classification of irreducible representations of
affine Hecke algebras effectively to the the classification of
irreducible representations of degenerate affine Hecke algebras
(extended by a group acting through diagram automorphisms, in general).
In this paper we make frequent use of a version of these results
adapted to suit the situation of arbitrary positive parameters
(see Theorem \ref{thm:red1} and Theorem \ref{thm:red2}). These
reductions respect the notions of temperedness and discreteness
of a representation. Using this type of results it suffices
to compute the multiplicities of the \emph{positive} components
of $\mathcal{O}(\mc{R})$ or equivalently, to compute the
multiplicities of the corresponding components in the parameter
space of a degenerate affine Hecke algebra
(possibly extended by a group acting through of diagram
automorphisms).

The results are as follows.
If $R_0$ is simply laced then the generic central character map itself
does not contain new information compared to the ordinary central character.
However with a small enhancement the generic central character map gives
a complete invariant for the discrete series of $D_n$ as well, using that
the degenerate affine Hecke algebra of type $D_n$ twisted by a diagram
involution is a specialization of the degenerate affine Hecke algebra
of type $B_n$. With this enhancement understood we can state that
the generic central character is a complete invariant for the irreducible discrete
series characters of a degenerate affine Hecke algebra associated with a
simple root system $R_0$, except when $R_0$ is of type $E_6, E_7, E_8$ or
$F_4$. In the $F_4$-case with both parameters unequal to zero there exist
precisely two irreducible discrete series characters which have the same
generic central character.

Our solution to problem (i) is listed in Sections \ref{sec:dsH} and
\ref{sec:ds}. This covers essentially all cases except type $E_n$ ($n=6,\,7,\,8$)
(in which cases we rely on \cite{KL} for the classification).
In this classification the irreducible discrete series characters are parametrized
in terms of their generic central character.
The solution to problem (ii) is given by the product formula (\ref{eq:iibis})
(see Theorem \ref{thm:main2}) which expresses the formal degree of $\delta_q$ explicitly
as a rational function with rational coefficients on the maximal domain
$U_\delta\subset\mc{Q}$ to which $\delta_q$ extends as a continuous family of
irreducible discrete series characters ($U_\delta$ is the interior of an
explicitly known convex polyhedral cone).
At present we do not know how to compute the rational numbers $d_\delta$ for
each continuous family so our solution is incomplete at this point.
\\[1.5mm]

Let us compare our results with the existing literature.
An important special case arises when the parameter function $q$ is constant on $S$,
which happens for example when the root system $R_0$ is irreducible and simply laced.
In this case all irreducible representations of $\mc H (\mc R ,q)$ (not only the discrete
series) have been classified by Kazhdan and Lusztig \cite{KL}. This classification
is essentially independent of $q \in \mathbb{C}^\times$, except for a few "bad" roots of
unity. This work of Kazhdan and Lusztig is of course much more than just a classification
of irreducible characters, it actually gives a geometric construction of standard
modules of the Hecke algebra for which one can deduce detailed information on the
internal structure in geometric terms (e.g. Green functions).
The Kazhdan-Lusztig parametrization yields the classification of
the tempered and the discrete series characters too.

Next Lusztig \cite{LuCL} has classified the irreducibles of
``geometric'' graded affine Hecke algebras (with certain unequal
parameters) which arise from a cuspidal local system on a
unipotent orbit of a Levi subgroup of a given almost simple simply
connected complex group ${}^LG$. In \cite{Lus3} these results were
refined to include a classification of tempered and discrete
series irreducible modules of the geometric graded Hecke algebras.
In \cite{Lu4} it is shown that such graded affine Hecke algebras
arise as completions of ``geometric'' affine Hecke algebras (with
certain unequal parameters) formally associated to the above
geometric data. On the other hand, let $k$ be a $p$-adic field and
let $G$ be the group of $k$-rational points of a split adjoint
simple group $\mathbf{G}$ over $k$ such that ${}^LG$ is the
connected component of its Langlands dual group. In \cite{Lu4} the
explicit list of unipotent ``arithmetic'' affine Hecke algebras is
given, i.e. affine Hecke algebras occurring as the Hecke algebra
of a type (in the sense of \cite{BK1}) for a $G$-inertial
equivalence class of a unipotent supercuspidal pair $(L,\sigma)$
(also see \cite{Mo1}, \cite{Mo2}). Remarkably, a case-by-case
analysis in \cite{Lu4} shows that the geometric affine Hecke
algebras associated with ${}^LG$ precisely match the unipotent
arithmetic affine Hecke algebras arising from $G$. More generally
such results hold if $\mathbf{G}$ is only assumed to be split over an
unramified extension of $k$ \cite{Lu4}.

The geometric parameters in terms of which Lusztig \cite{LuCL}, \cite{Lus3} classifies
the irreducible (tempered, discrete series) modules over geometric graded affine Hecke
algebras are rather complicated. Our present direct approach, based on
deformations in the harmonic analysis of ``arithmetic'' affine Hecke algebras,
gives different and in some sense complementary information (e.g. formal degrees).
We refer to \cite{Blo} for examples of affine Hecke algebras arising as Hecke algebras
of more general types. We refer to \cite{Lu5} for results and conjectures on the
theory of Kazhdan-Lusztig bases of abstract Hecke algebras with unequal parameters.

The techniques in this paper do not give an explicit construction of the discrete
series representations. In this direction it is interesting to mention Syu Kato's
geometric construction \cite{Kat2} of algebraic families of representations of
$\mc H (C_n^{(1)},q)$. One would like to understand how Kato's geometric model
relates to our continuous families of discrete series characters,
which are constructed by analytic methods.
\section{Preliminaries and notations}

\subsection{Affine Hecke algebras}

\subsubsection{Root data and affine Weyl groups}
\label{subsub:root}

Suppose we are given lattices $X,Y$ in perfect duality
$\langle\cdot,\cdot\rangle:X\times Y\to \mathbb{Z}$, and
finite subsets $R_0\subset X$ and $R_0^\vee\subset Y$ with a
given a bijection $\vee: R_0\to R_0^\vee$. Define endomorphisms
$r_{\alpha^\vee}:X\to X$ by $r_{\alpha^\vee}(x)=x-x(\alpha^\vee)\alpha$ and
$r_{\alpha}:Y\to Y$ by $r_{\alpha}(y)=y-\alpha(y)\alpha^\vee$. Then
$(R_0,X,R_0^\vee,Y)$ is called a root datum if
\index{X@$X$, lattice}
\index{Y@$Y$, dual lattice of $X$}
\index{R0@$R_0$, reduced integral root system}
\index{R0@$R_0^\vee$, dual root system}
\index{R2@$\mc R$, based root datum}
\begin{enumerate}
\item for all $\alpha\in R_0$ we have $\alpha(\alpha^\vee)=2$.
\item for all $\alpha\in R_0$ we have $r_{\alpha^\vee}(R_0)\subset R_0$
and $r_{\alpha}(R_0^\vee)\subset R_0^\vee$.
\end{enumerate}
As is well known, it follows that $R_0$ is a root system in
the vector space spanned by the elements of $R_0$. A \emph{based root datum}
$\mc{R}=(X,R_0,Y,R_0^\vee,F_0)$ consists of a root datum with a basis
$F_0\subset R_0$ of simple roots.
\index{F0@$F_0$,basis of $R_0$}

The (extended) affine Weyl group of $\mc{R}$ is the group
$W=W_0\ltimes X$ (where $W_0=W(R_0)$ is the Weyl group of $R_0$); it
\index{W@$W$, extended affine Weyl group of $\mc R$}
\index{W0@$W_0$, Weyl group of $R_0$}
naturally acts on $X$. We identify $Y\times \mathbb{Z}$ with the set
of affine linear, $\mathbb{Z}$-valued functions on $X$ (in this context we usually
denote an affine root $a=(\alpha^\vee,n)$ additively as $a=\alpha^\vee+n$).
Then the affine Weyl group $W$ acts linearly on the set $Y\times\mathbb{Z}$
via the action $wf(x):=f(w^{-1}x)$.
The affine root system $R$ associated to $\mc{R}$ is the $W$-invariant set
$R:=R_0^\vee\times\mathbb{Z}\subset Y\times \mathbb{Z}$.
\index{R@$R$, affine root system}
The basis $F_0$ of simple roots induces a decomposition
$R=R_+\cup R_-$ with $R_+:=R_{0,+}^\vee\times\{0\}\cup
R_0^\vee \times\mathbb{N}$ and $R_-=-R_+$. It is easy to see that
$R_+$ has a basis of affine roots $F$ \index{F@$F$, basis of $R$}
consisting of the set $F_0^\vee\times \{0\}$
supplemented by the set of affine roots of the form
$a=(\alpha^\vee,1)$ where $\alpha^\vee\in R_0^\vee$ runs over the set
of minimal coroots. The set $F$ is called the set of affine simple
roots. Every $W$-orbit $Wa\subset R$ with $a\in R$ meets the
set $F$ of affine simple roots. We denote by $\tilde{F}$ the set
of intersections of the $W$-orbits in $R$ with $F$.

To an affine root $a=(\alpha^\vee,n)$ we associate an affine
reflection $r_a:X\to X$ \index{rz@$r_a$, affine reflection} by $r_a(x)=x-a(x)\alpha$.
We have $r_a\in W$ and $wr_aw^{-1}=r_{wa}$. Hence the subgroup
$W^a\subset W$ \index{Wa@$W^a$, affine Weyl group of $R$}
 generated by the affine reflections $r_a$ with $a\in R$ is normal.
The normal subgroup $W^a$ has a Coxeter presentation $(W^a,S)$ with
respect to the set of Coxeter generators $S=\{r_a\mid a\in F\}$.
\index{S@$S$, simple affine reflections in $W_a$}
We call $S$ the set of affine simple reflections and we write
$S_0 = S \cap W_0$. \index{S0@$S_0$ simple reflections in $W_0$}
We call two elements
$s,t\in S$ equivalent if they are conjugate to each other inside $W$.
We put $\tilde{S}$ for the set of equivalence classes in $S$.
The set $\tilde{S}$ is in natural bijection with the set $\tilde{F}$.
\index{St@$\tilde{S}$, equivalence classes in $S$}

We define a length function $l:W\to\mathbb{Z}_+$ by
\index{lz@$l$, length function of $W$}
$l(w):=|w^{-1}(R_-)\cap R_+|$. The set $\Omega:=\{w\in W\mid l(w)=0\}$
\index{1z@$\Omega$, length zero elements in $W$}
is a subgroup of $W$. Since $W^a$ acts simply transitively on
the set of positive systems of affine roots it is clear that
$W=W^a\rtimes \Omega$. Notice that if we put
$X^+=\{x\in X\mid x(\alpha^\vee)\geq 0\ \forall\alpha\in F_0\}$ and
\index{X+@$X^+$, positive cone in $X$}
$X^-=-X^+$ then the sublattice $Z=X^+\cap X^-\subset X$ is the center of $W$.
It is clear that $Z$ acts trivially on $R$ and in particular, we have
$Z\subset\Omega$.
\index{Z@$Z$, length zero translations in $W$}
\index{Q@$Q(R_0 )$, root lattice of $R_0$}
We have $\Omega \cong W/W^a \cong X/Q(R_0)$ where $Q(R_0)$ denotes
the root lattice of the root system $R_0$. It follows easily that
$\Omega/Z$ is finite. We call $\mc{R}$ semisimple if $Z=0$.
By the above $\mc{R}$ is semisimple iff $\Omega$ is finite.

\subsubsection{The generic affine Hecke algebra and its specializations}
\label{subsub:defn}

We introduce invertible, commuting indeterminates $v([s])$ where
$[s]\in \tilde{S}$. Let $\Lambda=\mathbb{C}[v([s])^{\pm 1}:[s]\in \tilde{S}]$.
If $s\in S$ then we define $v(s):=v([s])$. The following definition is in fact
a theorem (this result goes back to Tits):
\index{vs@$v(s)\in\Lambda$, indeterminate of $s \in S$}
\index{1l@$\Lambda \,, \mh C$-algebra generated by $v(s)^{\pm 1}$}
\begin{defn} \label{def:Heckealg}
There exists a unique associative, unital $\Lambda$-algebra
$\mc{H}_\Lambda(\mc{R})$ which has a $\Lambda$-basis $\{N_w\}_{w\in W}$
parametrized by $w\in W$, satisfying the  relations
\index{HLa@$\mc{H}_\Lambda$, generic affine Hecke algebra}
\begin{enumerate}
\item $N_wN_{w^\prime}=N_{ww^\prime}$ for all $w,w^\prime\in W$
such that $l(ww^\prime)=l(w)+l(w^\prime)$.
\item $(N_s-v(s))(N_s+v(s)^{-1})=0$ for all $s\in S$.
\end{enumerate}
The algebra $\mc{H}_\Lambda = \mc{H}_\Lambda(\mc{R})$
is called the generic affine Hecke algebra with root datum $\mc{R}$.
\end{defn}
\index{Nw@$N_w$, basis elements of Hecke algebra}
We put $\mc{Q}_c=\mc{Q}(\mc{R})_c$ for the complex torus
\index{Q@$\mc Q_c = \text{Hom} (\Lambda, \mh C)$}
\index{HRq@$\mc H = \mc{H}(\mc{R},q)$, affine Hecke algebra}
of homomorphisms $\Lambda\to \mathbb{C}$.
We equip the torus $\mc{Q}_c$ with the analytic topology.
Given a homomorphism $q\in{\mc{Q}}_c$ we define a specialization
\index{q1@$q\in\mc{Q}_c$, parameter function $S \to \mh C^\times$}
\footnote {This is not compatible with the conventions in
\cite{Opd1}, \cite{Opd2}, \cite{Opd3}, \cite{OpdSol}! The parameter $q\in\mc{Q}$
in the present paper would be called $q^{1/2}$ in these earlier papers.}
$\mc{H}(\mc{R},q)$ of the generic algebra as follows (with $\mathbb{C}_q$ the
$\Lambda$-module defined by $q$):
\begin{equation}
\mc{H}(\mc{R},q):=\mc{H}_\Lambda(\mc{R})\otimes_\Lambda\mathbb{C}_q
\end{equation}
\index{1v@$\phi_s$, automorphism of $\mc H_\Lambda$}
Observe that the automorphism $\phi_s:\Lambda\to\Lambda$ defined by
$\phi_s(v(t))=v(t)$ if $t\not\sim_W s$ and $\phi_s(v(s))=-v(s)$ extends to an
automorphism of $\mc{H}_\Lambda$ by putting $\phi_s(N_t)=N_t$ if $t\not\sim_W s$
and $\phi_s (N_s)=-N_s$. Similarly we have automorphims
$\psi_s:\mc{H}_\Lambda\to\mc{H}_\Lambda$ given by $\psi_s(v(s))=v(s)^{-1}$,
\index{1y@$\psi_s$, automorphism of $\mc H_\Lambda$}
$\psi_s(v(t))=v(t)$ if $t\not\sim_W s$, $\psi_s(N_s)=-N_s$ and $\psi_s(N_t)=N_t$
if $t\not\sim_W s$. These automorphisms mutually commute and are involutive.
Observe that $\phi_s\psi_s$ respects the distinguished basis $N_w$ of
$\mc{H}_\Lambda$, and the automorphisms $\phi_s$ and $\psi_s$ individually
respect the distinguished basis up to signs.

\index{Q@$\mc Q$, positive points of $\mc Q_c$}
We write $\mc{Q}$ for the set of positive points of $\mc{Q}_c$,
i.e. points $q\in\mc{Q}_c$ such that $q(v(s))>0$ for all $s\in S$.
Then $\mc{Q}\subset\mc{Q}_c$ is a real vector group.

There are alternative ways to specify points of $\mc{Q}$ which
play a role in the spectral theory of affine Hecke algebras
(in particular in relation to the Macdonald $c$-function \cite{Ma2}).
In order to explain this we introduce the possibly nonreduced
root system $R_{nr}\subset X$ associated to $\mc{R}$ as follows:
\index{R11@$R_{nr}$, nonreduced root system}
\begin{equation}
R_{nr}=R_0\cup \{2\alpha\mid \alpha^\vee\in 2Y\cap R_0^\vee\}
\end{equation}
\index{R1@$R_1$, nonmultipliable roots in $R_{nr}$}
We let $R_1=\{\alpha\in R_{nr}\mid 2\alpha\not\in R_{nr}\}$ be the set of
nonmultipliable roots in $R_{nr}$. Then
$R_1\subset X$ is also a reduced root system, and $W_0=W(R_0)=W(R_1)$.

We define various functions with values in $\Lambda$.
First we define a $W$-invariant function $R\ni a\to v_a\in\Lambda$ by
requiring that \index{va@$v_a\in\Lambda$, value of $v$ at $a\in R$}
\begin{equation}\label{eq:shift}
v_{a+1}=v(s_a)
\end{equation}
for all simple affine roots $a\in F$.
Notice that all generators $v(s)$ of $\Lambda$ are in the image of this
function. Next we define a $W_0$-invariant function
$R_{nr}^\vee\ni\alpha^\vee\to v_{\alpha^\vee}\in\Lambda$ as
follows. If $\alpha\in R_0$ we view $\alpha^\vee$ as an element of $R$,
so that $v_{\alpha^\vee}$ has already been defined. If $\alpha=2\beta$
with $\beta\in R_0$ then we define:
\index{vw@$v(w)\in\Lambda$, value of $v$ at $w \in W$}
\index{vac@$v_{\alpha^\vee}\in\Lambda$, value of $v$ at $\alpha^\vee\in R_{nr}^\vee$}
\begin{equation}
v_{\alpha^\vee}=v_{\beta^\vee/2}:=v_{\beta^\vee+1}/v_{\beta^\vee}
\end{equation}
Finally there exists a unique length-multiplicative function $W\ni w\to v(w)\in\Lambda$
such that its restriction to $S$ yields the original assignment $S\ni s\to v(s)\in\Lambda$
of generators of $\Lambda$ to the $W$-orbits of simple reflections of $W$, and
$v(\omega)=1$ for all $\omega\in \Omega$. Here the notion length-multiplicative
refers to the property $v(w_1w_2)=v(w_1)v(w_2)$ if $l(w_1w_2)=l(w_1)+l(w_2)$.
We remark that with these notations we have
\begin{equation}
v(w)=\prod_{\alpha\in R_{nr,+}\cap w^{-1}R_{nr,-}}v_{\alpha^\vee}
\end{equation}
for all $w\in W_0$.

A point $q\in\mc{Q}$ determines a unique $W$-invariant function on $R$ with
values in $\mathbb{R}_+$ by defining $q_a:=q(v_a)$. Conversely such a
positive $W$-invariant function on $R$ determines a point $q\in \mc{Q}$.
Likewise we define positive real numbers
\index{qa@$q_a$, parameter of affine root $a$}
\begin{equation}
q_{\alpha^\vee}:=q(v_{\alpha^\vee})
\end{equation}
for $\alpha\in R_{nr}$ and
\index{qw@$q(w)$, parameter of $w \in W$}
\begin{equation}
q(w):=q(v(w))
\end{equation}
for $w\in W$. In this way the points $q\in\mc{Q}$ are in natural bijection with
the set of $W_0$-invariant positive functions on $R_{nr}^\vee$ and also with
the set of positive length-multiplicative functions on $W$ which restrict to
$1$ on $\Omega$.

Recall that if the finite root system $R_1$ is irreducible, it can be extended in a
unique way to an affine root system, which is called $R_1^{(1)}$.
\index{R@$R_1^{(1)}$, affine extension of $R_1$}

\begin{defn} If $\mc{R}$ is simple and $X=P(R_1)$ (the weight lattice
of $R_1$) we call $\mc{H}(\mc{R},q)$ of type $R_1^{(1)}$. This includes the
simple $3$-parameter case $C_n^{(1)}$ with $R_0=B_n$ and $X=Q(R_0)$.
\end{defn}
\index{PR1@$P(R_1)$, weight lattice of root system $R_1$}

\subsubsection{The Bernstein presentation and the center}

The length function $l:W\to\mathbb{Z}_{\geq 0}$ restricts to
a homomorphism of monoids on $X^+$. Hence the map
$X^+\to\mc{H}_\Lambda^\times$ defined by $x\to N_x$
is an homomorphism of monoids too. It has a unique
extension to a group homomorphism $\theta:X\to\mc{H}_\Lambda^\times$
which we denote by $x\to \theta_x$. We denote by
\index{1h@$\theta_x$, basis element of $\mc A_\Lambda$}
\index{ALa@$\mc A_\Lambda$, commutative subalgebra of $\mc H_\lambda$}
\index{A@$\mc A$, commutative subalgebra of $\mc H (\mc R ,q)$}
$\mc{A}_\Lambda\subset\mc{H}_\Lambda$ the commutative subalgebra of
$\mc{H}_\Lambda$ generated by the elements $\theta_x$ with $x\in X$.
Similarly we have a commutative subalgebra $\mc A \subset \mc H (\mc R,q)$.
\index{HLa@$\mc{H}_{\Lambda,0}$, generic Hecke algebra of $R_0$}
Let $\mc{H}_{\Lambda,0}=\mc{H}_\Lambda(W_0,S_0)$ be the Hecke
subalgebra (of finite rank over the algebra $\Lambda$) corresponding to
the Coxeter group $(W_0,S_0)$. We have the following important
result due to Bernstein-Zelevinski (unpublished) and Lusztig (\cite{Lus2}):
\begin{thm}
The multiplication map defines an isomorphism of
$\mc{A}_\Lambda-\mc{H}_{\Lambda,0}$-modules
$\mc{A}_\Lambda\otimes\mc{H}_{\Lambda,0}\to\mc{H}_\Lambda$ and an isomorphism
of $\mc{H}_{\Lambda,0}-\mc{A}_\Lambda$-modules
$\mc{H}_{\Lambda,0}\otimes\mc{A}_\Lambda\to\mc{H}_\Lambda$.
The algebra structure on $\mc{H}_\Lambda$ is determined
by the cross relation (with $x\in X$, $\alpha\in F_0$,
$s=r_{\alpha^\vee}$, and $s^\prime\in S$ is a simple reflection
such that $s^\prime\sim_W r_{\alpha^\vee+1}$):
\begin{equation}\label{eq:ber}
\theta_x N_s-N_s \theta_{s(x)}=\left((v(s)-v(s)^{-1})+
(v(s^\prime)-v(s^\prime)^{-1})\theta_{-\alpha}\right)
\frac{\theta_x-\theta_{s(x)}}{1-\theta_{-2\alpha}}
\end{equation}
(Note that if $s^\prime\not\sim_W s$ then $\alpha^\vee\in 2R_0^\vee$,
which implies  $x-s(x)\in 2\mathbb{Z}\alpha$ for all $x\in X$. This guarantees
that the right hand side of (\ref{eq:ber}) is always an element of $\mc{A}_\Lambda$).
\end{thm}
\begin{cor} The center $\mc{Z}_\Lambda$ of $\mc{H}_\Lambda$ is the algebra
$\mc{Z}_\Lambda=\mc{A}^{W_0}_\Lambda$. For any $q\in\mc{Q}_c$ the center of
$\mc{H}(\mc{R},q)$ is equal to the subalgebra
$\mc{Z}=\mc{A}^{W_0}\subset\mc{H}(\mc{R},q)$.
\end{cor}
\index{ZLa@$\mc{Z}_\Lambda$, center of $\mc{H}_\Lambda$}
\index{Z@$\mc Z$, center of $\mc{H}(\mc R ,q)$}
In particular $\mc{H}_\Lambda$ is a finite type algebra over its center
$\mc{Z}_\Lambda$, and similarly $\mc{H}(\mc{R},q)$ is a finite type algebra
over its center $\mc{Z}$. The simple modules over these algebras
are finite dimensional complex vector spaces. The primitive ideal spectrum
$\widehat{\mc{H}}_\Lambda$ is a topological space which comes equipped with
a finite continuous and closed map
\index{A@$\widehat A$, primitive ideal spectrum of an algebra $A$}
\index{ccL@$cc_\Lambda$, central character map for $\mc H_\Lambda$}
\begin{equation}\label{eq:ccL}
cc_\Lambda : \widehat{\mc{H}}_\Lambda\to\widehat{\mc{Z}}_\Lambda
=W_0\backslash T\times \mc{Q}_c
\end{equation}
to the complex affine variety associated with the unital complex
commutative algebra $\mc{Z}_\Lambda$. The map $cc_\Lambda$ is
called the central character map. Similarly, we have
central character maps \index{ccq@$cc_q$, central character map for $\mc H(\mc R,q)$}
\begin{equation}
cc_q : \widehat{\mc{H}(\mc{R},q)}\to\widehat{\mc{Z}}
\end{equation}
for all $q\in\mc{Q}_c$.

\index{T@$T$, character torus of $X$}
We put $T =\textup{Hom}(X,\mathbb{C}^\times)$, the complex torus
of characters of the lattice $X$ equipped with the Zariski topology.
This torus has a natural $W_0$-action. We have
$\widehat{\mc{Z}}=W_0\backslash T$ (the categorical quotient).

\subsubsection{Two reduction theorems}

The study of the simple modules over $\mc{H}(\mc{R},q)$
is simplified by two reduction theorems which are much in the
spirit of Lusztig's reduction theorems in \cite{Lus2}.
The first of these theorems reduces to the case of simple modules
whose central character is a $W_0$-orbit of characters of $X$ which
are positive on the sublattice of $X$ spanned by $R_1$ (see the explanation
below). The second theorem reduces the study of simple modules
of $\mc{H}(\mc{R},q)$ with a positive central character in the above sense
to the study of simple modules of an associated degenerate affine Hecke
algebra with real central character. These results will be useful
for our study of the discrete series characters.

\index{Tv@$T_v$, positive part of $T$} \index{Tu@$T_u$, unitary part  of $T$}
First of all a word about terminology.
The complex torus $T$ has a polar decomposition $T=T_vT_u$ with
$T_{v}=\textup{Hom}(X,\mathbb{R}_{>0})$ and $T_u=\textup{Hom}(X,S^1)$.
The polar decomposition is the exponentiated form of the
decomposition of the tangent space $V=\textup{Hom}(X,\mathbb{C})$
of $T$ at $t=e$ as a direct sum $V=V_r\oplus iV_r$ of real
subspaces where $V_r=\textup{Hom}(X,\mathbb{R})$.
The vector group $T_{v}$ is called the group of positive characters
and the compact torus $T_u$ is called the group of unitary characters.
This polar decomposition is compatible with the action of $W_0$ on $T$.
We call the $W_0$-orbits of points in $T_{v}$ ``positive'' and the
$W_0$-orbits of points in $T_u$ ``unitary''. In this sense can we speak
of the subcategory of finite dimensional $\mc{H}(\mc{R},q)$-modules with
positive central
character\footnote{In several prior publications
\cite{HO}, \cite{HOH}, \cite{Opd1}, \cite{Opd2}, \cite{Opd3} the central
characters in $W_0\backslash T_{v}$ were referred to as ``real central characters'',
where ``real'' should be understood as ``infinitesimally real''.
In the present paper however we change the terminology and speak of
``positive central characters'' instead.} which is a subcategory that plays
an important special role.
\begin{defn}\label{defn:iso} Let $\mc{R}$ be a root datum and
let $q\in\mc{Q}=\mc{Q}(\mc{R})$. For $s\in T_u$ we define
$R_{s,0}=\{\alpha\in R_0\mid r_\alpha(s)=s\}$. Let $R_{s,1}\subset R_1$
be the set of nonmultipliable roots corresponding to $R_{s,0}$.
One checks that
\index{R111@$R_{s,0}, R_{s,1}, R_{s,nr}$, root systems associated to $s \in T$}
\index{Fs1@$F_{s,1}$, basis of $R_{s,1}$}
\index{R2@$\mc R_s$, root datum associated to $s \in T$}
\index{Ws@$W_s$, isotropy group in $W_0$ of $s \in T$}
\begin{equation}
R_{s,1}=\{\beta\in R_1\mid \beta(s)=1\}
\end{equation}
Let $R_{s,1,+}\subset R_{s,1}$ be the unique system of positive roots
such that $R_{s,1,+}\subset R_{1,+}$, and let $F_{s,1}$ be the corresponding
basis of simple roots of $R_{s,1}$.
Then the isotropy group $W_s\subset W_0$ of $s$
is of the form
\begin{equation}
W_s=W(R_{s,1})\rtimes\Gamma_s
\end{equation}
where $\Gamma_s=\{w\in W_s\mid w(R_{s,1,+})=R_{s,1,+}\}$ is a group acting
through diagram automorphisms on the based root system $(R_{s,1},F_{s,1})$.

We form a new root datum $\mc{R}_s=(X,R_{s,0},Y,R_{s,0}^\vee,F_{s,0})$
and observe that $R_{nr,s}\subset R_{nr}$. Hence we can define a surjective
map $\mc{Q}(\mc{R})\to\mc{Q}(\mc{R}_s)$ (denoted by $q\to q_s$)
by restriction of the corresponding parameter function on $R_{nr}^\vee$
to $R_{nr,s}^\vee$. \index{qs@$q_s$, parameter corresponding to $s \in T_u$}
\end{defn}
\index{1cs@$\Gamma_s, \Gamma (t)$, groups of diagram automorphisms of
$(R_{s,1},F_{s,1})$} \index{W0t@$W_0 (t)$, subgroup of $W_0$}
Let $t=cs\in T_vT_u$ be the polar decomposition of an
element $t\in T$. We define $W_0(t)\subset W_s$ for the subgroup
defined by
\begin{equation}
W_0 (t) := \{w\in W_s\mid wt\in W(R_{s,1})t\}
\end{equation}
Observe that $W_0(t)$ is the semidirect product
$W_0(t)=W(R_{s,1})\rtimes\Gamma(t)$ where
\begin{equation}\label{eq:Gamma}
\Gamma(t)=\Gamma_s\cap W_0(t)
\end{equation}

Let $M_{W_0t}\subset\mc{Z}$ denote maximal ideal of $\mc{A}$
of elements vanishing at $W_0t\subset T$, and let
$\overline{\mc{Z}}$ be the $M_{W_0t}$-adic
completion of $\mc{Z}$. We define
\index{Z@$\overline{\mc{Z}}$, completion of $\mc Z$ at $W_0 t \in W_0 \backslash T$}
\index{A@$\overline{\mc{A}}$, completion of $\mc A$ at $W_0 t \in W_0 \backslash T$}
\begin{equation}
\overline{\mc{A}}=\mc{A}\otimes_{\mc{Z}}\overline{\mc{Z}}
\end{equation}
By the Chinese remainder theorem we have
\begin{equation}
\overline{\mc{A}}=\bigoplus_{t^\prime\in W_0t}\overline{\mc{A}}_{t^\prime}
\end{equation}
where $\overline{\mc{A}}_{t^\prime}$ denotes the formal completion of $\mc{A}$
at $t^\prime\in T$. Let $1_{t^\prime}$ denote the unit of the summand
$\overline{\mc{A}}_{t^\prime}$ in this direct sum decomposition.
We consider the formal completion
\index{HRq@$\overline{\mc{H}}(\mc{R},q)$, formal completion of $\mc H (\mc R ,q)$}
\begin{equation}\label{eq:Hcompl}
\overline{\mc{H}}(\mc{R},q)=
\mc{H}(\mc{R},q)\otimes_{\mc{Z}}\overline{\mc{Z}}
\end{equation}
On the other hand, we consider the affine Hecke algebra
$\mc{H}(\mc{R}_s,q_s)$ and its commutative subalgebra $\mc{A}_s$ (as defined
before when discussing the Bernstein basis) and center
$\mc{Z}_s=\mc{A}_s^{W(R_{s,1})}$.
Let $m_{W(R_{s,1})t}$ be the maximal ideal in $\mc{Z}_s$ of elements vanishing
at the orbit $W(R_{s,1})t=sW(R_{s,1})c$; let $\overline{\mc{Z}_s}$
and $\overline{\mc{H}}(\mc{R}_s,q_s)$ be the corresponding formal
completions as before.
\index{Zs@$\mc{Z}_s$, center of $\mc{H}(\mc{R}_s,q_s)$}

The group $\Gamma(t)$ acts on $\overline{\mc{H}}(\mc{R}_s,q_s)$
and on its center $\overline{\mc{Z}_s}$.
We note that there exists a canonical isomorphism
\begin{equation}
\overline{\mc{Z}}\to\overline{\mc{Z}_s}^{\Gamma(t)}
\end{equation}
As before we define a localization
\begin{equation}
\overline{\mc{H}}(\mc{R}_s,q_s)=
{\mc{H}}(\mc{R}_s,q_s)\otimes_{\mc{Z}_s}\overline{\mc{Z}_s}
\end{equation}
Let $e_t\in\overline{\mc{A}}\subset\overline{\mc{H}}(\mc{R},q)$
be the idempotent defined by
\index{et@$e_t$, idempotent in $\overline{\mc A}$ associated to $t \in T$}
\begin{equation}
e_t=\sum_{t^\prime\in W(R_{s,1})t} 1_{t^\prime}
\end{equation}
\begin{thm}\label{thm:red1} (``First reduction Theorem'' (see \cite[Theorem 8.6]{Lus2}))

Let $q\in \mc{Q}$ and let $t=cs$ be the polar decomposition of an element $t\in T$.
Let $n$ be the cardinality of the orbit $W_0t$ divided by the cardinality of the orbit
$W(R_{s,1})t$. Using the notations introduced above, there exists an isomorphism
of $\overline{\mc{Z}}$-algebras
\begin{equation}
(\overline{\mc{H}}(\mc{R}_s,q_s)\rtimes\Gamma(t))_{n\times n}
\to\overline{\mc{H}}(\mc{R},q)
\end{equation}
Via this isomorphism the idempotent $e_t\in\overline{\mc{H}}(\mc{R},q)$
corresponds to the $n\times n$-matrix with $1$ in the upper left corner
and $0$'s elsewhere. Hence the $\overline{\mc{Z}}$-algebras
$\overline{\mc{H}}(\mc{R},q)$ and
$\overline{\mc{H}}(\mc{R}_s,q_s)\rtimes\Gamma(t)$
are Morita equivalent. In particular the set of simple modules $U$ of
$\mc{H}(\mc{R},q)$ with central character $W_0t$ corresponds
bijectively to the set of simple modules $V$ of
$\mc{H}(\mc{R}_s,q_s)\rtimes\Gamma(t)$ with central character
$W_0(s)t=W(R_{s,1})t$, where the bijection is given by $U\to e_tU$.
\end{thm}
\begin{proof} The proof is a straightforward translation of Lusztig's
proof of \cite[Theorem 8.6]{Lus2}. We replace the equivalence relation
that Lusztig defines on the orbit $W_0t$ by the equivalence relation
induced by the action of $W(R_{s,1})$ (i.e. the equivalence
classes are the orbits of $W(R_{s,1})$ in $W_0t$; in other words,
the role of the subgroup $\mc{J}\langle v_0\rangle\subset T$ in Lusztig's
setup is now played by the vector subgroup $T_{v}$).
After this change the rest of the proof is identical to Lusztig's
proof.
\end{proof}
The second reduction theorem gives a bijection between simple modules of affine
Hecke algebra's with central character $W_0t$ satisfying $\alpha(t)>0$ for
all $\alpha\in R_1$ and simple modules of an associated degenerate affine Hecke
algebra with a real central character. We first need to define the appropriate
notion of the associate degenerate affine Hecke algebra.

Let $\mc{R}=(X,R_0,Y,R_0^\vee,F_0)$ be a root datum, let $q\in
\mc{Q}$, and let $W_0t\in W_0\backslash T$ be a central character
such that for all $\alpha\in R_1$ we have
$\alpha(t)\in\mathbb{R}_{>0}$. Then the polar decomposition of $t$
has the form $t=uc$ with $u\in T_u$ a $W_0$-invariant character of
$X$ and with $c\in T_v$ a positive character of $X$. Observe that
$\beta(u)=1$ if $\beta\in R_0\cap R_1$ and $\beta(u)=\pm 1$ if
$\beta\in R_0\cap \frac{1}{2}R_1$. We define a $W_0$-invariant
real parameter function $k_u:R_1\to\mathbb{R}$ by the following
prescription. If $\alpha\in R_1$ we put:
\index{kzu@$k_u\in\mc{K}$, parameter function corresponding to $u \in T_u$}
\begin{equation}\label{eq:k}
k_{u,\alpha}=
\left\{
\begin{array}{ll}
  \log(q_{\alpha^\vee}^2)&\text{\ if\ }\alpha\in R_0\cap R_1\\
   \log(q_{\alpha^\vee}^2q_{2\alpha^\vee}^4)&\text{\ if\ }\alpha=2\beta\text{\ with\ }
   \beta\in R_0\text{\ and\ }\beta(u)=1\\
   \log(q_{\alpha^\vee}^2)&\text{\ if\ }\alpha=2\beta\text{\ with\ }
   \beta\in R_0\text{\ and\ }\beta(u)=-1
\end{array}
  \right.
\end{equation}
\begin{defn}
\index{V@$V = \mh R \otimes X$} \index{kz@$k\in\mc{K}$, parameter function $R_1 \to \mh R$}
\index{H@$\mathbf{H}(R_1,V,F_1,k)$, degenerate affine Hecke algebra}
We define the degenerate affine Hecke algebra $\mathbf{H}(R_1,V,F_1,k)$
associated with the root system $R_1\subset V^*$ where $V=\mathbb{R}\otimes_{\mh Z} Y$
and the parameter function $k$ as follows. We put $P(V)$ for the polynomial algebra on
the vector space $V$. The Weyl group $W_0$ acts on $P(V)$ and we denote the action by
$w\cdot f=f^w$. Then $\mathbf{H}(R_1,V,F_1,k)$ is simultaneously
a left $P(V)$-module and a right $\mathbb{C}[W_0]$-module, and as such it has the structure
$\mathbf{H}(R_1,V,F_1,k)=P(V)\otimes\mathbb{C}[W_0]$. We identify
$P(V)\otimes e\subset\mathbf{H}(R_1,V,F_1,k)$ with $P(V)$ and
$1\otimes \mathbb{C}[W_0]\subset\mathbf{H}(R_1,V,F_1,k)$ with $\mathbb{C}[W_0]$
so that we may write $f w$ instead of $f\otimes w$ if $f\in P(V)$ and $w\in W_0$.
The algebra structure structure of $\mathbf{H}(R_1,V,F_1,k)$ is uniquely determined by the
cross relation (with $f\in P(V)$, $\alpha\in F_1$ and $s=s_\alpha\in S_1$):
\begin{equation}
fs-sf^s=k_\alpha\frac{f-f^s}{\alpha}
\end{equation}
\end{defn}
It is easy to see that the center of $\mathbf{H}(R_1,V,F_1,k)$ is equal to the algebra
$\mathbf{Z}=P(V)^{W_0}\subset\mathbf{H}(R_1,V,F_1,k)$.
The vector space $V_c=\mathbb{C}\otimes V$ can be identified with the Lie algebra
of the complex torus $T$. Let $\exp:V_c\to T$ be the corresponding exponential map.
It is a $W_0$-equivariant covering map which restricts to a group
isomorphism $V\to T_v$ of the real vector space $V$ to the vector group $T_v$.

\begin{thm}\label{thm:red2}
(``Second reduction Theorem'' (see \cite[Theorem 9.3]{Lus2}))

Let $\mc{R}=(X,R_0,Y,R_0^\vee,F_0)$ be a root datum with parameter function
$q\in\mc{Q}=\mc{Q}(\mc{R})$. Let $V_0\subset V$ be the subspace spanned by
$R_0^\vee$. Given $t\in T$ such that $\alpha(t)>0$ for all $\alpha\in R_1$ we let
$\xi=\xi_t\in V_0$ be the unique element such that $\alpha(t)=\textup{e}^{\alpha(\xi)}$
for all $\alpha\in R_1$. It is easy to see that the map $t\to \xi=\xi_t$ is $W_0$-equivariant;
in particular the image of $W_0t$ is equal to $W_0\xi$. Let $t=uc$ be the polar
decomposition of $t$. Then $u\in T_u$ is $W_0$-invariant, and we define a
$W_0$-invariant parameter function $k_u$ on $R_1$ by (\ref{eq:k}). Let
$\overline{\mathbf{Z}}$ be the formal completion of the center $\mathbf{Z}$ of
$\mathbf{H}(R_1,V,F_1,k_u)$ at the orbit $W_0\xi$. Let $\mathbf{P}=P(V)$ and
put $\overline{\mathbf{P}}=\mathbf{P}\otimes_\mathbf{Z}\overline{\mathbf{Z}}$
and $\overline{\mathbf{H}}(R_1,V,F_1,k_u)=
\mathbf{H}(R_1,V,F_1,k_u)\otimes_\mathbf{Z}\overline{\mathbf{Z}}$.
There exist natural compatible isomorphism of algebras
$\overline{Z}\to\overline{\mathbf{Z}} ,\, \overline{\mc{A}}\to\overline{\mathbf{P}}$
and $\overline{\mc{H}}(\mc{R},q)\to\overline{\mathbf{H}}(R_1,V,F_1,k_u)$.
\end{thm}
\index{V0@$V_0$, subspace of $V$ spanned by $R_0^\vee$}
\index{PV@\textbf{P} = $P(V)$, algebra of polynomials on $V$}
\begin{proof}
This is a straightforward translation of the proof of \cite[Theorem 9.3]{Lus2}.
\end{proof}
\begin{cor} The set of simple modules of ${\mc{H}}(\mc{R},q)$ with central character
$W_0t$ (satisfying the above condition that $\alpha(t)>0$ for all $\alpha\in R_1$) and the
set of simple modules of ${\mathbf{H}}(R_1,V,F_1,k_u)$ with central character
$W_0\xi$ (as described in Theorem \ref{thm:red2}) are in natural bijection.
\end{cor}
Combining the two reduction theorems we finally obtain the following result
(see \cite[Section 10]{Lus2}):
\begin{cor}\label{cor:simple} For all $s\in T_u$
the center of ${\mathbf{H}}(R_{s,1},V,F_{s,1},k_s)\rtimes\Gamma(t)$ is
equal to $\mathbf{Z}^{\Gamma(t)}$. If $t\in T$ is arbitrary with polar
decomposition $t=sc$, then the set of
simple modules of ${\mc{H}}(\mc{R},q)$
with central character $W_0t$ is in natural bijection with the set of
simple modules of ${\mathbf{H}}(R_{s,1},V,F_{s,1},k_s)\rtimes\Gamma(t)$
with the real central character $W_s\xi$. Here $\xi\in V$ is the unique
vector in the real span of $R_{s,1}^\vee$ such that
$\alpha(t)=\textup{e}^{\alpha(\xi)}$ for all $\alpha\in R_{s,1}$,
$k_s$ is the real parameter function on $R_{s,1}$ associated to $q_s$
described by (\ref{eq:k}), and $\Gamma(t)$
is the group defined by (\ref{eq:Gamma}).
\end{cor}

\subsection{Harmonic analysis for affine Hecke algebras}

\subsubsection{The Hilbert algebra structure of the Hecke algebra}

Let $\mc{R}$ be a based root datum and $q\in\mc{Q}$ a
\emph{positive} parameter function \index{*, anti-involution on $\mc H$}
for $\mc{R}$. We turn $\mc{H}=\mc{H}(\mc{R},q)$ into a $*$-algebra using
the conjugate lineair anti-involution $*:\mc{H}\to\mc{H}$ defined by
$N_w^*=N_{w^{-1}}$. We define a trace $\tau:\mc{H}\to\mathbb{C}$ by
$\tau(N_w)=\delta_{w,e}$. This defines a Hermitian form
$(x,y):=\tau(x^*y)$ with respect to which the basis $N_w$ is orthonormal.
In particular $(\cdot,\cdot)$ is positive definite.
\index{1t@$\tau$, trace on $\mc H$}
In fact it is easy to show \cite{Opd1} that this Hermitian inner
product defines the structure of a Hilbert algebra on $\mc{H}$.
\index{L2H@$L^2 (\mc H)$, Hilbert space completion of $\mc H$}
\index{C*@$\mathfrak C = C^*_r (\mc H (\mc R,q)) ,\; C^*$-completion of $\mc H$}
\index{1lz@$\lambda$, left regular representation of $\mc H$}
Let $L^2(\mc{H})$ be the Hilbert space completion of $\mc{H}$ and
$\lambda : \mc H \to B(L^2 (\mc H))$ the left regular representation of $\mc H$.
Let $\mathfrak{C}:=C^*_r(\mc{H})$ be the $C^*$-algebra completion of
$\lambda (\mc{H})$ inside $B(L^2(\mc{H}))$. It
is called the (reduced) $C^*$-algebra of $\mc{H}$. It is not hard
to show that $\mathfrak{C}$ is unital, separable and liminal, which
implies that the spectrum $\hat{\mathfrak{C}}$ of $\mathfrak{C}$ is
a compact $T_1$ space with countable base which contains an open dense
Hausdorff subset. The trace $\tau$ extends to a finite
tracial state $\tau$ on $\mathfrak{C}$. In this situation (see
\cite[Theorem 2.25]{Opd1}) there exists a unique positive Borel measure
$\mu_{Pl}$ on $\hat{\mathfrak{C}}$ such that for all $h\in\mc{H}$:
\begin{equation}\label{eq:plan}
\tau=\int_{\hat{\mathfrak{C}}}\chi_\pi d\mu_{Pl}(\pi)
\end{equation}
Since $\tau$ is faithful it follows that the support of $\mu_{Pl}$ is equal to
$\hat{\mathfrak{C}}$. \index{1mp@$\mu_{Pl}$, Plancherel measure of $\mc{H}$}
\begin{defn}
We call the measure $\mu_{Pl}$ the Plancherel measure of $\mc{H}$.
\end{defn}
\begin{defn}
An irreducible $*$-representation $(V,\pi)$ of the involutive
algebra $\mc{H}$ is called a discrete series representation of $\mc{H}$
if $(V,\pi)$ extends to a representation (also denoted $(V,\pi)$) of
$\mathfrak{C}$ which is equivalent to a subrepresentation of the
left regular representation of $\mathfrak{C}$ on $L^2(\mc{H})$.
In this case the finite trace $\chi_\pi$ defined by
$\chi_\pi(x)=\textup{Tr}_V(\pi(x))$ is called an irreducible
discrete series character.
\end{defn}
\index{1x@$\chi_\pi$, trace of a representation $\pi$}
We have seen that an irreducible representation $(V,\pi)$ of $\mc{H}$
is finite dimensional. In particular its character $\chi_\pi$ is a well defined
linear functional on $\mc{H}$.
We call $\chi_\pi$ an irreducible character of $\mc{H}$.
Clearly the character of a finite dimensional
representation of $\mc{H}$ only depends on the equivalence class of
the underlying representation. The irreducible characters of a set
of mutually inequivalent irreducible representations of $\mc{H}$
are linearly independent (see \cite[Corollary 2.11]{Opd1}). Hence
the equivalence class of a finite dimensional semisimple representation
is completely determined by its character. \index{1dq@$\Delta(\mc{R},q)$,
set of irreducible discrete series characters of $\mc{H}$}
\begin{defn} We denote by $\Delta(\mc{R},q)$ the set of irreducible
discrete series characters of $\mc{H}(\mc{R},q)$. For each irreducible
character $\chi\in\Delta(\mc{R},q)$ we choose and fix an irreducible
discrete series representation $(V,\delta)$ of $\mc{H}$ such that
$\chi=\chi_\delta$ (by abuse of language we will often identify
the set of irreducible discrete series characters and (the
chosen set of representatives of) the set of equivalence
classes of irreducible discrete series representations).
\end{defn}
The following criterion for an irreducible
representation $(V,\pi)$ of $\mc{H}$ to belong to the discrete series
follows from a general result of Dixmier (see \cite{Opd1}):
\begin{cor}\label{cor:dix} $(V,\pi)$ is a discrete series
representation iff $\mu_{Pl}(\{\pi\})>0$.
\end{cor}
\begin{cor}\label{cor:idemprim}(see \cite[Proposition 6.10]{Opd1})
There is a 1-1 correspondence between the set of irreducible
discrete series characters $\chi_\delta$ and the set of primitive central
Hermitian idempotents $e_\delta\in\mathfrak{C}$ of finite rank. The correspondence
is such that $\tau(e_\delta x)=\mu_{Pl}(\{\delta\})\chi_\delta(x)$ for all
$x\in\mc{H}$.
\end{cor}
\begin{cor}(see \cite[Proposition 6.10]{Opd1})\label{cor:DSfinite}
$(V,\pi)$ is a discrete series representation iff $\{[\pi]\}\subset\hat{\mathfrak{C}}$
is a connected component of $\hat{\mathfrak{C}}$. In particular,
the number of irreducible discrete series characters is finite.
\end{cor}
\subsubsection{The Schwartz algebra}
\label{subsec:Schwartz}

We define a nuclear Fr\'echet algebra $\mc{S}=\mc{S}(\mc{R},q)$
(the Schwartz algebra) which plays a pivotal role in the spectral
theory of the trace $\tau$ on $\mc{H}$.
\begin{defn}\label{defn:inner}
We choose once and for all a $W_0$-invariant inner
product $\langle\cdot,\cdot\rangle$ on the vector space
$V^*:=\mathbb{R}\otimes X$, which takes integral values on $X \times X$.
\end{defn}
Let $V^*_0$ be the real vector space spanned by $R_0$.
Its orthocomplement is the vector space
$V_Z^*=\mathbb{R}\otimes Z$ spanned by the center $Z$ of $W$.
Given $\phi\in V^*$ we decompose $\phi=\phi_0+\phi_Z$ with respect
to the orthogonal decomposition $V^*=V_0^* \oplus V_Z^*$.
\begin{defn}
We define a norm $\mc{N}:W\to\mathbb{R}_+$ on $W$ as follows:
if $w\in W$ we put \index{N@$\mc N$, norm on $W$}
\begin{equation}
\mc{N}(w)=l(w)+\Vert w(0)_Z\Vert
\end{equation}
\end{defn}
Next we define
seminorms $p_n:\mc{H}\to\mathbb{R}_+$ on $\mc{H}$ by
\begin{equation}
p_n(h):=\textup{max}_{w\in W}(1+\mc{N}(w))^n|(N_w,h)|
\end{equation}
\index{pz@$p_n$, seminorm on $\mc H$}
\index{SRq@$\mc S = \mc S (\mc R ,q)$, Schwartz algebra of $\mc H$}
\begin{defn} The Schwartz algebra $\mc{S}$ of $\mc{H}$ is the
completion of $\mc{H}$ with respect to the system of seminorms
$p_n$ with $n\in\mathbb{N}$.
\end{defn}
\begin{thm}\label{thm:frealg}(\cite{Opd1}, \cite{Sol};
see Appendix A)
The completion $\mc{S}$ is a Fr\'echet
algebra which is continuously and densely embedded in
$\mathfrak{C}$.
\end{thm}
\begin{rem} The Fr\'echet algebra
$\mc{S}$ is independent of the choice made in
Definition \ref{defn:inner}. $\mc{S}$ is also
independent of $q\in\mc{Q}$ as a Fr\'echet space.
\end{rem}
\begin{defn} A finite dimensional representation of $\mc{H}$
is called tempered if it has a continuous extension to $\mc{S}$.
\end{defn}
The Fr\'echet algebra structure of $\mc{S}$ depends
on $q\in\mc{Q}$.
The basic theorem \ref{thm:frealg} was first proven in \cite{Opd1} using
some qualitative analysis on the spectrum of $\mathfrak{C}$; the
proof in \cite{Sol} is more direct and uses an elementary but nontrivial
result due to Lusztig \cite{Lus1} on the multiplication table of $\mc{H}$ with
respect to the basis $N_w$. The latter proof also
reveals the following important fact with respect to the dependence
of $q\in\mc{Q}$:
\begin{thm}\label{thm:qcont}(\cite{Sol}; see Appendix A)
The dense subalgebra $\mc{S}\subset\mathfrak{C}$ is closed for
holomorphic calculus (also see \cite[Corollary 5.9]{DeOp1}).
The holomorphic calculus is continuous on $\mc{S}\times\mc{Q}$
in the following sense. Let $U\subset\mathbb{C}$ be an open set.
The set $V_U\subset\mc{S}\times\mc{Q}$ defined by
$V_U=\{(x,q)\mid\textup{Sp}(x,q)\subset U\}$ is open.
For any holomorphic function $f:U\to\mathbb{C}$ the map
$V_U\ni(x,q)\to f(x,q)\in\mc{S}$ is continuous.
\end{thm}
The following result shows the fundamental role
of $\mc{S}$ for the spectral theory of $\tau$:
\begin{thm}(\cite[Corollary 4.4]{DeOp1})
The support of $\mu_{Pl}$ consists precisely of the set of
equivalence classes of irreducible tempered representations
of $\mc{H}$.
\end{thm}
In particular the discrete series representations are tempered.
There are various characterizations of tempered representations
and of discrete series representations. Casselman's criterion
states that:
\begin{thm}\label{thm:cas}(Casselman's criterion, see
\cite[Lemma 2.22]{Opd1})
Let $(V,\delta)$ be an irreducible representation of $\mc{H}$.
The following are equivalent:
\begin{enumerate}
\item $(V,\delta)$ is a discrete series representation of $\mc{H}$.
\item All matrix coefficients of $(V,\delta)$ belong to $L^2(\mc{H})$.
\item The character $\chi_\delta$ of $(V,\delta)$ belongs to $L^2(\mc{H})$.
\item All generalized $\mc{A}$-weights $t\in T$ in $V$ satisfy:
$|x(t)|<1$ for all $x\in X^+\backslash \{0\}$.
\item For every matrix coefficient $m$ of $\delta$ there exist constants
$C,\epsilon>0$ such that $|m(N_w)|<C\textup{e}^{-\epsilon\mc{N}(w)}$ for all
$w\in W$.
\item The character $\chi_\delta$ of $(V,\delta)$ belongs to $\mc{S}$.
\end{enumerate}
\end{thm}
\begin{cor}\label{cor:affidem}
An irreducible representation $(V,\delta)$ of $\mc{H}$
is an irreducible discrete series representation iff $(V,\delta)$ is
afforded by a central primitive idempotent
$e_\delta\in\mc{S}$ of $\mc{S}$ (see Corollary \ref{cor:idemprim}).
\end{cor}
\begin{cor}\label{cor:Dss}
The set $\Delta(\mc{R},q)$ is nonempty only if $\mc{R}$ is semisimple.
\end{cor}
Casselman's criterion for discrete series in terms of the generalized
$\mc{A}$-weights can be transposed to define the notion of
discrete series modules over a crossed product
$\mathbf{H}(R_1,V,F_1,k)\rtimes\Gamma$
of a degenerate affine Hecke algebra
$\mathbf{H}(R_1,V,F_1,k)$ with a real parameter function $k$
and a finite group $\Gamma$ acting by diagram automorphisms of $(R_1,F_1)$ (thus
a simple module $(U,\delta)$ is a discrete series representation iff
the generalized ${\bf{P}}$-weights in $U$ are in the interior of the antidual
cone ($\subset V$) of the simplicial cone spanned by $F_1$).
It is clear that this definition is compatible with the bijections
afforded by the two reduction theorems
(Theorem \ref{thm:red1} and Theorem \ref{thm:red2}).
Hence we obtain from Corollary \ref{cor:simple}:
\begin{cor}\label{cor:dsred}
Let $t\in T$ with polar decomposition $t=sc$. The set $\Delta_{W_0t}$
of equivalence classes of irreducible discrete series representations
of ${\mc{H}}(\mc{R},q)$
with central character $W_0t$ is in natural bijection with the set of
equivalence classes of irreducible discrete series representations of
${\mathbf{H}}(R_{s,1},V,F_{s,1},k_s)\rtimes\Gamma(t)$
with the real central character $W_s\xi$.
Here $\xi\in V$ is the unique
vector in the real span of $R_{s,1}^\vee$ such that
$\alpha(t)=\textup{e}^{\alpha(\xi)}$ for all $\alpha\in R_{s,1}$,
$k_s$ is the real parameter function on $R_{s,1}$ described by
(\ref{eq:k}), and $\Gamma(t)$
is the group of diagram automorphisms of $(R_{s,1},F_{s,1})$
of (\ref{eq:Gamma}).
\end{cor}
\begin{cor}
If $\Delta_{W_0t}\not=\emptyset$ then the polar
decomposition $t=sc$ of $t$ has the property that
$R_{s,1}\subset R_1$ is a root subsystem of maximal rank.
\end{cor}
If $s=u\in T_u$ is $W_0$-invariant (i.e. if $\alpha(u)=1$ for all
$\alpha\in R_1$) then we obtain from Corollary \ref{cor:dsred}:
\begin{cor}\label{cor:dsgradedc}
Let $u\in T_u$ be $W_0$-invariant, and let $c\in T_v$.
There is a natural bijection between the set
$\Delta(\mc{R},q)_{uW_0c}$ of irreducible discrete series
characters of $\mc{H}(\mc{R},q)$ with central character of the
form $uW_0c\subset W_0\backslash T$ and the set
of irreducible discrete series characters of
${\mathbf{H}}(R_1,V,F_1,k_u)$ with the real
infinitesimal central character
$W_0\log(c)$.
\end{cor}
It is not hard to show that the central character of an irreducible
discrete series character of ${\mathbf{H}}(R_1,V,F_1,k_u)$ is real
(see \cite[Lemma 1.3.4]{SlootenThesis}).
Hence the previous corollary in particular says that:

\begin{cor}\label{cor:dsgraded}
Let $u\in T_u$ be $W_0$-invariant.
There is a natural bijection between the set $\Delta^{u}(\mc{R},q)$
of irreducible discrete series characters of $\mc{H}(\mc{R},q)$ with
a central character of the form $uW_0c$ with $c\in T_v$ on the one hand,
and the set $\Delta^{\mathbf{H}}(R_1,V,F_1,k)$ of irreducible discrete series
characters of ${\mathbf{H}}(R_1,V,F_1,k_u)$ on the other hand. In this bijection
the correspondence of the central characters is as described above.
\end{cor}
\index{1duR@$\Delta^{u}(\mc{R},q)$, elements of $\Delta (\mc{R},q)$ with
central character in $W_0 u T_v$} \index{1dHR@$\Delta^{\mathbf{H}}(R_1,V,F_1,k)$,
irreducible discrete series characters of ${\mathbf{H}}(R_1,V,F_1,k_u)$}
We can use Corollary \ref{cor:dsred} to reduce the general
classification problem of the irreducible discrete series characters
of $\mc{H}(\mc{R},q)$ for any semisimple root datum $\mc{R}$ to the case
of discrete series characters of a degenerate
affine Hecke algebra as well, but we have to pay the price of having
to deal with crossed products by certain groups of diagram automorphisms.
In order to deal with the crossed products one has to resort to Clifford theory
(cf. \cite{RamRam}).

Corollary \ref{cor:affidem} gives us yet another characterization of
the irreducible discrete series representations:
\begin{thm}\label{thm:dscrit} Let $(V,\delta)$ be a simple module over
$\mc{H}$. Equivalent are:
\begin{enumerate}
\item $(V,\delta)$ is a discrete series representation of $\mc{H}$.
\item $(V,\delta)$ extends to a projective $\mc{S}$-module.
\end{enumerate}
\end{thm}

\subsubsection{The Euler-Poincar\'e pairing and elliptic characters}

We recall the main result of \cite{OpdSol}:
\begin{thm} The affine Hecke algebra $\mc{H}=\mc{H}(\mc{R},q)$ has
global homological dimension equal to the rank of $X$. If $U,V$
are finite dimensional tempered $\mc{H}$-modules then
for all $i$ we have $\textup{Ext}_\mc{H}^i(U,V)\cong
\textup{Ext}_\mc{S}^i(U,V)$.
\end{thm}
Define the \index{G@$G(\mc H)$, complexified Grothendieck group of $\mc H$}
Euler-Poincar\'e pairing on the (complexified) Grothendieck group
$G(\mc{H})$ of finite dimensional virtual characters
by sesquilinear extension from the formula
\index{EPH@$EP_{\mc H}$, Euler-Poincar\'e pairing}
\begin{equation}
\textup{EP}_{\mc H} (U,V):=\sum_{i=0}^\infty(-1)^i
\textup{dim}(\textup{Ext}_\mc{H}^i(U,V))
\end{equation}
It can be seen that this defines a Hermitian positive
semidefinite pairing on $G(\mc{H})$ (\cite[Theorem 3.5]{OpdSol}).
The above result combined with Theorem \ref{thm:dscrit}
implies that:
\begin{cor}\label{cor:ON}
The irreducible discrete series characters of $\mc{H}$ form an orthonormal
set with respect to $\textup{EP}_{\mc H}$
and are orthogonal to all irreducible tempered characters
that are not in the discrete series.
\end{cor}
\index{Ell$(\mc H)$, quotient of $G (\mc H)$}
Another crucial result of \cite{OpdSol} says that $\textup{EP}_{\mc H}$
factors through the quotient $\textup{Ell}(\mc{H})$ of $G(\mc{H})$
by the subspace spanned by all the properly induced finite
dimensional tempered characters. Then $\textup{Ell}(\mc{H})$ is a finite
dimensional $\mc{Z}$-module, equipped with a positive semidefinite Hermitian
pairing $\textup{EP}_{\mc H}$ with respect to which elements with
a disjoint support on $W_0\backslash T$ are orthogonal. Let $\textup{Ell}_{W_0 t}
(\mc{H})$ be the $\mc Z$-submodule corresponding to $W_0 t$.

There exists a scaling map $\tilde \sigma_0 : G(\mc{H})\to G(W)$ (see
\cite[Theorem 1.7]{OpdSol}) which descends to a map
\[
\tilde{\sigma}_0 : \textup{Ell}(\mc{H}) \to \textup{Ell}(W) = \textup{Ell}(\mh C [W])
\]
The finite dimensional $\mc{Z}$-module $\textup{Ell}(W)$ can be
described completely explicitly in terms of the elliptic characters
of the isotropy groups $W_t$ (with $t\in T$) for the action of $W_0$
on $T$. The pairing $\textup{EP}_W$ on $\textup{Ell}(W)$ can be described
in these terms as well, and it turns out that $\textup{EP}_W$ is positive
definite on $\textup{Ell}(W)$ (for all these results, consult
\cite[Chapter 3]{OpdSol}). It turns out that $\textup{Ell}(W)$ is
nonzero only if $\mc{R}$ is semisimple, and that the support of
$\textup{Ell}(W)$ as a $\mc{Z}$-module is contained in the
the set of orbits $W_0s$ such that $R_{s,1}\subset R_1$ is
of maximal rank. From \cite{OpdSol} we have:
\begin{thm}\label{thm:scale}
\begin{enumerate}
\item The map $\tilde{\sigma}_0:\textup{Ell}(\mc{H})\to\textup{Ell}(W)$ is isometric
with respect to $\textup{EP}_{\mc H}$ and $\textup{EP}_W$.
\item For all $t\in T$ we have $\tilde{\sigma}_0(\textup{Ell}_{W_0t}(\mc{H}))\subset
\textup{Ell}_{W_0s}(W)$, where $t=sc$ with $s\in T_u$ and $c\in T_v$ is the polar
decomposition of $t$.
\end{enumerate}
\end{thm}
Combined with Corollary \ref{cor:ON} we obtain the following
upper bounds for the number of discrete series characters.
\begin{cor}\label{cor:upper}
If $s\in T_u$ then $W_s$ denoted the isotropy group of $s$ in $W_0$.
We call $w\in W_s$ elliptic if $s$ is an isolated fixed point
of $w$. Let $\textup{ell}(W_s)$ be the number of conjugacy classes
of $W_s$ consisting of elliptic elements of $W_s$.
For $s\in T_u$ we denote by
$\Delta^s(\mc{R},q)\subset \Delta(\mc{R},q)$ the subset consisting
of the irreducible discrete series characters of $\mc{H}(\mc{R},q)$
whose central characters are $W_0$-orbits which are contained
in the set $W_0sT_v$. Then $|\Delta^s(\mc{R},q)|\leq\textup{ell}(W_s)$.
\end{cor}

\subsection{The central support of tempered characters}

In this section deformations in the parameters $q$ of the Hecke algebra
play a fundamental role. Let us fix some notations and basic
structures. Recall that we attach to a based root datum
$\mc{R}=(X,R_0,Y,R_0^\vee,F_0)$ in a canonical way a parameter
space $\mc{Q}=\mc{Q}(\mc{R})$. This parameter space is itself a
vector group, defined as the space of length multiplicative
functions $q:W\to \mathbb{R}_+$ with the additional requirement
that $q|_\Omega=1$.

The following proposition is useful in order to reduce statements
about residual points to the case of simple root data.
\begin{prop}\label{prop:prod}
Let $\mc{R}=(X,R_0,Y,R_0^\vee,F_0)$ be a semisimple based root
datum.
\begin{enumerate}
\item[(i)]
Let $R_0=R_0^{(1)}\times\dots \times R_0^{(m)}$ be the
decomposition of $R_0$ in irreducible components. We denote by
$X^{(i)}$ be the projection of the lattice $X$ onto
$\mathbb{R}R_0^{(i)}$, and we define
$\mc{R}^{(i)}=(X^{(i)},R_0^{(i)},Y^{(i)},(R_0^{(i)})^\vee,F_0^{(i)})$
and $\mc{R}^\prime=\mc{R}^{(1)}\times\dots\times\mc{R}^{(m)}$.
Then the natural inclusion $X\hookrightarrow X^\prime$ defines an
isogeny $\psi:\mc{R}\to\mc{R}^\prime$
and if $\mc{Q}^{(i)}$ denotes be the parameter space of the root datum
$\mc{R}^{(i)}$ then $\psi$ yields a natural identification
$\mc{Q}({\mc{R}})=
\mc{Q}(\mc{R}^\prime)=\mc{Q}^{(1)}\times\dots\times \mc{Q}^{(m)}$.
\item[(ii)] We replace $X$ by the lattice $X^{max}=P(R_1)$, the weight
lattice of $R_1$ and denote the
resulting root datum by $\mc{R}^{max}$. Then $\mc{R}^{max}$ is a direct
product of irreducible root data and there exists an isogeny
$\psi:\mc{R}\to\mc{R}^{max}$ which yields a natural identification
$\mc{Q}({\mc{R}})=\mc{Q}({\mc{R}^{max}})$.
\end{enumerate}
\end{prop}
\index{Xma@$X^{max}$, extension of $X$}
\index{R2@$\mc R^{max}$, root datum with $X^{max}$}
\begin{proof}
A length multiplicative function $q:W\to\mathbb{R}_+$ is
determined by its restriction to the set of simple affine roots
and this restriction is a function which is constant on the
intersection of the $W$-orbits of affine roots intersected with
the simple affine roots. Conversely every such function on the
simple affine roots can be extended uniquely to a length
multiplicative function on $W$. The group $\Omega\simeq
X/Q(R_0)\subset W$ of elements of length $0$ acts on the set of
simple affine roots by diagram automorphisms which preserve the
components of the affine Dynkin diagram of the affine root system
$R^a=R_0^\vee\times\mathbb{Z}$. The action of $\Omega$ on the i-th
component factors through the action of
$\Omega^{(i)}:=X^{(i)}/Q(R^{(i)}_0)$. This proves (i).
We also see by this that length multiplicative function
$q\in\mc{Q}({\mc{R}})$ extends uniquely to a length
multiplicative function for
$W(\mc{R}^{max})$, since $\alpha^\vee\not\in 2Y$ for all
$\alpha\in R_0^\prime$ with $\mc{R}^\prime$ an indecomposable
summand which is not isomorphic to an irreducible root datum of
type $C_n^{(1)}$. This proves (ii).
\end{proof}
Given a root datum $\mathcal{R}$ and positive parameter function
$q\in\mc{Q}({\mathcal{R}})$ we define the Macdonald $c$-function of
the pair $(\mathcal{R},q)$. This is the rational function $c$ on
the torus $T=\operatorname{Hom}(X,\mathbb{C}^\times )$ defined by
\begin{equation}\label{eq:defc}
c=\prod_{\alpha\in R_{1,+}}c_\alpha,
\end{equation}
where $c_\alpha$ is defined for $\alpha\in R_1$ by
\begin{equation}\label{eq:c}
c_\alpha(t,q):=\frac{(1+q_{\alpha^\vee}^{-1}\alpha(t)^{-1/2})
(1-q_{\alpha^\vee}^{-1}q_{2\alpha^\vee}^{-2}\alpha(t)^{-1/2})}
{1-\alpha(t)^{-1}}
\end{equation}
Observe that the function $c_\alpha$ is rational in $t$
despite the appearance of the square root $\alpha(t)^{1/2}$. Indeed, if
$\alpha/2\not\in X$ then we have $q_{2\alpha^\vee}=1$, and the
numerator simplifies to $(1-q_{\alpha^\vee}^{-2}\alpha(t)^{-1})$.

The pole order at $t=r\in T$ of the rational function
\begin{equation}\label{eq:eta}
\eta(t):=(c(t)c(t^{-1}))^{-1}
\end{equation}
is defined as follows. \index{1g@$\eta$, rational function on $T$}
By definition $\eta(t)$ is a product of rational functions
of the form $\eta_\alpha:=(c_\alpha(t)c_\alpha(t^{-1}))^{-1}$ where $\alpha$ runs
over the set $R_{1,+}$. Let $\beta\in R_0$ be the unique root such that $\alpha$ is a
positive multiple of $\beta$.
Then $\eta_\alpha$ is the pull back via $\beta$ of a rational function $\rho_\alpha$
on $\mathbb{C}^\times$; we define the pole order  of $\eta_\alpha$ at $r$ to be equal
to minus the order of $\rho_\alpha$ at $\beta(r)\in\mathbb{C}^\times$.
The pole order $i_{\{r\}}$ of $\eta$ at $r\in T$ is
defined as the sum of these pole orders.
\begin{thm}\cite[Theorem 6.1]{Opd3}\label{thm:ord}
For any point $r\in T$, the pole order $i_{\{r\}}$ of $\eta(t)$ at
$t=r$ is at most equal to the rank $\operatorname{rk}(R_0)$ of
$R_0$.
\end{thm}
\begin{defn}\label{def:respt}
We call $r\in T$ a \emph{residual point} of the pair
$(\mathcal{R},q)$ if $i_{\{r\}}=\operatorname{rk}(X)$. The set of
$(\mathcal{R},q)$-residual points is denoted by
$\operatorname{Res}(\mathcal{R},q)$.
\end{defn}
\index{ResR@Res$(\mc R ,q)$, set of $(\mc R,q)$-residual points in $T$}
In particular the set $\operatorname{Res}(\mathcal{R},q)$ is
nonempty only if $\mathcal{R}$ is a semisimple root datum.

The next result is trivial but it explains in conjunction with
Proposition \ref{prop:prod}
how residual points for $\mc{R}$ can be expressed in terms of
residual points of the simple factors of $\mc{R}^{max}$:
\begin{lem}\label{lem:prod}
Let $\mc{R}=(X,R_0,Y,R_0^\vee)$ be a semisimple root datum.
\begin{enumerate}
\item[(i)]
Suppose that $\mc{R}\to\mc{R}^\prime$ is an isogeny which
yields an identification $\mc{Q}=\mc{Q}^\prime$ (e.g.
$\mc{R}^\prime=\mc{R}^{max}$ as in Proposition \ref{prop:prod}).
For all $q\in\mc{Q}$ we have: $r^\prime\in\textup{Res}(\mc{R}^\prime,q)$
iff $r=r^\prime|_{X}\in\textup{Res}(\mc{R},q)$.
\item[(ii)] Suppose that $\mc{R}=\mc{R}^{(1)}\times\dots\times\mc{R}^{(m)}$
is a direct product of simple factors (e.g. if $\mc{R}=\mc{R}^{max}$
as in Proposition \ref{prop:prod}). Let
$T=T^{(1)}\times\dots\times T^{(m)}$ be the corresponding factorization of $T$
and let $\mc{Q}(\mc{R})=\mc{Q}^{(1)}\times\dots\times\mc{Q}^{(m)}$
be the corresponding factorization of $\mc{Q}$.

For all $q=(q^{(1)},\dots,q^{(m)})\in\mc{Q}$ we have
a natural bijection
\begin{equation}
\textup{Res}(\mc{R},q)\stackrel{\approx}{\longrightarrow}
\textup{Res}(\mc{R}^{(1)},q^{(1)})\times\dots\times
\textup{Res}(\mc{R}^{(m)},q^{(m)})
\end{equation}
such that $r\to(r^{(1)},\dots,r^{(m)})$ iff
$r=r^{(1)}\dots r^{(m)}$ with $r^{(i)}\in T^{(i)}$ for all $i=1,\dots,m$.
\end{enumerate}
\end{lem}
The following result is straightforward as well:
\begin{lem}\label{lem:mon}
Let $\mc{R}$ be a semisimple root datum with root parameter function
$q\in \mc{Q}$. Let $r\in T$ with polar decomposition of the form $r=sc$.
Let $\mc{R}_s=(X,R_{s,0},Y,R_{s,0}^\vee)$ be the root datum
with the root parameters $q_s$ as in Definition \ref{defn:iso}.
Then $r$ is a $(\mc{R},q)$-residual point iff
$r$ is a $(\mc{R}_s,q_s)$-residual point.
In particular $\mc{R}_s$ is semisimple in this case.
\end{lem}
Let $L\subset T$ be a coset of a subtorus $T^L\subset T$. We
decompose the product (\ref{eq:eta}) as follows
\begin{equation}\label{eq:dec}
\eta=\eta_L\eta^L
\end{equation}
where $\eta_L$ is the product of the factors $c_\alpha$ where
$\alpha\in R_{L,1}\subset R_1$, the subset of $R_1$ consisting of
the roots that are constant on $L$, and $\eta^L$ is the product
over the remaining roots. We define the order $i_L$ of $\eta$ at
$L$ as the order of $\eta_L$ at $L$, viewed as a point of the
quotient torus $T/T^L$. Hence by Theorem \ref{thm:ord} we have
$i_L\leq \operatorname{rank}(R_L)$
for all cosets $L$, and we define
\begin{defn}\label{def:rescos}
We call a coset $L\subset T$ a residual coset if
$i_L=\operatorname{codim}(L)$ (in particular, $L=T$ is residual).
If we denote $L=rT^L$ where $r\in T_L$, the subtorus such that
$\operatorname{Lie}(T_L)$ is the orthogonal complement of
$\operatorname{Lie}(T^L)$, then $L$ is residual iff $r$ is a
residual point for the restriction of $\eta_L$ to $T_L$. We call $r$ a center of $L$
and we define the tempered part of $L$ to be
$L^{\operatorname{temp}}:=rT^L_u$ (this is well defined).
\end{defn}
\index{Lte@$L^{\text{temp}}$, tempered part of residual coset $L \subset T$}
Recall the following useful results for residual cosets.
\begin{prop}\cite[Lemma 4.1]{Opd3}\label{prop:flag}
Let $L$ be a residual coset, $L\not=T$. Then there exists a
residual coset $M\supset L$ such that
$\dim (M)=\dim L +1$.
\end{prop}
From this result one proves easily by induction to the rank of
$R_0$ (alternatively, it follows from Corollary \ref{cor:DSfinite}
in view of Theorem \ref{thm:rescos}):
\begin{thm}\cite[compare Theorem 1.1]{Opd3}\label{thm:finite}
The set of residual points is finite.
\end{thm}
We will also need the following results:
\begin{thm}\cite[Theorem 7.4]{Opd3}\label{thm:real}
Define $t^*:=\overline{t}^{-1}$. Then
$W_0(L^{\operatorname{temp}})^*=W_0L^{\operatorname{temp}}$.
\end{thm}
\begin{thm}\cite[Theorem 6.1]{Opd3}\label{thm:nonnested}
If $L\not=M$ are residual cosets of $T$ then
$L^{\operatorname{temp}}\not\subset M^{\operatorname{temp}}$.
Equivalently, the restriction of $\eta^L$ to
$L^{\operatorname{temp}}$ is smooth.
\end{thm}
The relevance of the notion of residual cosets stems from:
\begin{thm}\cite[Theorem 3.29]{Opd1}, \cite[Theorem 6.1]{Opd3}\label{thm:rescos}
An orbit $W_0r\in W_0\backslash T$ is the central character of a
discrete series character of $\mc{H}(\mc{R},q)$ iff $r$ is a residual point,
and $W_0r$ is the central character of a tempered character of
$\mc{H}(\mc{R},q)$ iff $r\in S(q)$, where
\begin{equation}
S(q) = \bigcup_{L\ \operatorname{tempered}}L^{\operatorname{temp}}
\end{equation}
\end{thm}
\index{Sq@$S(q)$, union of tempered residual cosets for $(\mc R,q)$}
\begin{rem}
As we have seen above,
$\operatorname{Res}(\mathcal{R},q)\not=\emptyset$ only if
$\mathcal{R}$ is semisimple. By Lemma \ref{lem:prod} their classification
reduces to the case of simple root data.
The residual points for simple root
data have been classified (\cite[Section 4]{HO} and \cite[Appendix
A]{Opd1}), and various of the above properties of residual points
and cosets were first proven by classification. In \cite{Opd3}
most of these properties were proved conceptually (with exception
of \cite[Theorem A.14(iii), Theorem A.18]{Opd1}). In this paper we
will only use properties of residual points for which we know a
classification-free proof \emph{unless stated otherwise}.
\end{rem}

\subsection{Generic residual points}

We will study the deformation of discrete series characters
with respect to the parameter $q\in\mc{Q}$. We begin by studying
the dependence of the central characters on $\mc{Q}$.
We denote the set of all positive real parameter functions for
$\mc{R}$ by $\mc{Q}=\mc{Q}(\mc{R})$. Recall the following
terminology:
\index{q@\textbf{q}, number such that $q(s) = \mb{q}^{f_s}$}
\index{fz@$f_s$, real number such that $q(s) = \mb{q}^{f_s}$}
\begin{rem}\label{rem:base}
We choose a base $\mb{q}>1$ and define $f_s\in\mathbb{R}$ such that
$q(s)=\mb{q}^{f_s}$ for all $s\in S^{\mathrm{aff}}$. We equip
$\mc{Q}$ in the obvious way with the structure of the vector group
$\mathbb{R}_+^N$ where $N$ denotes the number of $W$-conjugacy
classes in $S^{\operatorname{aff}}$. Given a base $\mb{q}>1$ we
identify $\mc{Q}$ with the finite dimensional real vector space of
real functions $s\to f_s$ on $S^{\operatorname{aff}}$ which are
constant on $W$-conjugacy classes. In this sense we speak of
(linear) \emph{hyperplanes} in $\mc{Q}$ (this notion is
independent of $\mb{q}$). By a \emph{half line} in $\mc{Q}$ we
mean a family of parameter functions $q\in\mc{Q}$ in which the
$f_s\in\mathbb{R}$ are kept fixed and are not all equal to $0$ and
$\mb{q}$ is varying in $\mathbb{R}_{>1}$.
\end{rem}
As was remarked in \cite{Opd2}, it follows easily from
\cite[Theorem A.7]{Opd1} that the residual points arise in generic
$\mc{Q}$-families. Let us state and prove this result precisely.
\begin{defn}
A real analytic function $r:\mc{Q}\to T$ is called a \emph{generic
residual point} of $\mc{R}$ if there exists an open, dense subset
$U\subset \mc{Q}$ such that the element $r(q)\in
\operatorname{Res}(\mc{R},q)$ for all $q\in U$. The set of generic
residual points of $\mc{R}$ is denoted by
$\operatorname{Res}(\mc{R})$.
\end{defn}
\index{ResR@Res$(\mc R)$, generic residual points of $\mc R$}
\index{Q@$\mc{Q}_{W_0r}^{reg} \subset \mc Q$,
subset of $W_0r$-regular parameters}
\begin{defn}\label{defn:reg}
Let $r\in\operatorname{Res}(\mc{R})$. We call $q\in\mc{Q}$ an
$r$-regular (or $W_0r$-regular) parameter if
$r(q)\in\operatorname{Res}(\mc{R},q)$. We denote by
$\mc{Q}_{W_0r}^{reg}\subset\mc{Q}$ the subset of $W_0r$-regular parameters.
\end{defn}
It is clear that $\mc{Q}_{W_0r}^{reg}\subset\mc{Q}$ is the complement of
a closed real analytic subset (for a more precise statement, see Theorem
\ref{thm:mgen}). This implies the following basic finiteness
result:
\begin{prop}\label{prop:fin}
The set $\operatorname{Res}(\mc{R})$ of generic residual points is
finite and $W_0$-invariant. This set is nonempty iff $\mc{R}$ is semisimple.
\end{prop}
\begin{proof}
Suppose that there exist infinitely many distinct generic residual
$\mc{Q}$-families $q\to r(q)$. Choose countably infinitely many
distinct residual families $r_1,r_2,\dots$.
By Baire's theorem we can choose $q\in\mc{Q}$ such that
the $r_i(q)$ are all residual and mutually distinct. But by Theorem
\ref{thm:finite} there are at most finitely many residual points
for $q$, a contradiction. Hence the set $\operatorname{Res}(\mc R )$ is
finite. The $W_0$-invariance is clear. By Theorem
\ref{thm:ord} it follows that this set is empty if the rank of $R_0$
is not equal to the rank of $X$.

For the converse, assume that $\mc{R}$ is semisimple and consider
the onedimensional representation $N_w \mapsto q(w)$ of $\mc H
(\mc R,q)$. By Theorem \ref{thm:cas} this is a discrete series
representation whenever $|q(s)| < 1$ for all $s \in S$. So by
Theorem \ref{thm:rescos} its $X$-character $r(q) \in T$ lies in
Res$(\mc R,q)$ for all such $q$. Since the corresponding subset of
$\mc Q$ is Zariski-dense and $r : \mc Q \to T$ is algebraic, it is
a generic residual point.
\end{proof}

\subsubsection{Results on the reduction to simple root systems}
\label{subsub:red}

The following result is useful to reduce statements about
generic residual points to the case of simple root data.
\begin{lem}\label{lem:genprod}
\begin{enumerate}
\item[(i)] Let $\mc{R},\mc{R}^\prime$ be as in Lemma \ref{lem:prod}(i).
The restriction map $r^\prime\to r=r^\prime|_{\mc{Q}\times X}$ is a
surjection $\operatorname{Res}(\mc{R}^\prime)\to\operatorname{Res}(\mc{R})$
with fibers of order $|X^\prime:X|$.
\item[(ii)] Let $\mc{R}$ be as in Lemma \ref{lem:prod}(ii).
Then we have a natural bijection
\begin{equation}
\textup{Res}(\mc{R})\stackrel{\approx}{\longrightarrow}
\textup{Res}(\mc{R}^{(1)})\times\dots\times
\textup{Res}(\mc{R}^{(m)})
\end{equation}
such that $r\to(r^{(1)},\dots,r^{(m)})$ iff
$r(q^{(1)},\dots,q^{(m)})=r^{(1)}(q^{(1)})\dots r^{(m)}(q^{(m)})$ with
$r^{(i)}(q^{(i)})\in T^{(i)}$ for all $i=1,\dots,m$ and
all $q=(q^{(1)},\dots,q^{(m)})\in\mc{Q}$.
\item[(iii)]
Let $\mc{R}$ be arbitrary semisimple and let
$\mc{Q}=\mc{Q}^{(1)}\times\dots\times \mc{Q}^{(m)}$ be the
decomposition of $\mc{Q}=\mc{Q}(\mc{R})$ as in Proposition
\ref{prop:prod}(i). Suppose that $\mc{Q}^\prime\subset\mc{Q}$ is
a connected closed subgroup of $\mc{Q}$ such that for each
$i=1,\dots,m$ the projection $\pi_i:\mc{Q}^\prime\to\mc{Q}^{(i)}$ is surjective.
Let $r^\prime:\mc{Q}^\prime\to T$ be real analytic with the property
that $r^\prime(q^\prime)\in\textup{Res}(\mc{R},q^\prime)$ for almost all
$q^\prime\in\mc{Q}^\prime$. Then there exists a unique
$r\in\textup{Res}(\mc{R})$ such that $r^\prime=r|_{\mc{Q}^\prime}$.
\end{enumerate}
\end{lem}
\begin{proof}
The first two assertions are clear so let us look at (iii). Let
\[
\tilde{r^\prime} : \mc{Q}^\prime \to
\text{Hom} (X^{max},\mh C^\times ) = T^{(1)} \times \dots \times T^{(m)}
\]
be a lifting of $r^\prime$. Choose homomorphisms
$\phi_i:\mc{Q}^{(i)}\to\mc{Q}^\prime$ such that
$\pi_i\circ\phi_i=\textup{id}_{\mc{Q}^{(i)}}$ for all $i$.
Lemma \ref{lem:prod} implies that
the map $\tilde{r}_i:\mc{Q}^{(i)}\to T^{(i)}$ defined by $\tilde{r}^{(i)}(q^{(i)}):=
(\tilde{r^\prime}(\phi_i(q^{(i)})))^{(i)}$ is a generic residual point for
$\mc{R}^{(i)}$. Let $\tilde{r}\in\textup{Res}(\mc{R})$ correspond
to $(\tilde{r}^{(1)},\dots,\tilde{r}^{(m)})$ (using the notation of (ii)).
Then (i) implies that $r=\tilde{r}|_{\mc{Q}\times X}$ meets the requirement.
If $r_1$ also meets the requirement let $\widetilde{r_1}$
be the unique lifting of $r_1$ to $\textup{Res}(\mc{R}^{max})$ such that
$\widetilde{r_1}|_{\mc{Q}^\prime}=\tilde{r^\prime}$. Then it is clear that for all $i$
we must have $\widetilde{r_1}^{(i)}=\tilde{r}^{(i)}$. The uniqueness follows.
\end{proof}
Recall the result of Lemma \ref{lem:mon}. We see that if $r=sc$ is a residual
point then $s\in T_u$ belongs to the finite set of points with the
property that $\mc{R}_s$ is semisimple. In particular, if $r:\mc{Q}\to T$
is a generic residual point then the unitary part $s$ of $r$ is independent
of $q\in\mc{Q}$ and $\mc{R}_s$ is semisimple.

\index{Ress@$\mr{Res}^s (\mc{R})$, generic residual points with unitary part $s$}
\begin{cor}\label{cor:genreds}
Suppose that $\mc{R}$ is semisimple and
$s\in T_u$ is such that $\mc{R}_s$ is semisimple. Let
$\phi_s:\mc{Q}(\mc{R})\to\mc{Q}(\mc{R}_s)$ denote the homomorphism $q\to q_s$.
\begin{enumerate}
\item[(i)] Let $\textup{Res}^s(\mc{R})$ denote the set of generic residual
points $r$ with unitary part $s$. There exists a natural bijection
\begin{align*}
\Phi_s:\textup{Res}^s(\mc{R})&\to\textup{Res}^s(\mc{R}_s)\\
r&\to r\circ\phi_s
\end{align*}
\item[(ii)]
Using the notation of Definition \ref{defn:iso}, we have a natural bijection
\begin{align*}
\Phi_{W_0s}^{W_0}:W_0\backslash\textup{Res}^{W_0s}(\mc{R})&\to\Gamma_s
\backslash \big( W(R_{s,1})\backslash\textup{Res}^s(\mc{R}_s) \big) \\
W_0r&\to \Gamma_sW(R_{s,1})(r\circ\phi_s)
\end{align*}
Here $W_0\backslash\textup{Res}^{W_0s}(\mc{R})$ denotes the set of $W_0$-orbits of
generic residual points whose unitary part is $W_0s$.
\end{enumerate}
\end{cor}
\begin{proof}
The image $\mc{Q}^\prime=\phi(\mc{Q})\subset\mc{Q}_s$ satisfies
the condition as in Lemma \ref{lem:genprod}(iii). The result (i) then follows
from Lemma \ref{lem:mon} and Lemma \ref{lem:genprod}(iii). The assertion (ii)
follows from (i) and Definition \ref{defn:iso}.
\end{proof}
The previous Corollary reduces the classification of the set
$\textup{Res}(\mc{R})$ to the classification of those elements
$r\in\textup{Res}(\mc{R})$ which are of the form $r=sc$ where $s$
is $W_0$-invariant. In this case we further reduce to the level
of the degenerate Hecke algebra:

\index{K@$\mc K$, space of $W_0$-invariant real functions on $R_1$}
\index{Resl@$\mr{Res}^{lin}(R_1)$, generic linear residual points for $R_1$}
\begin{defn}\label{defn:linres}
Let $R_1\subset V^*$ be a semisimple,
reduced root system and let $\mc{K}$ be the space of $W_0$-invariant
real valued functions on $R_1$. We denote by $\textup{Res}^{lin}(R_1)$
the set of linear maps $\xi:\mc{K}\to V$ such that for almost all $k$
the point $\xi(k)\in V$ is $(R_1,k)$-residual in the sense of \cite{HO}, i.e.
\begin{equation}
|\{\alpha\in R_1\mid \alpha(\xi(k))=k_\alpha\}|=
|\{\alpha\in R_1\mid \alpha(\xi(k))=0\}|+\dim(V)
\end{equation}
We refer to this set as the set of generic linear residual
points associated to the root system $R_1$.
\end{defn}
\begin{prop}\label{lem:redrat}
Let $\mc{R}$ be semisimple and let $s\in T_u$ be $W_0$-invariant.
Let $\mc{K}$ be the vector space of real $W_0$-invariant functions
on $R_1$, and given $q\in\mc{Q}$ let $k_s\in \mc{K}$ be the
$W_0$-invariant function on $R_1$ associated to $q$ by the formulas
of equation (\ref{eq:k}). Let $r=sc$ be a generic $\mc{R}$-residual point.
\begin{enumerate}
\item[(i)] There exists a unique generic linear residual point
$\xi\in\textup{Res}^{lin}(R_1)$ such that
$\alpha(c(q))=e^{\alpha(\xi(k_s))}$ for all $\alpha\in R_1$ and all $q\in\mc{Q}$
(where $k_s$ is related to $q$ as above). We express this relation between $r$
and $\xi$ by $r=s\exp(\xi)$.
\item[(ii)] This yields a $W_0$-equivariant bijection between
$\textup{Res}^{W_0s}(\mc{R})$ and $\textup{Res}^{lin}(R_1)$.
\item[(iii)] For all $q\in\mc{Q}$ we have: $r(q)$ is $(\mc{R},q)$-residual iff
$\xi(k_s)$ is $(R_1,k_s)$-residual (in the sense of \cite{HO}).
\item[(iv)] The generic linear residual points $\xi$ are rational
in the sense that the $\alpha(\xi(k))$ is
a rational linear combination of the values $k_{\beta}$ for all
$\alpha\in R_1$.
\end{enumerate}
\end{prop}
\begin{proof} The existence of $\xi$ is a special case of
\cite[Theorem A.7]{Opd1}, and the uniqueness is clear since
$R_1$ spans $V^*$. Similarly (ii) follows from \cite[Theorem A.7]{Opd1}.
The rationality of $\xi$ follows
from the fact that the set of roots contributing to the pole
order of $c$ at $r$ span a sublattice of $X$ of finite index
as a consequence of Theorem \ref{thm:ord}.
\end{proof}
The following reduction to simple root systems follows easily from the definitions:
\begin{prop}\label{prop:linprod}
Let $R_1=R_{1,1},\dots,R_{N,1}$ be the decomposition in simple root systems. Then
$\mc{K}=\mc{K}_1\times\dots\times\mc{K}_N$
and $\textup{Res}^{lin}(R_1)=\textup{Res}^{lin}(R_{1,1})\times\dots\times
\textup{Res}^{lin}(R_{N,1})$.
\end{prop}

\subsubsection{Rationality results for generic residual points}

Nothing that follows in this paper depends on the results in this paragraph
in any essential way, but these results simplify notations and reveal certain
basic facts. The proofs in this paragraph depend on the classification of
positive generic residual points for irreducible root systems.
\begin{thm}\label{thm:L}
Let $\mc{R}$ be a semisimple root datum,
and let $r:\mc{Q}\to T$ be a generic residual point of the form
$r=sc$. For all $x\in X$
the expression $x(c)\in\Lambda$ is a monomial in the generators $v(s)^{\pm 1}$
with $s\in S$. Here $v(s)$ is viewed as a function on $\mc{Q}$ by
$(v(s))(q):=q(v(s))$. In other words,
$r$ is (the restriction to $\mc{Q}$ of) a $\mc{Q}_c$-valued point of $T$.
\end{thm}
\begin{proof}
Using Lemma \ref{lem:genprod} if suffices to show this for
$\mc{R}=(X,R_0,Y,R_0^\vee,F_0)$ with $R_0$ irreducible and $X$ the weight
lattice of $R_1$. By Corollary \ref{cor:genreds} it suffices to consider the case
where $s\in T$ is $W_0$-invariant.
Then we are in the situation of Proposition \ref{lem:redrat}.
In terms of the rational linear function $\xi:\mc{K}\to V$ of Proposition \ref{lem:redrat}
the assertion amounts to showing that $2\xi$ is integral, i.e. $x(2\xi)$ is
an integral linear combination of the functions $k_\beta$ (with $\beta\in R_1$)
for all integral weights $x$. We call $\xi$ a generic residual point for $R_1$
(in the sense of \cite{HO}).

If $R_1=A_n$ it is easy to see that $2\xi$ is integral (even even for even $n$).

If $R_1=B_n$ it suffices to remark that
the integrality of $\xi$ with respect to the root lattice
follows from the description of the residual points as in
\cite[Section 4]{HO} (also see Section \ref{sec:genlinres}).

The generic residual points for $R_1$ of type $C_n$
are in bijection to those of type $B_n$ as follows. Let $k_1$ denote
the parameter of the $C_n$ roots of the form $\pm e_i\pm e_j$ and $k_2$
the parameter of the $C_n$ roots $\pm 2e_i$. If $\xi^\prime$ is a generic
$B_n$-residual point then $\xi(k_1,k_2)=\xi^\prime(k_1,k_2/2)$ is
a generic $C_n$ residual point. This sets up a bijective correspondence
between the generic residual points of $B_n$ and of $C_n$.
Hence if $\xi$ is residual for $C_n$ then $2\xi$ is integral with respect
to the root lattice of $B_n$, which is equal to the the weight lattice
of $C_n$.

If $R_1$ is of type $D_n$ or $E_n$ we use that $\xi$ is
integral with respect to the root lattice \cite[Corollary B2]{Opd1}.
In order to check the integrality of $2\xi$ with respect to the
weight lattice one needs to check in addition the integrality
of $x(2\xi)$ with respect to the minuscule fundamental weights.
This is an easy verification using the explicit descriptions
of the Bala-Carter diagrams of the distinguished parabolic
subgroups in \cite[Section 5.9]{C} (see \cite[Appendix B]{Opd1}
for the explanation of the relation between residual points and
Bala-Carter diagrams for the simply laced types) and the table
1 of \cite[Chapter III, Section 13.2]{Hum1} expressing the
fundamental weights in the simple roots. For $R_1=D_n$ there
are $3$ minuscule fundamental weights to check, and for
$R_1=E_{6}$ there are $2$ of these.
For $E_7$ and $E_8$ the integrality of $\xi$ with respect to the
root lattice suffices since the index of the root lattice in the
weight lattice is at most $2$.

For $F_4$ and $G_2$ the root lattice
is equal to the weight lattice. In these cases the result follows simply
from the tables in \cite[Section 4]{HO}.
\end{proof}
We introduce the following notation
\begin{defn}
Let $r=sc\in\operatorname{Res}(\mc{R})$. Recall that for all $\alpha\in R_1$
we have $\alpha(r)=\alpha(s)\alpha(c)$ with
$\alpha(s)$ a root of $1$ and $\alpha(c)$ a monomial in the
variables $v_{\beta^\vee}^{\pm 1}$ (with $\beta^\vee\in R_{\mathrm{nr}}$)
as described above. We define
\begin{align*}
R_{r,1}^{p,-}&=
\{\alpha\in R_0\cap R_1\mid
v_{\alpha^\vee}^2\alpha(r)-1=0\}\cup
\{2\beta\in R_1\backslash R_0\mid
v_{\beta^\vee/2}v_{\beta^\vee}^2\beta(r)-1=0\}
\\
R_{r,1}^{p,+}&=\{2\beta\in R_1\backslash R_0\mid
v_{\beta^\vee/2}\beta(r)+1=0\}\\
R_{r,1}^{z}&=\{\alpha\in R_1\mid
\alpha(r)-1=0\}\\
\end{align*}
and we define an element $m_{W_0r}\in K(\Lambda)$ in the quotient field
$K(\Lambda)$ of $\Lambda$ by
(with $w_0\in W_0$ the longest element):
\index{Kx@$K (\Lambda)$, quotient field of $\Lambda$}
\index{mW0@$m_{W_0 r}$, function $\mc Q \to \mh C$ that indicates $r$-regularity}
\index{wz@$w_0$, longest element of $W_0$}
\begin{equation}\label{eq:m}
m_{W_0r}:=\frac{v(w_0)^{-2}\prod_{\alpha\in R_1\backslash R_{r,1}^z}(\alpha(r)^{-1}-1)}
{\prod_{\alpha\in R_1\backslash
R_{r,1}^{p,+}}(v_{\alpha^\vee}^{-1}{\alpha(r)^{-1/2}}+1)
\prod_{\alpha\in R_1\backslash
R_{r,1}^{p,-}}(v_{\alpha^\vee}^{-1}v_{{2\alpha}^\vee}^{-2}\alpha(r)^{-1/2}-1)}
\end{equation}
As before, if $\alpha\in R_0\cap R_1$ then $v_{2\alpha^\vee}=1$ and the corresponding
terms in the denominator simplify to $(v_{\alpha^\vee}^{-2}{\alpha(r)}^{-1}-1)$.
Therefore the expression is rational in the values $\alpha(r)$ with $\alpha\in R_0$.
Observe that the above definition of $m_{W_0r}$ is indeed independent of the
choice of $r$ in the $W_0$-orbit $W_0r$, justifying the notation $m_{W_0r}$.
\end{defn}
\index{Q@$\mc{Q}^{sing}_{W_0 r}$, complement in $\mc Q$ of $\mc{Q}^{reg}_{W_0 r}$}
\begin{thm}\label{thm:mgen}
Let $r$ be a generic residual point. We view the generators $v(s)$ of
$\Lambda$ as functions on $\mc{Q}$ via $v(s)(q):=q(v(s))$ as before.
The function $m_{W_0r}$ is real analytic on $\mc{Q}$.
The set of $r$-regular points
$\mc{Q}^{\operatorname{reg}}_{W_0r}:=\{q\in\mc{Q}\mid
r(q)\in\operatorname{Res}(\mc{R},q)\}$ of $\mc{Q}$ is the complement of the zero
locus $\mc{Q}^{sing}_{W_0 r}$ of $m_{W_0r}$ in $\mc{Q}$. In particular this set
is the complement of a union of finitely many (rational) hyperplanes in $\mc{Q}$.
\end{thm}
\begin{proof}
Since $r(q)$ is generically residual it is clear that
$|R_{r,1}^{p,+}\cup
R_{r,1}^{p,-}|-|R_{r,1}^z|=\operatorname{rank}(X)$. By Theorem
\ref{thm:ord} it is therefore clear that for all $q\in\mc{Q}$ the
number of factors that are zero at $q$ in the numerator of $m_{W_0r}$
has to be at least equal to the number of factors that are zero at
$q$ in the denominator. This implies that $m_{W_0r}$ is real analytic
on $\mc{Q}$, and that the zero locus of $m_{W_0r}$ in $\mc{Q}$ is
precisely the set of $q$ such that $r(q)$ is not residual.
\end{proof}
\index{Resq@$\mr{Res}_q (\mc R)$, generic residual points which are residual at $q \in \mc Q$}
\begin{defn}\label{def:genresq}
Let $q\in\mc{Q}$. We define $\operatorname{Res}_q(\mc{R})=
\{r\in \text{Res}(\mc{R})\mid r(q)\in\operatorname{Res}(\mc{R},q)\}$.
Thus $\operatorname{Res}_q(\mc{R})$ is the set of generic residual
points whose specialization at $q$ is residual.
\end{defn}
Let $r=sc\in\textup{Res}(\mc{R})$. By Lemma \ref{lem:mon}
the evaluations $x(s)$ with $x\in X$ are roots of unity.
Let $K\supset\mathbb{Q}$ be the Galois extension of $\mathbb{Q}$
generated by the values $x(s)$ with $x\in X$. Theorem
\ref{thm:L} implies that for all $x\in X$ we have
$x(\tilde{r})\in K[v(s)^{\pm 1}:s\in S]$,
the ring of Laurent polynomials in the variables $v(s)^{\pm 1}$
(with $s\in S$) with coefficients in $K$.
Let $\sigma\in\textup{Gal}(K/\mathbb{Q})$.
By the above there is a canonical action $r\to\sigma(r)$ of
$\textup{Gal}(K/\mathbb{Q})$
on $\operatorname{Res}(\mc{R})$ characterized by $x(\widetilde{\sigma(r)})
=\sigma(x(\tilde{r}))$ for all $x\in X$,
where $\sigma$ on the right hand side is acting on the coefficients
of $x(\tilde{r})\in\Lambda$ (these are indeed elements of $\Lambda$ with
algebraic coefficients, by Lemma \ref{lem:mon} and Theorem \ref{thm:L}).
\index{Kxz@$K (\Lambda_\mathbb{Z})$, quotient field of $\Lambda_\mathbb{Z}$}
\begin{prop}\label{prop:genfam}
Let $\mc{R}$ be a semisimple root datum.
\begin{enumerate}
\item[(i)] Let $r\in\textup{Res}(\mc{R})$ and $\sigma\in\textup{Gal}(K/\mathbb{Q})$.
Then $\sigma(r)|_{Q(R_0)}\in W_0r|_{Q(R_0)}$ where $Q(R_0)\subset X$ denotes the
root lattice of $R_0$.
\item[(ii)] For all $r\in\textup{Res}(\mc{R})$ we have
$m_{W_0r}\in K(\Lambda_\mathbb{Z})$, the quotient field of the subring
$\Lambda_\mathbb{Z}:=\mathbb{Z}[v([s])^{\pm 1}:[s]\in \tilde{S}]\subset\Lambda$
of $\Lambda$. \index{1lz@$\Lambda_{\mh Z} ,\, \mh Z$-algebra generated by $v(s)^{\pm 1}$}
\item[(iii)] In the situation of Lemma \ref{lem:genprod}(i)
we have $m_{W_0r}=m_{W_0r^\prime}$ and in the situation of Lemma
\ref{lem:genprod}(ii) we have
$m_{W_0r}(q)=m_{W_0^{(1)}r^{(1)}}(q^{(1)})\dots m_{W_0^{(k)}r^{(k)}}(q^{(k)})$.
\end{enumerate}
\end{prop}
\begin{proof}
The first assertion follows from the proof of \cite[Proposition 3.27]{Opd1}.
Then (ii) follows from (i) by the fact that $m_{W_0r}$ only depends
on the restriction of $r$ to $Q(R_0)$ and the fact that the assignment
$r\to m_{W_0r}$ is $W_0$-invariant.
The assertions of (iii) are trivial.
\end{proof}

\subsubsection{Deformation of residual points in the parameter $q$}

The following result is very important: it says that all residual
points are obtained from specialization of the generic residual points.
\index{evq@$\mr{ev}_q$, evaluation of functions on $\mc Q$ at $q$}
\begin{prop}\label{prop:resdef}
Let $\mc{R}$ be a semisimple based root datum.
The evaluation map
$\operatorname{ev}_q:\operatorname{Res}_q(\mc{R})\to
\operatorname{Res}(\mc{R},q)$ given by
$\operatorname{ev}_q(r)=r(q)$ is surjective for all $q\in\mc{Q}$.
\end{prop}
\begin{proof}
We prove this fact by induction to the rank of $R_0$. If the rank of
$R_0$ is one the assertion can be verified by an easy inspection. Assume
that the result holds for all maximal proper parabolic subsystems
of $R_0$. Let $r_0\in T$ be a residual point for the parameter
value $q_0\in \mc{Q}$. By Proposition \ref{prop:flag} we know that
that there exists a residual line $L_0=r_{L,0}T^L$ where
$r_{L,0}\in T_L$ is a residual point for a proper maximal
parabolic subsystem $R_L\subset R_0$ with the property that
$r_0\in L_0$. By the induction hypothesis, $L_0=L(q_0)$ for a
generic family of residual lines $L(q)=r_L(q)T^L$ (in other words,
the $\mc{R}_L$-residual point $r_{L,0}$ is the specialization
$r_{L,0}=r_L(q_0)$ at $q_0$ of a generic $\mc{R}_L$ residual point
$r_L$). By Theorem \ref{thm:ord} and Definition \ref{def:rescos}
it follows easily that for each fixed $q\in\mc{Q}$ such that
$r_L(q)$ is residual the rational function $\eta^L$ (see
(\ref{eq:dec})) on $L(q)$ has
poles of order at most one on $L(q)$, and $x\in L(q)$ is
$(\mc{R},q)$-residual if and only if $x$ is a pole of
$\eta^L(\cdot,q)$. In particular $r_0$ is a simple pole of
$\eta^L(\cdot,q_0)$. Considering the form of the factors in the
denominator of $\eta^L$ this implies easily that $r_0$ is the
specialization at $q=q_0$ of least one $\mc{Q}$-family of the form
$q\to r(q)\in L(q)$ such that $r(q)$ is residual for all $q$ in an
open neighborhood of $q_0$. Hence
$r\in\operatorname{Res}_q(\mc{R})$ and
$\operatorname{ev}_{q_0}(r)=r(q_0)=r_0$ as desired.
\end{proof}

\begin{defn}\label{defn:gen}
Let $\mc{R}$ be a semisimple root datum and let
$r\in\operatorname{Res}(\mc{R})$. We call
$q\in\mc{Q}_{W_0r}^{\operatorname{reg}}$ an $r$-\emph{generic} (or
$W_0r$-generic) parameter if for all
$r^\prime\in\operatorname{Res}(\mc{R})$ the equality $W_0r^\prime(q)=W_0r(q)$
implies that $r^\prime\in W_0r$.
The set of $r$-generic parameters
is denoted by $\mc{Q}_{W_0r}^{\operatorname{gen}}$. We define the set
$\mc{Q}^{\operatorname{gen}}$ of \emph{generic} parameters by
$\mc{Q}^{\operatorname{gen}}=
\cap_{r\in\operatorname{Res}(\mc{R})}\mc{Q}_{W_0r}^{\operatorname{gen}}$.
\end{defn}
\index{Q@$\mc{Q}_{W_0 r}^{\operatorname{gen}}$, set of $r$-generic parameters in
$\mc{Q}^{\operatorname{reg}}$} \index{Q@$\mc{Q}^{\operatorname{gen}}$,
parameters that are generic $\forall r \in \text{Res}(\mc R)$}
\begin{prop}\label{prop:gencone}
Let $\mc{R}$ be a semisimple
root datum. For all $r\in\operatorname{Res}(\mc{R})$ the set
$\mc{Q}_{W_0r}^{\operatorname{gen}}$ is the complement of a finite
collection of rational hyperplanes in $\mc{Q}$.
\end{prop}
\begin{proof} This follows easily from Corollary
\ref{cor:genreds} and Proposition \ref{lem:redrat}.
\end{proof}
The proof of the following important Proposition depends on the
classification of residual points.
\begin{prop}\label{prop:closed}
Recall that the central support of the set of tempered irreducible
characters of $\mc{H}(\mc{R},q)$ is given by the union
$S(q)=\cup_{L}L^{\operatorname{temp}}$ (union over the set
$(\mc{R},q)$-residual cosets $L\subset T$) (see Theorem
\ref{thm:rescos}). Let
$S_i(q)=\cup_{L}L^{\operatorname{temp}}\subset S(q)$ denote the
subset of $S(q)$ where the union is taken only over the residual
cosets of dimension at least $i$. The sets
$\cup_{q\in\mc{Q}}(q,S_i(q))\subset \mc{Q}\times T$ are closed for
all $i$.
\end{prop}
\begin{proof}
In view of Definition \ref{def:rescos} it is clear that it
suffices to show that if $r\in\operatorname{Res}(\mc{R})$ and
$q_0\in\mc{Q}_{W_0r}^{\operatorname{sing}}$, then there exists a
residual coset $L$ such that $r(q_0)\in L^{\operatorname{temp}}$.
By \cite[Theorem A.7]{Opd1} this reduces to the statement that if
$c$ is a positive generic residual point, then $c(q_0)$ coincides with
the center of a positive residual coset. Since the collection of
centers of \emph{positive} residual cosets does not depend on the
choice of the lattice $X$ we may replace $X$ by $X^{max}$ (as in
Proposition \ref{prop:prod}). Since $\mc{R}^{max}$ is a direct sum
of irreducible summands this shows that it suffices to prove the
statement for a root datum $\mc{R}$ with $R_0$ irreducible.

In the case where $R_0$ is simply laced this follows from the
remark that $\mc{Q}_{W_0r}^{\operatorname{sing}}=\{q_0=1\}$ for all
$r\in\operatorname{Res}(\mc{R})$. By Lemma \ref{lem:mon} we have
$r(1)=e$, which is the center of $T^{\operatorname{temp}}=T_u$. If
$R_0$ is of type $B_n$ or $C_n$, then this is \cite[Proposition
4.15]{Slooten}. For type $G_2$ and $F_4$ it can be read off from
the tables \cite[Table 4.10, Table 4.15]{HO}.
\end{proof}

\section{Continuous families of discrete series}

In this section we show that every discrete series character
of $\mc{H}=\mc{H}(\mc{R},q)$ is the specialization of a unique
maximal ``continuous parameter family'' of discrete series characters.
Using this fact and our results on $EP_{\mc{H}}$ the discrete
series can be parametrized explicitly for all irreducible root
data $\mc{R}$ which are not simply laced. An important ingredient
is the fact that the central characters of the irreducible
discrete series characters are precisely the $W_0$-orbits of
residual points.

Another main result in this section states that the formal
degree of a continuous family of irreducible discrete series
characters is a rational function on $\mc{Q}$ with rational
coefficients. This function has a product expansion in terms of
the central character of the family, and an alternating sum
expansion in terms of the branching multiplicities of the discrete
series representation to finite dimensional Hecke subalgebras.

\subsection{Parameter deformation of the discrete series}

In this subsection we show that each irreducible discrete series
character is a specialization in the parameter $q$ of a unique
continuous $\mc{Q}$-family of irreducible discrete series
characters.

It is useful to remark that such deformations are well understood
for \emph{scaling deformations} of the parameters along half
lines. What we are about to discuss in this subsection is what
happens for general deformations. Therefore this yields no extra
information whatsoever for the simply laced cases. On the other
hand, for the non-simply laced root systems the method turns out
to be sufficient in most cases to distinguish the irreducible
discrete series characters with the same central character form
each other, and parametrize them by continuous $\mc{Q}$-families of
discrete series characters.
\index{Pr0@$\mc P (r_0)$, generic residual points $r$ with $r (q_0) \in W_0 r_0$}
\index{1dW0@$\Delta_{W_0 t} (\mc R,q_0 )$, irreducible discrete series characters
with central character $W_0 t$}
\begin{defn}\label{def:resmult}
Let $\mc{R}$ be a semisimple root datum, $q_0\in\mc{Q}$, and let
$r_0\in T$ be a $(\mc{R},q_0)$-residual point. We denote by
$\mc{P}(r_0)= \{W_0r\in W_0\backslash\operatorname{Res}(\mc{R})
\mid W_0r(q_0)=W_0r_0\}$ the finite set of $W_0$-orbits of generic
residual points which coalesce at $W_0r_0$ for the parameter value $q=q_0$.

For $t \in T$ let $\Delta_{W_0 t} (\mc R,q_0 ) \subset \Delta (\mc R ,q_0 )$ be the
collection of irreducible discrete series characters with central character $W_0 t$.
\end{defn}

\begin{lem}\label{lem:est}
Let $r_0=s_0c_0$ be a $(\mc{R},q_0)$-residual point, and let
$0<\epsilon<1/3$. \index{zz@$z$, positive element in $\mh{C}[T]^{W_0}$}
There exists an open neighborhood $U\subset\mc{Q}$ of $q_0$ and a
Hermitian element $z\in\mathbb{C}[T]^{W_0}$ such that
\begin{enumerate}
\item[(i)] $z$ is positive on $S(q)$ for all $q\in\mc{Q}$.
\item[(ii)] $z(t)<\epsilon$ for all $q\in U$ and
$t\in S(q)\backslash\{W_0r(q)\mid r\in\mc{P}(r_0)\}$.
\item[(iii)] There exists an $M\geq 1$ such that
$1-\epsilon<z(W_0r(q))<M$ for all $q\in U$ and $r\in\mc{P}(r_0)$.
\end{enumerate}
\end{lem}
\begin{proof}
According to \cite[Lemma 3.5]{Opd1} for any $\delta>0$ there exist
elements $a\in\mathbb{C}[T]^{W_0}$ such that $a(W_0r_0)=1$ and
such that the uniform norm of $a$ on a $(\mc R,q_0 )$-residual coset $S_{c}(q_0)$
is smaller than $\delta$ for all centers $c$ such that $W_0c\not=W_0c_0$. By
Theorem \ref{thm:nonnested} we know that $r_0$ is disjoint from
the union of the tempered residual cosets of dimension at least
$1$ (in particular, $c_0\not=e$). Hence we can multiply $a$ by
further factors in order to make sure that $a$ is equal to zero on
all tempered residual cosets contained in $S_{c_0}(q_0)$ other
than $r_0$. By taking $\delta$ small enough we can arrange that
the uniform norm of $a$ on all components of $S(q_0)$ other than
the points of $W_0r_0$ is smaller than $\epsilon$. Define
$z\in\mathbb{C}[T]^{W_0}$ by
$z(t):=a(t)\overline{a(\overline{t}^{-1})}$. Using Theorem
\ref{thm:real} we see that $z(r_0)=1$ and that $z$ is nonnegative
on $S(q)$ (for all $q\in\mc{Q}$). This proves (i).

Define two open subsets $V_+:=\{t\in T\mid |z(t)|>1-\epsilon\}$
and $V_-:=\{t\in T\mid |z(t)|<\epsilon\}$ of $T$. By Proposition
\ref{prop:closed} we see that for all $q\in\mc{Q}$ the support
$S(q)$ is the following union of compact subsets
\begin{equation}
S(q)=\bigcup_{P}\bigcup_{r\in\operatorname{Res}(\mc{R}_P)}W_0r(q)T^P_u
\end{equation}
Put $W_0r(q)T^P_u=S(P,r,q)$. By the above it is clear that
$S(P,r,q_0)\subset V_+$ iff $R_P=R_0$ and $W_0r\in\mc{P}(r_0)$. On
the other hand, $S(P,r,q_0)\subset V_-$ iff $R_P=R_0$ and
$W_0r\not\in\mc{P}(R_0)$ or if $R_P\not=R_0$. By the compactness
of the sets $T^P_u$ and the continuity of the generic residual
cosets $r\in\operatorname{Res}(\mc{R_P})$ (viewed as functions on
$\mc{Q}$) it is clear that there exists an open neighborhood $U$
of $q_0$ such that for all $q\in U$, and for all pairs $(P,r)$ we(in which case
there exist precisely two irreducible discrete series characters which have
the same generic central character (unless one of the parameters is $0$))
have: $S(P,r,q)\subset V_-$ iff $S(P,r,q_0)$ and $S(P,r,q)\in V_+$
iff $S(P,r,q_0)\in V_+$. Hence for all $q\in U$ we have
\begin{equation}
S(q)=(S(q)\cap V_+)\cup(S(q)\cap V_-)
\end{equation}
and $S(q)\cap V_+=\mc{P}(r_0)(q)$. From this we easily deduce (ii)
and (iii).
\end{proof}
\index{L2W@$L^2 (W)$, Hilbert space of functions on $W$}
\index{SW@$\mc S (W)$, space of rapidly decaying functions on $W$}
Let $L^2(W)$ denote the abstract Hilbert space with Hilbert basis
$(\tilde{N}_w)_{w\in W}$ indexed by the elements of $W$. We
identify $L^2(W)$ with the Hilbert completion
$L^2(\mc{H}(\mc{R},q))$ (for any fixed $q\in\mc{Q}$) by
identifying $\tilde{N}_w\in L^2(W)$ with the basis element
$N_w\in\mc{H}(\mc{R},q)$. In this way $L^2(W)$ comes equipped with
the structure of a module over the $C^*$-algebra completion of
the pre-$C^*$-algebra $\mc{H}(\mc{R},q)$. By abuse of notation we
will denote the basis elements $\tilde{N}_w$ of the module
$L^2(W)$ simply by $N_w$. Similarly we use the notation
$\mc{S}(W)$ for the abstract Fr\'echet space of functions on $W$
which are of rapid decay with respect to the norm function
$\mc{N}$ on $W$. For each fixed $q\in\mc{Q}$ we identify
$\mc{S}(W)$ with the Fr\'echet algebra completion
$\mc{S}(\mc{R},q)$ of $\mc{H}(\mc{R},q)$.

\index{zzz@$z_q ,\, z$ regarded as an element of $\mc H (\mc R ,q)$}
Given $q\in\mc{Q}$ and $z\in\mathbb{C}[T]^{W_0}$ let
$z_q\in\mc{H}(\mc{R},q)$ denote the element $z$ viewed as an
element of $\mc{H}(\mc{R},q)$ via the isomorphism defined by the
Bernstein basis of the center $\mc{Z}(q)$ of $\mc{H}(\mc{R},q)$
with $\mathbb{C}[T]^{W_0}$. The above Lemma implies that
$z_q\in\mc{H}(\mc{R},q)$ is a positive central element such
that if $q\in U$ its spectrum on $L^2(\mc{H}(\mc{R},q))$ is
contained in $[0,\epsilon)\cup (1-\epsilon,M]$.

\begin{thm}\label{thm:def}
Let $U$, $M>0$ and $\epsilon>0$ be as in the previous Lemma. Let
$e_q:=p_{>1-\epsilon}(z_q)\in\mc{S}(\mc{R},q)$ denote the element
of $\mc{S}(\mc{R},q)$ obtained by holomorphic calculus applied to
$z_q\in\mc{H}(\mc{R},q)$ with respect to a function
$p_{>1-\epsilon}$ on the spectrum that is equal to $0$ in an open
neighborhood of $[0,\epsilon]$ and is equal to $1$ on an open
neighborhood of $[1-\epsilon,M]$.
\begin{enumerate}
\item[(i)] For all $q\in U$, $e_q\in\mc{S}(\mc{R},q)$ is a self-adjoint,
central idempotent.
\item[(ii)] For all $q\in U$
we have an orthogonal decomposition
\index{eq@$e_q$, idempotent in $\mc S (\mc R,q)$ associated to $z_q$}
\index{eq@$e_{\delta (q),q}$, idempotent in $\mc S (\mc R ,q)$, smaller than $e_q$}
\begin{equation}
e_q=
\sum_{W_0r\in\mc{P}(r_0)}\sum_{\delta(q)\in\Delta_{W_0r(q)}(\mc{R},q)}e_{\delta(q),q}
\end{equation}
where $e_{\delta(q),q}$ is the primitive central idempotent of
$\mc{S}(\mc{R},q)$ corresponding to the irreducible discrete
series character $\delta(q)\in\Delta_{W_0r(q)}(\mc{R},q)$
(the set of irreducible discrete series characters of
$\mc{H}(\mc{R},q_0)$ with central character $W_0r_0$).
\item[(iii)] For all $q\in U$ the two-sided ideal
$\mc{I}_q:=e_q\mc{S}(\mc{R},q)\subset\mc{S}(\mc{R},q)$ is a finite
dimensional, semisimple, involutive subalgebra of
$\mc{S}(\mc{R},q)$.
\item[(iv)] The family $q\to e_q\in\mc{S}(\mc{R},q)\simeq\mc{S}(W)$ is
continuous with respect to the parameter $q\in U$.
\item[(v)] The dimension $\operatorname{dim}_{\mathbb{C}}(\mc{I}_q)$ is
independent of $q\in U$.
\item[(vi)] The isomorphism class of $\mc{I}_q$ viewed as a (finite dimensional)
$C^*$-algebra is independent of $q\in U$.
\end{enumerate}
\end{thm}
\begin{proof}
By the previous Lemma it is clear that $p_{>1-\epsilon}$ is
holomorphic on the spectrum of $z_q$, hence we may apply
holomorphic functional calculus. Hence (i) follows from the fact
that $\mc{S}$ is closed for holomorphic functional calculus,
see Theorem \ref{thm:qcont}, and the basic properties of the holomorphic
functional calculus. The assertion (ii) follows from the previous
Lemma and the definition of the idempotent $e_q$. The finite
dimensionality  of $\mc{I}_q$ follows simply from (ii). Clearly
$\mc{I}_q$ is an involutive algebra because $e_q$ is central and
self-adjoint. Thus the trace $\tau$ and the anti-involution $*$
give rise to a positive definite Hermitian inner product on
$\mc{I}_q$ with the property $(ab,c)=(b,a^*c)$. Hence $\mc{I}_q$
is a semisimple subalgebra, proving (iii). It is easy to see that
$U\ni q\to z_q\in\mc{S}(W)$ is a continuous family (by expressing
$z$ in the $N_w$ basis of $\mc{H}(\mc{R},q)$). Hence (iv) follows
from the continuity of the holomorphic functional calculus, see
Theorem \ref{thm:qcont}. For (v) we first remark that it is
clear that for all $q\in U$ the projection
$\lambda(e_q)\in\mathfrak{B}(L^2(\mc{H}(\mc{R},q)))$ (where
$\lambda$ denotes the left regular representation) is of finite
rank (since only finitely many central characters support the
image of $e_q$ by construction). On the other hand it is clear
from Theorem \ref{thm:qcont} and \cite[Proposition 5.6]{Sol}
that this family of projections is norm continuous in
$\mathfrak{B}(L^2(\mc{H}(\mc{R},q)))$, implying in particular that
the rank is constant in the family. Finally observe that
$\mc{I}_q=\lambda(e_q)(L^2(\mc{H}(\mc{R},q)))$.
In order to prove (vi) we use
the notion of approximate matrix units in a $C^*$-algebra
\cite[Definition 2.2]{BKR}.
Let $m^{(i)}_{j,k}(q_0)$ be a basis
of matrix units of $\mc{I}_{q_0}$. Given an element $q\in U$ we define
$\tilde{m}^{(i)}_{j,k}(q)=e_q\cdot m^{(i)}_{j,k}(q_0)$, where in
the right hand side we view $m^{(i)}_{j,k}(q_0)$ as an element of
$\mc{S}(\mc{R},q)$ via the canonical isomorphism
$\mc{S}(W)\simeq\mc{S}(\mc{R},q)$. Let $\epsilon^\prime>0$. By
(iv), (v) and \cite[Proposition 5.6]{Sol} we obtain that there
exists an open neighborhood $q_0\in U_{\epsilon^\prime}\subset U$
of $q_0$ such that for all $q\in U_{\epsilon^\prime}$ the elements
$\tilde{m}^{(i)}_{j,k}(q)$ form a basis of
${\epsilon^\prime}$-approximate matrix units of $\mc{I}_q$. This
means that for all $i,j,k,l,m,n$ and for all $q\in
U_{\epsilon^\prime}$ we have
\begin{equation}
\Vert \tilde{m}^{(i)}_{j,k}(q)\tilde{m}^{(l)}_{m,n}(q)-
\delta_{i,l}\delta_{k,m}\tilde{m}^{(i)}_{j,n}(q)
\Vert<{\epsilon^\prime}
\end{equation}
and
\begin{equation}
\Vert\tilde{m}^{(i)}_{j,k}-(\tilde{m}^{(i)}_{k,j})^*\Vert<{\epsilon^\prime}
\end{equation}
(where the norm refers to the $C^*$-algebra norm).
Now \cite[Lemma 2.3]{BKR} implies that for
${\epsilon^\prime}>0$ sufficiently small there exists a basis of
matrix units $m^{(i)}_{j,k}(q)$ of $\mc{I}_q$ with the property
that for all $i,j,k$:
\begin{equation}
\Vert \tilde{m}^{(i)}_{j,k}(q)-m^{(i)}_{j,k}(q)\Vert<{\epsilon^\prime}
\end{equation}
In particular it follows that $\mc{I}_q$ for $q\in
U_{\epsilon^\prime}$ is isomorphic to $\mc{I}_{q_0}$ as a finite
dimensional $C^*$-algebra. Using a suitable open covering of $U$
this result extends easily to $q\in U$, proving $(vi)$.
\end{proof}

\begin{thm}\label{thm:cont}
Keep the notations as in Theorem \ref{thm:def}.
Let $r_0\in\operatorname{Res}(\mc{R},q_0)$.
\begin{enumerate}
\item[(i)] There exists an open neighborhood $U$ of $q_0$ such that
for each $\delta_0\in\Delta_{W_0r_0}(\mc{R},q_0)$
there exists a
unique family of primitive central idempotents $U\ni q\to
e_{\delta(q),q}\in \mc{S}(\mc{R},q)= \mc{S}(W)$ with the following
properties:
\begin{enumerate}
\item[(a)] $\delta(q_0)=\delta_0$.
\item[(b)] The function $U\ni q\to\lambda(e_{\delta(q),q},q)
\in\mathcal{B}(L^2(W))$ is continuous.
\item[(c)] For all $q\in U$, the value $e_{\delta(q),q}\in\mc{I}_q$
of this function is a primitive central idempotent.
\item[(d)] The degree of the irreducible character $\delta(q)$ of
$\mc{I}_q$ afforded by $e_{\delta(q),q}$ is independent of $q$.
\item[(e)] For all $q\in U$ the set
$\{e_{\delta(q),q}\}_{\delta(q_0)\in\Delta_{W_0r_0}(\mc{R},q_0)}$ is
the complete set of mutually inequivalent primitive central
idempotents of $\mc{I}_q$.
\end{enumerate}

\item[(ii)] The continuous families
of primitive central idempotents $U\ni q\to e_{\delta(q),q}$ (with
$\delta(q_0)\in\Delta_{W_0r_0}(\mc{R},q_0)$) define, for
all $q\in U$, a canonical bijection $\delta(q_0)\to\delta(q)$
between the set $\Delta_{W_0r_0}(\mc{R},q_0)$ and the union
\begin{equation}
\bigcup_{W_0r\in\mc{P}(r_0)}\Delta_{W_0r(q)}(\mc{R},q)
\end{equation}
\end{enumerate}
\end{thm}
\begin{proof}
Using the notations of the previous Theorem, we define for all
$q\in U_{\epsilon^\prime}$ and for all $i$:
\begin{equation}
e^{(i)}(q):=\sum_j m^{(i)}_{j,j}(q)
\end{equation}
This is a primitive central idempotent in $\mc{I}_q$ which is
independent of the choices of the matrix units $m^{(i)}_{j,k}(q)$.
Indeed, another choice of the matrix units would lead to a central
primitive idempotent norm close to $e^{(i)}(q)$. This implies
unitary equivalence in the $C^*$-algebra $\mc{I}_q$ of these
idempotents, but since these idempotents are also central, unitary
equivalence means actual equality. It follows from this argument
that the family of central primitive idempotents
$U_{\epsilon^\prime}\ni q\to e^{(i)}(q)$ is continuous at $q_0$ in
the following sense: The family of bounded operators
$U_{\epsilon^\prime}\ni q\to\lambda(e^{(i)}(q),q)$ on
$L^2(\mc{H}(\mc{R},q)=L^2(W)$ is continuous at $q_0$. Using the
independence of the central primitive idempotents for the choice
of the matrix units we may repeat this arguments for any $q\in
U_{\epsilon^\prime}$ to prove that the families
$U_{\epsilon^\prime}\ni q\to e^{(i)}(q)$ are continuous on
$U_{\epsilon^\prime}$. If we put $U:=U_{\epsilon^\prime}$ it is
now straightforward to prove the listed properties of (a)-(e) for
the constructed continuous families $e^{(i)}$ of primitive
idempotents. Finally the uniqueness follows again from the above
rigidity argument for central primitive idempotents, in
combination with the continuity, proving (i).

In view of Theorem \ref{thm:def}(ii) this sets up, for each value
of $q\in U$, a bijection between the set of continuous (in the
above sense) families of primitive central idempotents $e^{(i)}$
and set of irreducible discrete series characters
$\delta(q)\in\Delta_{W_0r(q)}(\mc{R},q)$ where $W_0r$ runs over the
set $W_0r\in\mc{P}(r_0)$. This proves (ii).
\end{proof}
The above notion of continuity of a $q$-family of irreducible discrete
series characters is special for discrete series characters:
\begin{defn}\label{def:cont}
Let $q_0\in\mc{Q}$ and let $\delta_0\in\Delta(\mc{R},q_0)$.
For $q\in U$ (as above) we denote by $\delta(q)$ the equivalence
class of irreducible discrete series representations afforded by
$e_{\delta(q),q}$. For any open set $U\subset\mc{Q}$ we refer to such
a family $\delta:q\to\delta(q)$ of equivalence classes of
representations afforded by a continuous family
of central primitive idempotents in $\mc{S}$
(in the above sense, thus in the operator norm of $\mc{B}(L^2(W))$)
as a ``continuous family of irreducible discrete series characters
on $U$''. We denote the set of such continuous families by $\Delta(\mc{R},U)$.
\end{defn}
There is also a weaker notion of continuity for a $q$-family of
characters which is applicable to more general characters:

\index{1dRU@$\Delta(\mc{R},U) \,, \Delta^{wk}(\mc{R},U)$, (weakly) continuous
families of irreducible discrete series characters on $U \subset \mc Q$}
\begin{defn}\label{defn:wkcont}
Let $U\ni q\to {\pi(q)}$ be a family of equivalence classes of
irreducible representations $\pi(q)$ of $\mc{Q}(\mc{R},q)$.
We say that $q\to\pi(q)$ is a \emph{weakly continuous} family of irreducible
characters of $\mc{H}(\mc{R})$ if $U\ni q\to\chi_{\pi(q)}(N_w)$ is a continuous
function for all $w\in W$.

We denote by $\Delta^{wk}(\mc{R},U)$ be the set of
weakly continuous families $U\ni q\to\delta(q)$ of irreducible
discrete series characters (i.e. weakly continuous families $q\ni U\to\delta(q)$
such that for all $q\in U$ we have $\chi_{\delta(q)}\in\Delta(\mc{R},q)$).
\end{defn}

Continuity of a family of discrete series characters implies weak continuity:
\begin{prop}\label{prop:cont}
Let $U\subset\mc{Q}$ and let $\delta\in\Delta(\mc{R},U)$. Then the
family $q\to\delta(q)$ is also weakly continuous.
\end{prop}
\begin{proof}
Indeed, by the Plancherel formula for $\mc{H}(\mc{R},q)$ we have
\begin{equation}\label{eq:fdim2}
\tau(e_{\delta(q),q})=\operatorname{deg}(\delta(q))\mu_{Pl}(\delta(q))
\end{equation}
and hence this function is positive, and continuous by Theorem
\ref{thm:cont}(i)(b). Hence the basic formula

\begin{equation}\label{eq:char}
\chi_{\delta(q)}(N_w)=\operatorname{deg}(\delta(q))
\frac{\tau(e_{\delta(q),q}N_w)}{\tau(e_{\delta(q),q})}
\end{equation}
combined with Theorem \ref{thm:cont}(i)(b),(d) implies the desired continuity.
\end{proof}

\index{ccd@$cc (\delta )$, generic central character map $\mc Q \supset U \to W_0 \backslash T$}
\begin{prop}\label{prop:cc}
Let $\delta\in\Delta^{wk}(\mc{R},U)$. We define the generic
central character map $cc(\delta,\cdot):U\to W_0\backslash T$ by
$cc(\delta,q)=cc(\delta(q))$. Then $cc(\delta)$ is continuous and
for all $q\in U$ we have
$cc(\delta,q)\in\operatorname{Res}(\mc{R},q)$.
\end{prop}
\begin{proof} This is a trivial consequence of Theorem
\ref{thm:rescos} and Proposition \ref{prop:cont}.
\end{proof}
In fact it is true that $cc(\delta)\in W_0\backslash\textup{Res}(\mc{R})$, but
this is not obvious at this point. This result will be shown in
Theorem \ref{thm:gencc}.

Actually weak continuity and continuity
are equivalent for families of discrete series characters. We have:
\index{1dR@$\Delta (\mc R) \,, \Delta^{wk}(\mc{R})$, sheaf on $\mc Q$ of (weakly)
continuous irreducible discrete series characters}
\index{1dN@$\Delta_{\mh N} (\mc R) \,, \Delta_{\mh N}^{wk}(\mc{R})$, sheaf on
$\mc Q$ of (weakly) continuous discrete series characters}
\index{1dNU@$\Delta_{\mh N} (\mc R ,U)$, continuous families of discrete series characters
on $U \subset \mc Q$}
\begin{thm}\label{thm:wkcnt}
Consider the sheaves $\Delta(\mc{R})$ and $\Delta^{wk}(\mc{R})$ on
$\mc{Q}$ defined by the presheaves $U\to\Delta(\mc{R},U)$ and
$U\to\Delta^{wk}(\mc{R},U)$, respectively.
\begin{enumerate}
\item[(i)] The natural sheaf map $\Delta(\mc{R})\to\Delta^{wk}(\mc{R})$ is
an isomorphism.
\item[(ii)] Let $\Delta_\mathbb{N}(\mc{R})$ denote the sheaf of
nonnegative integral linear combinations of $\Delta(\mc{R})$, and
let $\Delta_\mathbb{N}^{wk}(\mc{R})$ denote the sheaf of weakly
continuous families of (not necessarily irreducible) discrete
series characters. The natural map
$\Delta_\mathbb{N}(\mc{R})\to\Delta_\mathbb{N}^{wk}(\mc{R})$ is an
isomorphism.
\end{enumerate}
\end{thm}
\begin{proof} It is clear that all presheaves involved
are sheaves of sets.

Let us prove (i). Given $\delta\in\Delta^{wk}(\mc{R},U)$ we need
to show that $\delta$ is continuous in the strong sense. Let
$q_0\in U$, and let $W_0r_0$ be the central character of
$\delta(q_0)$. By Theorem \ref{thm:cont}(ii) there exists a
neighborhood $V\subset\mc{Q}$ of $q_0$ such that for any
$\sigma\in\Delta_{W_0r_0}(\mc{R},q_0)$ there exists
$\tilde\sigma\in\Delta(\mc{R},V)$ such that
$\sigma=\tilde\sigma_{q_0}:=\operatorname{ev}_{q_0}(\tilde\sigma)$
(the evaluation of the strongly continuous family $\tilde\sigma$
at $q_0\in V$). Moreover
Theorem \ref{thm:cont}(ii) asserts that for all $q\in V$ the
irreducible discrete series characters $\tilde\sigma_q$ (with
$\sigma\in\Delta(\mc{R},q_0)$) are mutually distinct and range
over the set of all irreducible discrete series characters of
$\mc{H}(\mc{R},q)$ whose central character is of the form
$W_0r(q)$ for some generic $W_0r\in\mc{P}(r_0)$. Now consider
$\delta\in\Delta^{wk}(\mc{R},U)$. By Proposition \ref{prop:cc} it
is clear that for all $q\in V$ the central character
$cc(\delta(q))$ is of the form $W_0r^\prime(q)$ for some
$W_0r^\prime\in\mc{P}(r_0)$. The linear independence of
irreducible characters, the finiteness of
$\Delta_{W_0r_0}(\mc{R},q_0)$ and Proposition \ref{prop:cont}
imply that there exists a finite set $A\subset W$ and a
neighborhood $V^\prime\ni q_0$ such that for all fixed $q\in
V^\prime$ the finite set of vectors
$\Sigma(q):=\{\xi_\sigma^A(q)\in\mathbb{C}^A\mid
\sigma\in\Delta_{W_0r_0}(\mc{R},q_0)\}$ with
$\xi_\sigma^A(q):=(\chi_{\tilde\sigma_q}(N_w))_{w\in A}$ is linearly
independent. In particular the irreducible characters
$\tilde\sigma_q$ are separated by the vector $\xi_\sigma^A(q)$ of
their values on $N_w$ with $w\in A$. Obviously the maps
$\xi_\sigma^A:U\to\mathbb{C}^A$ are continuous. By the weak
continuity of $\delta$ it follows similarly that the map
$\xi_\delta^A:U\to\mathbb{C}^A$ is continuous and by the above, for
all $q\in V$ we have $\xi_\delta^A(q)\in\Sigma(q)$. This implies
that there exists a unique $\sigma\in\Delta_{W_0r_0}(\mc{R},q_0)$
such that $\delta|_{V^\prime}=\tilde\sigma|_{V^\prime}$, proving
that $\delta$ is strongly continuous at $q_0$. Since $q_0\in U$
was arbitrary the result follows.

Let us now prove (ii). Let
$\delta\in\Delta_\mathbb{N}^{wk}(\mc{R},U)$. We need to show that
$\delta$ is continuous in a strong sense. Let $q_0\in U$, and let
$W_0r_i$ (where $i=1,\dots,k$) be the set of central characters of
the irreducible constituents of $\delta(q_0)$. We have
$\delta|_{U^{gen}}=\sum_{W_0r} \delta_{W_0r}|_{U^{gen}}$ (where $W_0r$
runs over the set $W_0\backslash \operatorname{Res}(\mc{R})$ of
orbits of generic residual points) where
$U^{gen}:=\mc{Q}^{gen}\cap U$ and where $U^{gen}\ni
q\to\delta_{W_0r}(q)$ is a weakly continuous family of discrete
series characters such that for all $q\in U^{gen}$,
$cc(\delta_{W_0r}(q))=W_0r(q)$. Recall that $\mc{Q}^{gen}$ is the
complement of finitely many rational hyperplanes in $\mc{Q}$.

We claim that for every connected component $U^\prime\subset
U^{gen}$ which contains $q_0$ in its boundary we have
$\delta_{W_0r}|_{U^\prime}\not=0$ only if
$W_0r\in\cup_i\mc{P}(r_i)$. Indeed, there exists a $z\in\mc{Z}$
such that $z(W_0r_i)=0$ for $i=1,\dots,k$ but with
$z(W_0r(q_0))=1$ for all orbits of generic residual points $W_0r$
such that $W_0r(q_0)\not\in\{W_0r_1,\dots,W_0r_k\}$. Observe that
for all $r\in\operatorname{Res}(\mc{R})$ the value
$\operatorname{deg}(\delta_{W_0r}|_{U^\prime})\in\mathbb{Z}_+$ is
independent of $q\in U^\prime$ since the family
$\delta_{W_0r}|_{U^\prime}$ is weakly continuous. By the weak
continuity of $\delta$ on $U$ we see that $U\ni q\to\chi_q:=
\chi_{\delta(q)}(z)$ must be continuous at $q_0$; however, by
definition of $z$ it follows on the one hand that $\chi_{q_0}=0$,
while on the other hand the limit for $q\to q_0$ from $U^\prime$
yields
$\sum_{W_0r\not\in\cup_i\mc{P}(r_i)}\operatorname{deg}(\delta_{W_0r}|_{U^\prime})$.
The claim follows.

We now prove in a similar fashion to the proof in (i) that if
$W_0r\in\cup_i\mc{P}(r_i)$ and if $U^\prime\subset U^{gen}$ is a
connected component which contains $q_0$ in its boundary then
$\delta_{W_0r}|_{U^\prime}$ is strongly continuous and in fact
extends uniquely to a neighborhood $U''$ of $q_0$ in a strongly
continuous sense. This finishes the proof.
\end{proof}
\begin{rem}\label{rem:ident} We identify the sheaves $\Delta(\mc{R})$,
$\Delta^{wk}(\mc{R})$, $\Delta_\mathbb{N}(\mc{R})$ and
$\Delta_\mathbb{N}^{wk}(\mc{R})$ on $\mc{Q}$ with their \'etale spaces.
These sheaves are Hausdorff spaces. As sets we have
\begin{equation}
\Delta(\mc{R})=\coprod_{q\in\mc{Q}}\Delta(\mc{R},q)
\end{equation}
\end{rem}
\begin{proof}
By Theorem \ref{thm:wkcnt} it suffices to show this for $\Delta(\mc{R})$.
In this case the result follows simply from Theorem \ref{thm:cont}(ii).
\end{proof}
\begin{prop} A continuous family of irreducible discrete series
characters $U\ni q\to \delta(q)$ is compatible with the scaling
maps $\tilde\sigma_\epsilon$ (with $\epsilon>0$) of
\cite[Theorem 1.7]{OpdSol} in the sense that
$\tilde\sigma_\epsilon(\delta(q))=\delta(q^\epsilon)$.
\end{prop}
\begin{proof} We may assume that $U\subset \mc{Q}$ is an open ball
centered around of $q_0\in\mc{Q}$ such that
$\operatorname{ev}_{q_0}:\Delta(\mc{R},U)\to\Delta(\mc{R},q_0)$ is
an isomorphism. Let $\mc{L}\subset\mc{Q}$ be the half line
generated by $q_0$. Let $\delta\in\Delta(\mc{R},q_0)$ and let
$\tilde\delta\in\Delta(\mc{R},U)$ be such that
$\operatorname{ev}_{q_0}(\tilde\delta)=\delta$. Consider the
continuous family $\delta^{(1)}$ defined by restricting the
section $\tilde\delta$ to $\mc{L}\cap U$, and the continuous
family $\delta^{(2)}$ defined by scaling $\mc{L}\cap U\ni
q^\epsilon_0\to \tilde\sigma_{\epsilon}(\delta)$. It follows
from the analyticity (\cite[Theorem 1.7]{OpdSol}, property 1)) that
$\delta^{(2)}\in\Delta^{wk}(\mc{R},\mc{L}\cap U)$. The result
$\delta^{(1)}=\delta^{(2)}$ follows from Theorem \ref{thm:wkcnt}.
\end{proof}
\begin{cor}\label{cor:inv}
We can extend any continuous family of irreducible discrete series
characters $\delta\in\Delta(\mc{R},U)$ in a unique way to
$\tilde\delta\in\Delta(\mc{R},\tilde{U})$ where
$\tilde{U}=\cup_{\epsilon>0}U^\epsilon$ is the open cone in
$\mc{Q}$ generated by $U$.
\end{cor}
\begin{proof}
Let $\mc{L}\subset\tilde U$ be a half line. By the above
Proposition and the properties of the scaling maps (namely, for
$\epsilon>0$ these maps induce bijections of the sets of
equivalence classes of irreducible discrete series characters) we
see that the restriction $\Delta_\mc{L}(\mc{R})$ of
$\Delta(\mc{R})$ to $\mc{L}$ is a constant sheaf. The result
follows easily from this remark.
\end{proof}

\section{The generic formal degree}

Let $U\subset\mc{Q}$ be a
connected open cone, and let $\delta\in\Delta^{wk}(\mc{R},U)$. In
this subsection we prove the rationality of the formal degree
$U\ni q\to\mu_{Pl}(\delta(q))$, i.e. we prove that this function is
the restriction to $U$ of a rational function of the root parameters
$q_{\alpha^\vee}$ with rational coefficients, i.e. of
an element of $K(\Lambda_\mathbb{Z})$.
We refer to this rational function
as the \emph{generic formal degree} of the family $\delta$.
We combine the rationality of the generic formal degree with
the product formula \cite[Theorem 4.10]{Opd3} for the formal
degree of $\delta(q)$ valid for $q$ varying in a half line in
$\mc{Q}$. We then obtain the factorization of the generic
formal degree as element of $K(\Lambda)$.
\subsection{Rationality of the generic formal degree}
Let $\mc{R}$ be a semisimple root datum and let $\Omega\subset W$
be the finite subgroup of length zero elements. If $f$ is a facet of the
fundamental alcove $C$, then we denote by $W_f \subset W^a$ the finite
subgroup generated by the simple affine reflections $s \in S$ that fix $f$,
and by $\Omega_f\subset\Omega$ the (setwise) stabilizer of $f$ in $\Omega$.
Let $\langle f\rangle\subset E$ be
the affine subspace spanned by $f$, and let $E/\langle f\rangle$
be the linear space formed by cosets $e-\langle f\rangle$ (with
$e\in E$) of the linear subspace associated to $\langle f\rangle$.
Let $\epsilon_f$ be the determinant character of the linear action
of $\Omega_f$ on $E/\langle f\rangle$. The involutive subalgebras
$\mc{H}(\mc{R},f,q) = \mc{H}(W_f,q)\rtimes\Omega_f\subset
\mc{H}(\mc{R},q)$ are finite dimensional (since $W_f \rtimes \Omega_f$
is finite) and semisimple by \cite[Lemma 1.4]{OpdSol}.
\index{1zf@$\Omega_f$, stabilizer in $\Omega$ of a facet $f$}
\index{HRf@$\mc{H}(\mc{R},f,q)$, Hecke algebra associated to a facet $f$}
\index{Wf@$W_f$ stabilizer in $W^a$ of a facet $f$}

Let $F$ be an algebraic closure of $K(\Lambda_\mathbb{Z})$ and let
$I\subset F$ be the integral closure of $\Lambda_\mathbb{Z}$.
We choose an extension to $I$ of the homomorphism $q : \Lambda_\mathbb{Z}\to\mathbb{C}$.
Consider the semisimple $F$-algebra
$\mc H_F (\mc R ,f) = \mc H_F (W_f ) \rtimes \Omega_f$.
Let $\chi^F$ be the character of a simple $\mc
H_F (\mc R ,f)$-module. According to a well known argument of
Steinberg (see e.g. \cite[Proposition 10.11.4]{C})
one has $\chi^F(N_w)\in I$ for all $w\in W_f \rtimes \Omega_f$.
Furthermore the $\mathbb{C}$-linear map $\chi:\mc{H}(\mc R ,f,q)\to\mh{C}$ defined by
$\chi(N_w)=q(\chi^F(N_w))$
is the character of a simple
$\mc H (\mc R ,f,q)$-module, and this provides a bijection between
$\widehat{\mc H_F (\mc R ,f)}$ and $\widehat{\mc{H}(\mc{R},f,q)}$
(cf. loc. cit.).
\begin{lem}\label{lem:rat}
Let $d_{\chi}\in F$ be the formal degree of $\chi^F$ with
respect to the trace form $\tau$ restricted to the algebra
$\mc{H}_F(\mc{R},f)$. Then $d_\chi\in K(\Lambda_\mathbb{Z})$ and $d_\chi$ is
regular on $\mc{Q}$.
\end{lem}
\begin{proof}
For all $q\in\mc{Q}$ the trace form $\tau$ of the algebra
$\mc{H}(\mc{R},f,q)$ has a nonzero discriminant, proving that
$\mc{H}(\mc{R},f,q)$ (and a fortiori $\mc{H}_F(\mc{R},f)$)
is a symmetric (and thus semisimple) algebra.
Let $(V,\sigma^F)$ be a matrix representation
of $\mc{H}_F(\mc{R},f)$ whose character equals $\chi^F$. We write
$d_\sigma:=d_\chi$ for its formal degree (with respect to $\tau$).

The orthogonality of characters of a
symmetric algebra implies that
\begin{equation}\label{eq:schur}
d_\sigma=1/S_\sigma
\end{equation}
where $S_\sigma$ is the Schur element of $\sigma^F$, given by
\begin{equation}
\operatorname{dim}_F(V)S_\sigma=\sum_{w\times\omega\in W_f\rtimes\Omega_f}
\chi^F(N_{w\times\omega})\chi^F(N_{(w\times\omega)^{-1}})
\end{equation}
By a well known result (see e.g. the argument in
\cite[Proposition 4.6]{Geck}, which applies to our situation as
well as one easily checks) one also has the following formula for
the Schur element:
\begin{equation}
\operatorname{dim}_F(V)S_\sigma(q)\operatorname{Id_V}=
\sum_{w\times\omega\in W_f\rtimes\Omega_f}
\sigma^F(N_{w\times\omega}N_{(w\times\omega)^{-1}})
\end{equation}
But clearly (loc. cit.)
\begin{equation}
\sum_{w\times\omega\in W_f\rtimes\Omega_f}
\sigma^F(N_{w\times\omega}N_{(w\times\omega)^{-1}})=
|\Omega_f|\sum_{w\in W_f}\sigma^F(N_{w}N_{w^{-1}})
\end{equation}
This last equality implies that if $(\sigma_1^F,V_1)$ is any
simple submodule of the restriction of $\sigma^F$ to
$\mc{H}_F(W_f)$ then
\begin{equation}
\operatorname{dim}_F(V)S_\sigma=
|\Omega_f|\operatorname{dim}_F(V_1)S_{\sigma_1}
\end{equation}
The right hand side of this equation is known to be
in $K(\Lambda_\mathbb{Z})$ (see \cite[Section 13.5]{C}), proving the desired result.
The last assertion follows from the well known fact that the Schur
element of $S_\sigma$ is nonzero at $q$ iff $\sigma^F$ corresponds to
a projective irreducible representation of the specialized algebra
$\mc{H}(W_f,q)$. Since $\mc{H}(W_f,q)$ is semisimple for $q\in\mc{Q}$
this holds true for all $\sigma$.
\end{proof}
Let $\delta\in\Delta^{wk}(\mc{R},U)$. Following \cite{ScSt},
\cite{Re} we define for $q\in U$ the index function
$f_{\delta,q}\in\mc{H}(\mc{R},q)$ by
\begin{equation}\label{eq:index}
f_{\delta,q}=\sum_{f}(-1)^{\operatorname{dim}(f)}
\sum_{\sigma\in \widehat{\mc{H}(\mc{R},f,q)}}
\operatorname{deg}(\sigma)^{-1}
[\delta_q|_{\mc{H}(\mc{R},f,q)}\otimes\epsilon_f:\sigma] e_\sigma
\end{equation}
where $f$ runs over a complete set of representatives of the
$\Omega$-orbits of faces of the fundamental alcove $C$, and where
$e_\sigma\in\mc{H}(\mc{R},f,q)$ denotes the primitive central
idempotent in the finite dimensional complex semisimple algebra
$\mc{H}(\mc{R},f,q)$ affording $\sigma$. The importance of the element
$f_{\delta,q}\in\mc{H}(\mc{R},q)$ is that it links character theory
with the elliptic pairing.
Indeed, following \cite{ScSt}, \cite{Re} one shows,
using the Euler-Poincar\'e principle and Frobenius reciprocity,
that for all representations $\pi$ of finite length of
$\mc{H}(\mc{R},q)$ one has (see \cite[Proposition 3.6]{OpdSol}):
\begin{equation}\label{eq:pseudo}
\chi_\pi(f_{\delta,q})=EP_\mc{H}(\delta(q),\pi)
\end{equation}
\begin{defn}\label{lem:rattriv}
The multiplicities
$[\delta(q)|_{\mc{H}(\mc{R},f,q)}\otimes\epsilon_f:\sigma]$ are
independent of $q\in U_{\delta}$  by Proposition \ref{prop:cont}.
We denote these multiplicities by
$[\delta_f\otimes\epsilon_f:\sigma]\in\mathbb{Z}_{\geq 0}$.
\end{defn}
\begin{thm}\label{thm:rat}
Let $U\subset\mc{Q}$ be a connected open cone and let
$\delta\in\Delta^{wk}(\mc{R},U)$. We have the
following index formula for the formal
degree $\mu_{Pl}(\{\delta(q))\})$ (with $q\in U$):
\begin{equation}\label{eq:eulplanmeas}
\mu_{Pl}(\{\delta(q)\})=\tau(f_{\delta,q})=
\sum_{f}(-1)^{\operatorname{dim}(f)}
\sum_{\sigma\in \widehat{\mc{H}(\mc{R},f,q)}}
[\delta_f\otimes\epsilon_f:\sigma] d_\sigma(q)
\end{equation}
Here $f$ runs over a complete set of representatives of the
$\Omega$-orbits of faces of $C$, and $d_\sigma(q)$ denotes
the formal degree of $\sigma$ in the finite dimensional Hilbert
algebra $\mc{H}(\mc{R},f,q)$ (as in Lemma \ref{lem:rat}).
\end{thm}
\begin{proof}
We apply the Plancherel formula (\ref{eq:plan}) to $f_{\delta,q}$.
In view of (\ref{eq:pseudo}) and Corollary \ref{cor:ON} we see that
$\mu_{Pl}(\{\delta(q)\})=\tau(f_{\delta,q})$. Now use (\ref{eq:index})
and Definition \ref{lem:rattriv}.
\end{proof}
\begin{cor}\label{cor:rat}
Let $U\subset\mc{Q}$ be a connected open cone and let
$\delta\in\Delta^{wk}(\mc{R},U)$. The formal degree $U\ni q\to
\mu_{Pl}(\{\delta(q)\})$ is the restriction to $U$ of a rational
function in the parameters $q_{\alpha^\vee}$ (with
$\alpha\in R_{\operatorname{nr}}$) with rational coefficients
(or in other words, an element of $K(\Lambda_\mathbb{Z})$ in the
notation of Proposition \ref{prop:genfam}(ii)).
This rational function is regular on $\mc{Q}$ and positive on $U$.
\end{cor}
\begin{proof}
Consider the index formula as given in Theorem
\ref{thm:rat}. The result now follows from Lemma \ref{lem:rat}
(the positivity on $U$ is obvious).
\end{proof}

\subsection{Factorization of the generic formal degree}

\begin{lem}\label{lem:respt}
Let $\delta\in\Delta^{wk}(\mc{R},U)$ be a weakly continuous family
of irreducible discrete series characters on a convex open
cone $U\subset\mc{Q}$. The map $cc(\delta(\cdot)):U\to
W_0\backslash T$ is continuous. There exist finitely many mutually
disjoint, nonempty connected open subcones $U_i\subset U$ such
that $\cup_i U_i\subset U$ is dense, and such that for each $i$
there exists an orbit $W_0r_i$ of generic residual cosets such
that $\overline{U_i}\cap U\subset\mc{Q}^{gen}_{W_0r_i}$ and
$cc(\delta)|_{U_i}=W_0r_i|_{U_i}$. In particular $cc(\delta)$ is
continuous and piecewise analytic.
\end{lem}
\begin{proof} The continuity of $cc(\delta)$ on $U$ follows
from Proposition \ref{prop:cc}. Let $U_i$ run over the
finite set of connected components of $U\cap\mc{Q}^{gen}$. Then
the restriction of $cc(\delta)$ to $U_i$ must coincide with the
restriction of a unique orbit of generic residual points, by the
continuity of $cc(\delta)$ and the definition of $\mc{Q}^{gen}$.
By continuity, for all $q\in\overline{U}_i\cap U$ the orbit
$W_0r_i(q)$ carries discrete series representations. Hence
$r_i(q)$ is residual, or equivalently $q\in\mc{Q}_{W_0r_i}^{reg}$.
\end{proof}
\begin{thm}\label{thm:ratfam} Let $\delta\in\Delta^{wk}(\mc{R},U)$
be a weakly continuous family of irreducible discrete series
characters on a convex open cone $U\subset\mc{Q}$. Let $r$ be
a generic residual point such that there exists a nonempty
connected open subcone $U_i\subset U$ such that
$cc(\delta)|{U_i}=W_0r|{U_i}$ (see Lemma \ref{lem:respt}). There
exists a constant $d\in\mathbb{Q}^\times$ (depending on $\delta$ and $W_0r$)
such that we have the
following equality in $K(\Lambda_\mathbb{Z})$:
\begin{equation}\label{eq:fdim}
\mu_{Pl}(\{\delta\})=d m_{W_0r}
\end{equation}
Here $m_{W_0r}\in K(\Lambda_\mathbb{Z})$ (see Proposition \ref{prop:genfam}(ii))
is the function defined in (\ref{eq:m}).
\end{thm}
\begin{proof}
We fix $f_s\in\mathbb{R}$ and we denote the corresponding half
line in $\mc{Q}$ by $\mc{L}\subset \mc{Q}$ (see Remark
\ref{rem:base}). Notice that either $\mc{L}\cap U_i=\emptyset$ or
$\mc{L}\subset U_i$; assume that $\mc{L}$ is such that we are in
the latter situation. By \cite[Corollary 3.32, Theorem 5.6]{Opd1}
we have
\begin{equation}\label{eq:fdhelp}
\mu_{Pl}(\{\delta(q)\})=d(q) m_{W_0r}(q)
\end{equation}
for all $q\in U_i$, where $d(q)\in\mathbb{R}^\times$ has
the property that for all $\epsilon\in\mathbb{R}_+$
\begin{equation}\label{eq:scale}
d(q^\epsilon)=d(q)
\end{equation}
where $q^\epsilon$ is defined by $q^\epsilon(s)=(q(s))^\epsilon$
for all affine simple reflections $s$. By Theorem \ref{thm:mgen},
Corollary \ref{cor:rat} and (\ref{eq:fdhelp}) we see that
$d$ is itself a rational function which is regular on $U_i$.

Recall that we view $\mb{q}>1$ as coordinate on $\mc{L}$. The
expressions $\alpha(r(q))=\alpha(s)\alpha(c(q))$ and
$q_{\alpha^\vee}$ (with $\alpha\in R_{\mathrm{nr}}$ and
$q\in\mc{L}$) are thus viewed as functions of $\mb{q}>1$. By the
form of the right hand side of (\ref{eq:fdhelp}) as given in
(\ref{eq:eulplanmeas}), and in view of Corollary \ref{cor:rat}
we see
that there exists a unique real number $f$ such that
\begin{equation}
\lim_{\bf{q}\to\infty}{\bf{q}}^f\mu_{Pl}(\{\delta\})({\bf{q}})=
a_{\mc{L}}\in\mathbb{Q}^\times
\end{equation}
On the other hand, by $(\ref{eq:scale})$ the rational function
$d$ has a constant value, $d_{\mc{L}}$ say, on
$\mc{L}$. Hence (\ref{eq:fdhelp}) implies, in view of
(\ref{eq:m}) and Proposition \ref{prop:genfam}(ii), that
$d_{\mc{L}}b_\mc{L}=a_{\mc{L}}$ where
\begin{equation}
\lim_{\bf{q}\to\infty}{\bf{q}}^f
m_{W_0r}({\bf{q}})=b_{\mc{L}}\in\mathbb{Q}^\times
\end{equation}
Since $d(q)$ is continuous as a function of $q\in U_i$
this implies that $d_{\mc{L}}\in\mathbb{Q}$ is independent
of $\mc{L}\subset U_i$ and thus that $d(q)=d$ is
independent of $q\subset U_i$. Since $U_i$ is an open set, the
equality (\ref{eq:fdim}) of rational functions which we have now
proved on $U_i$ extends to $\mc{Q}$ (recall that both sides are
regular on $\mc{Q}$).
\end{proof}
\begin{cor}\label{cor:mreg}
Let $\delta\in\Delta^{wk}(\mc{R},U)$ be weakly continuous on a
convex open cone $U$. Let $W_0r_i,\,W_0r_j$ be orbits of generic residual
points associated with $\delta$ as in Lemma \ref{lem:respt}. There exists
a constant $d\in\mathbb{Q}^\times$ such that $m_{W_0r_i}=dm_{W_0r_j}$.
\end{cor}

\section{The generic central character map and the formal degrees}

The following result depends on the classification of residual points:
\begin{lem}\label{prop:sep}
Let $\mc{R}=(X,R_0,Y,R_0^\vee)$ be a simple root datum such that
$R_0$ is not simply laced, and let
$r,r^\prime\in\operatorname{Res}(\mc{R})$ be generic residual points
with equal unitary part~$s$, which is $W_0$-invariant.
If there exists a constant $d\in\mathbb{C}^\times$ such that
$m_{W_0r}=dm_{W_0r^\prime}$ then $W_0r=W_0r^\prime$.
\end{lem}
\begin{proof}
Using Lemma \ref{lem:genprod} and Proposition \ref{prop:genfam}(iii)
we reduce to the
case where $\mc{R}$ is irreducible and $X=P(R_1)$,
and $r,r^\prime$ are generic residual points
with equal $W_0$-invariant unitary part $s\in T_u$.
Let us write $r=sc$ and $r^\prime=sc^\prime$.
In the $C_n^{(1)}$-case we have $s=(1,\dots,1)$ or $s=(-1,\dots,-1)$.
We use Proposition \ref{lem:redrat}. In the first case we find that
$c,c^\prime$ extend to positive generic residual points
for the root datum $\mc{R}^\prime$ defined by $R_0^\prime=B_n$ and
$X^\prime=P(R_0)$, with the parameters $\tilde{q}$ defined by
$\tilde{q}_{e_i\pm e_i}=q_{e_i\pm e_j}$ and
$\tilde{q}_{2e_i}=q_{2e_i}^{1/2}q_{2e_i+1}^{1/2}$. In the second
case $c,c^\prime$ are positive generic residual points for $\mc{R}^\prime$
with the parameter $\tilde{q}$ defined by $\tilde{q}_{e_i\pm e_i}=q_{e_i\pm e_j}$
and $\tilde{q}_{2e_i}=q_{2e_i}^{-1/2}q_{2e_i+1}^{1/2}$.
In the first case we substitute $q_{2e_i}=q_{2e_i+1}$
and in the second case we substitute $q_{2e_i}=q_{2e_i+1}^{-1}$;
with this substitution we have in either case
\begin{equation}
m^{\mc{R}}_{W_0r}(q)=m^{\mc{R}^\prime}_{W_0c}(\tilde{q})\text{\ and\ }
m^{\mc{R}}_{W_0r^\prime}(q)
=m^{\mc{R}^\prime}_{W_0c^\prime}(\tilde{q})
\end{equation}
Therefore it suffices to prove the assertion for irreducible root data
$\mc{R}$ such that $R_0$ is non-simply laced and $X=P(R_0)$ where
$W_0r,\,W_0r^\prime$ are orbits of generic residual points with the same
$W_0$-invariant unitary part $s$. We may now replace $s$ by $1$ without
loss of generality.
Hence we may and will assume that $W_0r,\,W_0r^\prime$ are orbits of positive
residual points. We again use Proposition \ref{lem:redrat} to compare such points to
the classification in \cite[Section 4]{HO}.

In the cases $G_2$ and $F_4$ the $W_0$-orbit $W_0r$
of a generic positive residual points $W_0r$ is distinguished by the set
$\mc{Q}_{W_0r}^{\operatorname{reg}}$ as can be seen from Tables 2 and
4. Since this set is the complement of the zero set of $m_{W_0r}$
(by Theorem \ref{thm:mgen}) the desired conclusion follows.

Next consider the cases $B_n$ and $C_n$.
Let $f$ be a rational function in $q_1,q_2$ of the form
\begin{equation}
f(q)=q_1^{N_1}q_2^{N_2}\prod_{i}\prod_{j\geq
0}(q_1^iq_2^j-1)^{n_{i,j}}
\end{equation}
(with $n_{i,j}\in\mathbb{Z}$). Then
the exponents $n_{i,j}\in\mathbb{Z}$ are determined by $f$. Let
$q_1$ denote the parameter of the roots $\pm e_i\pm e_j$ and $q_2$
the parameter of $\alpha^\vee$ for $\alpha=e_i$ (if $R_0$ has type
$B_n$) or $2e_{i}$ (if $R_0$ has type $C_n$). The functions $m_{W_0r}$
are of all of the above form where the exponent of $q_2$ is
$0,\, 2$ or $4$. The $W_0$-orbits of generic positive
residual points are parametrized by partitions of $n$ (see
\cite[Section 4]{HO}, and \cite[Theorem A.7]{Opd3}). Let
$\lambda\vdash n$ and let $W_0r_\lambda$ be the corresponding
$W_0$-orbit of residual points. Let us use the notation
$m_{W_0r}=m_\lambda$ if $W_0r=W_0r_\lambda$. In the case $B_n$, the factors
of $m_\lambda$ of the form $(q_1^{2i}q_2^2-1)$ have multiplicity
$n_{2i,2}$ equal to twice the number of boxes $b\in\lambda$ such
that $c(b)=i$ (where $c(b)$ denotes the content of $b$). Hence
$m_\lambda$ determines for each $i$ the number of boxes in
$\lambda$ with content $i$. Clearly this determines $\lambda$. If
$R_0$ is of type $C_n$ we use the correspondence between $B_n$ and
$C_n$ positive generic residual points as explained in the proof
of Theorem \ref{thm:L}. It follows that the factors of $m_\lambda$ of type
$(q_1^{4i}q_2^2-1)$ have multiplicity $n_{4i,2}$ equal to twice the
number of boxes $b$ of $\lambda$ with $c(b)=i$, and again we
conclude that $\lambda$ is determined by $m_\lambda$.
\end{proof}
\begin{cor}\label{cor:gccinv} Let $\mc{R}$ be semisimple and let
$q_0\in\mc{Q}=\mc{Q}(\mc{R})$. Suppose that $\delta_0\in\Delta(\mc{R},q_0)$ and
that $cc(\delta_0)=W_0r_0$ for a $r_0\in\textup{Res}^s(\mc{R},q_0)$ with
$s\in T_u$ which is $W_0$-invariant.
Then there exists a unique orbit $W_0r\in W_0\backslash\textup{Res}(\mc{R})$ of
generic
residual points which has the following property: there exists
an open neighborhood $U\subset\mc{Q}$ of $q_0$ and a continuous family of discrete
series characters $U\ni q\to \delta(q)\in\Delta_{W_0r(q)}(\mc{R},q)$ such that
$cc(\delta(q))=W_0r(q)$ for all $q\in U$.
\end{cor}
\begin{proof}
The uniqueness of such an orbit $W_0r$ of generic residual points is clear from
the fact that a generic residual point is real analytic on $\mc{Q}$. Hence $W_0r$
is determined by its restriction to $U$.

For existence we first choose a lift $\tilde{r}_0\in\textup{Res}(\mc{R}^{max},q_0)$
of $r_0$ and a $\pi_0\in\Delta_{W_0\tilde{r}_0}(\mc{R},q_0)$ with the property
that $\delta_0$ is a component of the restriction of $\pi_0$ to $\mc{Q}(\mc{R},q_0)$.
According to
Theorem \ref{thm:cont} there exists an open neighborhood $U\subset\mc{Q}$
such that $\pi_0$ extends to a continuous family $\pi$ of irreducible
discrete series characters of $\mc{H}(\mc{R}^{max})$. It is obvious that
$\pi=\pi^{(1)}\otimes\dots\otimes\pi^{(m)}$ with $\pi^{(i)}$ a continuous family
of irreducible discrete series characters of $\mc{H}(\mc{R}^{(i)})$ defined on $U^{(i)}$
(where $\mc{R}^{(i)}$ with $i=1,\dots,m$ runs through the simple factor of
$\mc{R}^{max}$ as in Proposition \ref{prop:prod}).

For each $i$ there exists a generic residual point
$\tilde{r}^{(i)}\in\textup{Res}(\mc{R}^{(i)})$ such that
$cc(\pi^{(i)})=W(R_0^{(i)})\tilde{r}^{(i)}$ on $U^{(i)}$.
Indeed, if $\mc{R}^{(i)}$ is simply laced then this is trivial by
the scaling isomorphisms \cite[Theorem 1.7(1),(5)]{OpdSol}.
So let us assume that $\mc{R}^{(i)}$ is not simply laced.
Then the assertion follows from Lemma \ref{lem:respt},
Theorem \ref{thm:ratfam}, and Lemma \ref{prop:sep} applied to
\begin{equation}
\pi^{(i)}_0\in\Delta_{W(R_0^{(i)})\tilde{r}^{(i)}_0}(\mc{R}^{(i)},q_0^{(i)})
\end{equation}
Let
$r\in\textup{Res}(\mc{R})$ be the generic residual point that corresponds
to $(\tilde{r}^{(1)},\dots,\tilde{r}^{(m)})$ by restriction as in
Lemma \ref{lem:genprod}(i).

If we restrict the continuous family $\pi$ from $\mc{H}(\mc{R}^{max})$ to
$\mc{H}(\mc{R})$ we obtain a continuous family of discrete series characters,
i.e. a section $\pi|_{\mc{H}(\mc{R})}\in\Delta_\mathbb{N}(\mc{R},U)$. Observe that
all irreducible components of $\pi(q)|_{\mc{H}(\mc{R},q)}$ have the same
central character.
Using the linear independence of irreducible characters and
Theorem \ref{thm:cont}(ii) we see that $\pi|_{\mc{H}(\mc{R})}$
contains the continuous extension $\delta$ of $\delta_0$ to $U$ with
multiplicity at least $1$.
In particular we see that the composition of $cc(\pi):U\to W_0\backslash T^{max}$
with the natural projection $W_0\backslash T^{max}\to W_0\backslash T$
is the central character $cc(\delta)$ of the family $\delta$ on $U$.
We conclude that $cc(\delta)$ is given on $U$ by $W_0r|_U$,
where $r\in\textup{Res}(\mc{R})$ was constructed above.
This finishes the proof.
\end{proof}
Now we come to the main result of this section. It generalizes
Corollary \ref{cor:gccinv} to general irreducible discrete series
characters.\:
\begin{thm}\label{thm:gencc}
Let $\delta_0\in\Delta(\mc{R},q_0)$. Let $U\subset\mc{Q}$ be a (connected)
open neighborhood of $q_0$ such that there exists a $\delta\in\Delta(\mc{R},U)$
with $\delta(q_0)=\delta_0$ (see Theorem \ref{thm:cont}).
There exists a unique orbit
$W_0r\in W_0\backslash \textup{Res}_q(\mc{R})$ such that
$cc(\delta(\cdot))=W_0r|_U$.
\end{thm}
\begin{proof}
We first show that the notion of weak continuity of a family of
characters (see Definition \ref{defn:wkcont}) is to some extent compatible
with the reduction results Theorem \ref{thm:red1} and Corollary \ref{cor:dsred}.

Let $cc(\delta(q))=W_0t(q)$ where $U\ni q\to t(q)\in T$ is continuous. Write $s$ for the
unitary part of $t(q)$ (independent of $q$). Let $\psi_s:\mc{Q}\to\mc{Q}_s =
\mc Q (\mc{R}_s)$ be the homomorphism given by $q\to q_s$.

We denote by $\pi_0\in\Delta_\mathbb{N}(\mc{R}_s,\psi_s(q_0))$ the restriction of
the irreducible
discrete series module of $\mc{H}(\mc{R}_s,\psi_s(q_0))\rtimes \Gamma(t(q_0))$ to
$\mc{H}(\mc{R}_s,\psi_s(q_0))$. By Theorem \ref{thm:cont} and Theorem \ref{thm:wkcnt}
there exists a (connected) open neighborhood $U_s\subset\mc{Q}_s$ of $\psi_s(q_0)$
and a family
\begin{equation}
\pi\in\Delta_\mathbb{N}(\mc{R}_s,U_s)
\end{equation}
such that $\pi(\psi(q_0))=\pi_0$. We may
and will shrink $U$ in such a way that $\psi_s(U)\subset U_s$.

Let $N^s_w\in\mc{H}(\mc{R}_s,q_s)$ for $w\in W(\mc{R}_s)$ denote the standard basis
for the affine Hecke algebra $\mc{H}(\mc{R}_s,q_s)$.
Recall from Lusztig's construction (in the variation Theorem \ref{thm:red1})
that $\mc{H}(\mc{R}_s,q_s)$ is embedded as a subalgebra of the formal
completion $\overline{\mc{H}}(\mc{R},q)$ (as defined by (\ref{eq:Hcompl}))
via the map $N_w^s\to e_{t(q)} N_w$ where $w\in W(\mc{R}_s)$ and where
$e_{t(q)}\in\overline{\mc{H}}(\mc{R},q)$ denotes the idempotent as in
Theorem \ref{thm:red1}.

Let $\delta_t(q)$ be the irreducible discrete series representation of
$\overline{\mc{H}}(\mc{R}_s,q_s)\rtimes\Gamma(t(q))$ corresponding to $\delta(q)$
according to Theorem \ref{thm:red1}. This implies in particular that
\begin{equation}\label{eq:deltat}
\chi_{\delta_t(q)}(N^s_w)=\chi_{\delta(q)}(e_{t(q)}N_w)
\end{equation}
for all $w\in W(\mc{R}_s)$.

We claim that
\begin{equation}\label{eq:claimcont}
\chi_{\pi(q_s)}(N_w^s)=\chi_{\delta_t(q)}(N^s_w)
\end{equation}
for all $q\in U$ and $w\in W(\mc{R}_s)$. By Theorem \ref{thm:cont}
and Theorem \ref{thm:wkcnt} it suffices to show that for all $w\in W$
the right hand side of (\ref{eq:deltat}) is continuous as a function
of $q\in U$.

By the continuity of $U\ni q\to cc(\delta(q))$
it is easy to see that one can construct for each $N\in\mathbb{N}$
a continuous family $U\ni q\to a_{t,q}\in\mc{A}=\mathbb{C}[T]$
(i.e. a $q$-family of Laurent polynomials on $T$ whose
coefficients depend continuously on $q$) such that for all $q\in
U$ and $t^\prime\in W(R_{s,1})t(q)$: $a_{t,q}\in 1+m_{t^\prime}^N$ while for
all $t^\prime\in W_0t(q)\backslash W(R_{s,1})t(q)$ one has $a_{t,q}\in
m_{t^\prime}^N$. If $N$ is sufficiently large this implies easily
that for all $q\in U$ and for any $w\in W(\mc{R}_s)$ one has
\begin{equation}
\chi_{\delta(q)}(e_{t(q)}N_w)= \chi_{\delta(q)}(a_{t,q} N_w)
\end{equation}
which is indeed continuous in $q\in U$ as was required, thus proving
(\ref{eq:claimcont}).

According to Corollary \ref{cor:gccinv} we find that
$cc(\pi_\lambda)\in W(R_{s,1})\backslash \textup{Res}^s(\mc{R}_s)$
for any irreducible component $\pi_\lambda$ of $\pi$.
By relation (\ref{eq:claimcont}) and application of
Corollary \ref{cor:genreds} it follows that for any component
$\pi_\lambda$ of $\pi$ that
\begin{equation}
cc(\delta)=(\Phi_{W_0s}^{W_0})^{-1}(\Gamma_s(cc(\pi_\lambda)))
\end{equation}
This finishes the proof.
\end{proof}
In view of Theorem \ref{thm:L} this means that the central character
of $\delta\in\Delta(\mc{R},U)$ actually extends to a $\mc{Q}_c$-valued
point of $W_0\backslash T$.

\begin{defn}\label{defn:gcc}
(Generic central character for discrete series)

Let $q\in\mc{Q}$. Theorem \ref{thm:gencc} yields a map
$gcc_q:\Delta(\mc{R},q)\to W_0\backslash\operatorname{Res}_q(\mc{R})$ which
extends to a continuous map (in the sense of Remark \ref{rem:ident})
$gcc:\Delta(\mc{R})\to W_0\backslash\operatorname{Res}(\mc{R})$.
We call $gcc_q$ and $gcc$ the ``generic central character'' maps.
\end{defn}
\index{gcc@$gcc \,, gcc_q \,, GCC$, generic central character maps}
\index{OR@$\mc O (\mc R)$, subspace of $W_0 \backslash \text{Res}(\mc R) \times \mc Q$}
\begin{defn}\label{defn:resorb}
Consider the topological space $\mc{O}(\mc{R})$ given by
\begin{equation}
\mc{O}(\mc{R})=
\{(W_0r,q)\in W_0\backslash\operatorname{Res}(\mc{R})\times\mc{Q}
\mid q\in\mc{Q}_{W_0r}^{reg}\}
\end{equation}
The finite map $\pi_2:\mc{O}(\mc{R})\to\mc{Q}$ is a local
homeomorphism and the projection
\begin{equation}
\pi_1:\mc{O}(\mc{R})\to W_0\backslash\operatorname{Res}(\mc{R})
\end{equation}
on the first factor defines for all $q\in\mc{Q}$
a bijection between the fibre
$\mc{O}(\mc{R})_q$ of
$\pi_2$ at $q\in\mc{Q}$ and the set
$W_0\backslash\operatorname{Res}_q(\mc{R})$.
We define the following evaluation map:
\begin{align*}
\textup{ev}:\mc{O}(\mc{R})&\to W_0\backslash T\times\mc{Q}\\
(W_0r,q)&\to (W_0r(q),q)
\end{align*}
\end{defn}
The generic central character map of Definition \ref{defn:gcc}
can be characterized as follows:
\begin{prop}\label{prop:GCC}
We define $GCC=gcc\times\pi:\Delta(\mc{R})\to\mc{O}(\mc{R})$ where
$\pi:\Delta(\mc{R})\to\mc{Q}$ is the canonical map.
Then $GCC$ is the unique continuous map such that the following
diagram commutes:
\begin{equation}
\begin{CD}
\Delta(\mc{R})@>GCC>>\mc{O}(\mc{R})\\
@V{cc_\Lambda}VV    @VV\textup{ev}V\\
W_0\backslash T\times\mc{Q}@=W_0\backslash T\times\mc{Q}
\end{CD}
\end{equation}
\end{prop}
\begin{proof}
This is a reformulation of Theorem \ref{thm:gencc}.
\end{proof}
We are now in the position to formulate the first main
result of this paper:
\begin{thm}\label{thm:main1}
The map $GCC=gcc\times\pi:\Delta(\mc{R})\to\mc{O}(\mc{R})$
is a surjective local homeomorphism and gives $\Delta(\mc{R})$ the
structure of a locally constant sheaf on $\mc{O}(\mc{R})$.
\end{thm}
\begin{proof} Clearly $GCC$ is a local homeomorphism.
Using Definition \ref{defn:gcc} and Proposition \ref{prop:GCC}
we can reformulate Theorem \ref{thm:cont}(ii) by
stating that for any $W_0r\in W_0\backslash\operatorname{Res}(\mc{R})$ and
any connected component $U\subset\mc{Q}_{W_0r}^{\operatorname{reg}}$
inverse image $\Delta_C(\mc{R}):=GCC^{-1}(C)\subset\Delta(\mc{R})$ of
$C=\{W_0r\}\times U\subset\mc{O}(\mc{R})$
is a locally constant sheaf on $C$. In particular the cardinality
of the fibres of $GCC|_{\Delta_C(\mc{R})}$ is constant. Hence the
surjectivity of $GCC$ follows from Theorem \ref{thm:rescos}
by considering a generic parameter $q\in U$.
\end{proof}
\begin{cor} Let $W_0r\in W_0\operatorname{Res}(\mc{R})$ and
let $U\subset\mc{Q}_{W_0r}^{\operatorname{reg}}$ be a connected
component as in the proof of Theorem \ref{thm:main1}.
The restriction $\Delta_{C}(\mc{R})$ of $\Delta(\mc{R})$
to the connected component $C=\{W_0r\}\times U\subset\mc{O}(\mc{R})$
of $\mc{O}(\mc{R})$ is a constant sheaf.
\end{cor}
\begin{proof}
Since $U$ is the interior of a convex polyhedral cone
by Theorem \ref{thm:mgen} this follows trivially from
Theorem \ref{thm:main1}.
\end{proof}
\begin{cor}\label{cor:gendscc}
For all $q\in\mc{Q}$ the map
$gcc_q:\Delta(\mc{R},q)\to W_0\backslash\operatorname{Res}_q(\mc{R})$
is surjective.
\end{cor}
\begin{proof}
This follows immediately from the surjectivity of $GCC$.
\end{proof}
In particular, if $\delta_0\in\Delta(\mc{R},q_0)$ with
$gcc_q(\delta_0)=W_0r\in\operatorname{Res}_{q_0}(\mc{R})$ is an
irreducible discrete series character and
$U\subset\mc{Q}_{W_0r}^{reg}$ denotes the component of $q_0$, then
there exists a unique continuous family
$\delta\in\Delta(\mc{R},U)$ such that
$\operatorname{ev}_{q_0}(\delta)=\delta_0$. Observe that the open
cone $U\subset \mc{Q}$ is the maximal set to which $\delta$ can be
continued as a discrete series character (since the central
character $W_0r(q)$ will cease to be residual at
every boundary point of $U$). Hence the open cone $U$ is
determined by $\delta$.
\begin{defn}\label{denf:gends}
We denote this open cone by $U_\delta$,
and we call a continuous family of irreducible discrete series
characters $\delta$ which is extended to its maximal domain of definition
$U_\delta\ni q\to\delta(q)$ a \emph{generic irreducible discrete series
character}. We denote by $\Delta^{gen}(\mc{R})$ the finite set of
generic irreducible discrete series characters.
\end{defn}
\index{1dz@$\delta\in\Delta^{gen}(\mc{R})$, set of generic irreducible discrete series characters}
\begin{cor}\label{cor:mult}
For each component $C=\{W_0r\}\times U$
of $\mc{O}(\mc{R})$ we define a multiplicity $M_{C}\in\mathbb{Z}_{\geq0}$
of $C$ by $M_{C}:=
|\{\delta\in\Delta^{gen}(\mc{R})\mid GCC(\delta)= C\}|$.
Then $M_{C}>0$ for all components $C=\{W_0r\}\times U$.
For all $q\in U$ one has
$M_{C}=|\Delta_{W_0r}(\mc{R},q)|$, and for all
$q\in\mc{Q}$ one has (where $\chi_U$ denotes the characteristic
function of $U$):
\begin{equation}
|\Delta(\mc{R},q)|= \sum_{W_0r\in
W_0\backslash\operatorname{Res}(\mc{R})}
\sum_{U\in\mc{C}_{W_0r}}\chi_U(q)M_{\{W_0r\}\times U}
\end{equation}
\end{cor}
We reformulate Theorem \ref{thm:ratfam}
using our results on the generic central character. This
is the second main theorem of this paper:
\begin{thm}\label{thm:main2} Let $\delta\in\Delta^{gen}(\mc{R})$.
There exists a rational constant $d_\delta\in\mathbb{Q}^\times$
such that for all $q\in U_\delta$ we have
\begin{equation}
\mu_{Pl} (\{ \delta(q) \})=d_\delta m_{gcc(\delta)}(q)
\end{equation}
Here $m_{gcc(\delta)}\in K(\Lambda_\mathbb{Z})$ is explicitly
given by (\ref{eq:m}).
\end{thm}
\begin{rem} This result proves in particular Conjecture \cite[2.27]{Opd1},
and it shows that the constants defined in Conjecture
\cite[2.27]{Opd1} for special values of the parameters
can be determined from the rational constants $d_\delta$
defined for the irreducible generic discrete series characters.
Indeed, any irreducible discrete series character $\delta_0\in\Delta(\mc{R},q_0)$
determines a unique $\delta\in\Delta^{gen}(\mc{R})$ such that $\delta_0=\delta(q_0)$.
The constant defined in Conjecture \cite[2.27]{Opd1} is equal to
$d_\delta$ multiplied by a rational number depending on $q_0$ which
can be easily expressed in terms of the sets $R_{r,1}^{p,+}$, $R_{r,1}^{p,-}$,
and $R_{r,1}^{z}$ of roots whose associated factor in $m_{W_0r}$ becomes
zero at $q_0$).
\end{rem}

\section{The generic linear residual points and the evaluation map}
\label{sec:genlinres}

In this section we summarize, following \cite{HO} and \cite{Slooten},
the classification of the $W_0$-orbits of the generic linear residual points
for all irreducible root systems $R_1$
and we describe the evaluation map at a given parameter
$k\in\mc{K}=\mc{K}(R_1)$ of the parameter space associated with $R_1$.

\index{Kre@$\mc{K}_\xi^{reg}$, parameters $k \in \mc K$ for which $\xi (k)$ is residual}
For each generic linear residual point $\xi$ of $R_1$ we will describe the
open dense set $\mc{K}_\xi^{reg}$ of parameters $k$ such that
$\textup{ev}_k(\xi)=\xi(k)$ is still residual. In addition we will describe the set
$W_0\backslash\textup{Res}^{lin}(R_1,V,k)$ of residual orbits for each
$k\in\mc{K}$. To do this it is convenient to use the notion of
$k$-weighted distinguished
Dynkin diagrams with respect to a given bases $F_1=\{\alpha_1,\dots,\alpha_n\}$
of simple roots of $R_1$:
\index{Dyn@$\textup{Dyn}^{dist}(R_1,V,F_1,k)$, set of distinguished weighted Dynkin diagrams}
\begin{defn}\label{defn:Dyn}
For $k\in\mc{K}$ we define the set $\textup{Dyn}^{dist}(R_1,V,F_1,k)$
of distinguished $k$-weighted Dynkin diagrams for $(R_1,V,F_1,k)$ as the
set of $F_1$-dominant linear $(R_1,k)$-residual points. There is a canonical bijection
\begin{equation}
W_0\backslash\textup{Res}^{lin}(R_1,V,k)\stackrel{\simeq}{\longrightarrow}
\textup{Dyn}^{dist}(R_1,V,F_1,k)
\end{equation}
by which we will identify these two sets. We will represent
$D\in\textup{Dyn}^{dist}(R_1,V,F_1,k)$ by the Dynkin diagram of
$F_1$ in which the vertex corresponding to $\alpha_i\in F_1$
is labelled by the weight $\alpha_i(D)>0$ (or simply by
the list of values $(\alpha_1(D),\dots,\alpha_n(D))$).
\end{defn}
Given $k\in\mc{K}$ let $W_0\backslash\textup{Res}_k^{lin}(R_1)$
be the set of orbits of generic linear residual points $W_0\xi$ such that
$k\in\mc{K}^{reg}_\xi$. We will also describe in this section
the fibers of the evaluation map
\index{evk@$\mr{ev}_k$, evaluation of functions on $\mc K$ at $k$}
\begin{align}\label{eq:evaluation}
\textup{ev}_k:W_0\backslash\textup{Res}_k^{lin}(R_1)&\to
\textup{Dyn}^{dist}(R_1,V,F_1,k)\\
\nonumber W_0\xi&\to D=\xi(k)_+
\end{align}
where $\xi(k)_+\in W_0\xi(k)$ is the unique $F_1$-dominant element in the
orbit $W_0\xi(k)$.

If $D\in\textup{Dyn}^{dist}(R_1,V,F_1,k)$ and $\lambda>0$ then
$\lambda D\in\textup{Dyn}^{dist}(R_1,V,F_1,\lambda k)$ and
$-w_0(D)=D$ (using \cite[Theorem A.14(i)]{Opd1}).
This gives canonical identifications
\begin{equation}
\textup{Dyn}^{dist}(R_1,V,F_1,\lambda k)=|\lambda|\textup{Dyn}^{dist}(R_1,V,F_1,k)
\end{equation}
for all $\lambda\in\mathbb{R}^\times$.
Since the generic linear residual points depend linearly
on $k$ this remark implies that we only need to describe the set
$\textup{Dyn}^{dist}(R_1,V,F_1,k)$ and the fibres of $\widetilde{\textup{ev}}_k$
on all lines in the parameter space.

If $k_\alpha=2$ for all $\alpha\in R_1$ then the set $\textup{Dyn}^{dist}(R_1,V,F_1,k)$
is the usual set of distinguished Dynkin diagrams, corresponding to the set of
distinguished unipotent orbits of $\mathfrak{g}_\mathbb{C}(R_1)$ via the
Bala-Carter theorem.
For classical root systems it is known how to generalize combinatorially
the set of (distinguished) unipotent classes and the Bala-Carter bijection to
the set of $k$-weighted Dynkin diagrams \cite{Slooten}.
As this is a very useful description we will give these generalized
Bala-Carter maps as well.

\index{1dHR@$\Delta^{\mathbf{H}}(R_1,V,F_1,k)$, irreducible discrete series characters of
$\mb H (R_1,V,F_1,k)$} \index{gcc@$gcc^{\mathbf{H}}$, degenerate generic central
character map}
Let $\Delta^{\mathbf{H}}(R_1,V,F_1,k)$ be the collection of irreducible discrete series
characters of $\mb H (R_1,V,F_1,k)$. Consider the ``degenerate'' generic central
character map $gcc^{\mathbf{H}}$, which is the map
\begin{equation}\label{eq:bij}
gcc^{\mathbf{H}} : \Delta^{\mathbf{H}}(R_1,V,F_1,k) \to
W_0\backslash\textup{Res}^{lin}_k(R_1)
\end{equation}
\index{1dHW@$\Delta^{\mathbf{H}}_{W_0 D}(R_1,V,F_1,k)$,
subset of $\Delta^{\mathbf{H}}(R_1,V,F_1,k)$}
corresponding to the restriction of $gcc$ to the set $\Delta^{s}(\mc{R},q)$
(with $s\in T_u$ a $W_0$-invariant element) via the canonical bijections of
Corollary \ref{cor:dsgraded} and Proposition \ref{lem:redrat}). In the next section
we will prove that for all irreducible non-simply laced root systems
the map $gcc^{\mathbf{H}}$ maps the subset $\Delta^{\mathbf{H}}_{W_0 D}
(R_1,V,F_1,k) \subset \Delta^{\mathbf{H}}(R_1,V,F_1,k)$ of elements
with central character $W_0 D$ bijectively onto the fiber $\textup{ev}^{-1}_k(D)$,
where $\textup{ev}_k$ is the evaluation map
of (\ref{eq:evaluation}) for $R_1$, with one remarkable exception: in the case
$F_4$ it turns out that one has to count every occurrence of the unique singular
generic linear residual orbit ``$f_8$'' with multiplicity $2$.
In other words, in the notation of
Corollary \ref{cor:mult}, the multiplicities $M_{W_0r\times U}$ are always $1$
for orbits $W_0r$ of positive generic residual point, except for the unique singular
one (called $f_8$) of $F_4$, in which case the multiplicity is always $2$
(these results will be shown in the next section).

It is interesting in addition that this bijection also holds for type $D_n$ after we
make a small adaptation in order to see type $D_n$ as a specialization of type $B_n$.
The proofs of these facts do not depend on the classical Kazhdan-Lusztig classification.
The only point where one needs to resort to nontrivial computations is in the verification
of the fact that the multiplicity of $f_8$ is always $2$. This follows from results by
Reeder \cite{Re}. Since our parametrization clearly also holds for type
$A_n$ it follows that the deformation method gives the classification
of the discrete series in all cases except for types $E_{6,7,8}$ (in which cases
the Kazdan-Lusztig classification is available of course).

In the ``classical situation'' $k_\alpha=2$ for all
$\alpha\in R_1$ one associates a set of Springer representations
$\Sigma_{u(D)}$ of $W_0$ to the distinguished unipotent orbit
$u=u(D)$ of $G_\mathbb{C}^{ad}(R_1)$ associated with $D$.
The Kazhdan-Lusztig parametrization says that the set
$\Delta_{W_0D}^{\mathbf{H}}(R_1,V,F_1,k_\alpha=x)$ (equal parameters with $x>0$)
is in canonical bijection with the set $\Sigma_{u(D)}$.

For classical root systems \cite{Slooten} explained how to generalize combinatorially
the set of ``$k$-unipotent'' elements $u(D)$ associated to
$D\in\textup{Dyn}^{dist}(R_1,V,F_1,k)$ and the set of corresponding
``$k$-Springer representations'' $\Sigma_{u(D)}(k)$ of $W_0$.
This makes it possible to recast the above parametrizations
in the form of a generalized Kazhdan-Lusztig correspondence between the set
$\Delta_{W_0D}^{\mathbf{H}}(R_1,V,F_1,k_\alpha=x)$ and the sets of
$k$-Springer representations $\Sigma_{u(D)}(k)$ on a combinatorial level for
arbitrary $k$. Our result thus establishes this aspect of the conjectures by
Slooten (\cite{Slooten}).

We will include the generalized Kazhdan-Lusztig parameters for the classical
root systems, and describe their relation with the
alternative parametrization (\ref{eq:bij}).

\subsection{The case $R_1=A_n,\ n\geq 1$}
\label{A:res}

In this case $\mc{K} \cong \mathbb{R}$.
Choose the bases of simple roots $F_1=\{e_1-e_2,\dots,e_{n-1}-e_n\}$ for $R_1$,
and define $\xi:\mc{K}\to V$ by
the equations $\alpha(\xi(k))=k$ for all $\alpha\in F_1$. Then
$W_0\backslash\textup{Res}^{lin}(R_1)=\{W_0\xi\}$.
The set $\mc{K}^{reg}_\xi$ is equal to $\mc{K}\backslash\{0\}$.
For all $k\in\mc{K}^{reg}_\xi$ we have
$\textup{Dyn}^{dist}(R_1,V,F_1,k)=\{D(k)\}$ with $D(k)=(|k|,\dots,|k|)$.
We have $\textup{ev}^{-1}_k(D(k))=\{W_0\xi\}$.

\subsection{The case $R_1=B_n,\ n\geq 2$}\label{B:res}

The results in this subsection are due to Slooten \cite{Slooten}.
Put $R_1=\{\pm e_i\pm e_j\mid 1\leq i\not=j\leq n\}\cup\{\pm e_i\mid 1\leq i\leq n\}$.
Choose as a basis $F_1=\{e_1-e_2,\dots,e_{n-1}-e_n,e_n\}$. We put
$k(e_i\pm e_j)=k_1\in\mathbb{R}$ and $k(e_i)=k_2\in\mathbb{R}$ and in this
way make the identification $\mc{K}=\mathbb{R}^2$. If $k_1\not=0$ then
we define $m\in\mathbb{R}$ by $m=k_2/k_1$.

We first describe the generic linear residual points.
Given a partition $\lambda\in\mc{P}(n)$ (i.e. a partition $\lambda\vdash n$)
we define a $\mc{K}$-valued \index{Pn@$\mc P (n)$, set of partitions of $n \in \mh N$}
point $\xi_\lambda$ as follows. Given a box $b$ of $\lambda$ let $i(b)$ be
its row number and $j(b)$ its column number. We define the content $c(b)$
of the box $b$ by $c(b)=j(b)-i(b)$. We call the tableau of shape $\lambda$
in which the boxes $b\in\lambda$ are filled with the expression
$c(b)k_1+k_2$ the \emph{generic $k$-shifted tableau} of $\lambda$,
denoted by $T(\lambda,k)$.
We order the boxes of $T(\lambda,k)$ in the standard way by
reading the tableau from left to right and from top to bottom.
Then we define $\xi_\lambda$ as the $\mc{K}$-valued point of $V$
such that the i-th coordinate $e_i(\xi)$ is equal to the filling
$c(b_i)k_1+k_2$ of the i-th box of $T(\lambda,k)$.
\begin{thm}\label{thm:genresBn}
We have a bijection
\begin{align*}
\Lambda:\mc{P}(n)&\to W_0\backslash\textup{Res}^{lin}(R_1)\\
\lambda&\to W_0\xi_\lambda
\end{align*}
\end{thm}
The set $\mc{K}^{reg}_\lambda$ of regular parameters for $\xi_\lambda$
is of the form
\begin{equation}
\mc{K}^{reg}_{\lambda}=\mc{K}\backslash\bigcup_{m\in M_\lambda^{sing}}L_m
\end{equation}
where $L_m=\{(k_1,k_2)\mid k_2=mk_1\}\subset\mc{K}$ and
where $M_\lambda^{sing}$ is a set of half-integral
ratio's $m\in\mathbb{Z}/2$ which are called singular with respect to $\lambda$
and which will be described in Proposition \ref{prop:sing} below.
We first define for $m\in\mathbb{Z}/2$ the $m$-shifted content tableau
$T_m(\lambda)$ of $\lambda$ as follows.
The tableau $T_m(\lambda)$ has shape $\lambda$ and
the box $b$ of $T_m(\lambda)$ is filled with the value $|c(b)+m|$
(i.e. the absolute value of the filling of the same box in
$T(\lambda,(1,m))$. The following notion
plays an important role:
\begin{defn}\label{defn:extr}
Let $\lambda\vdash n$ and $m\in\mathbb{Z}/2$.
The list of extremities of $T_m(\lambda)$ is the weakly increasing
list consisting of the following numbers.
If $m\in\mathbb{Z}$ (resp. $m\in\mathbb{Z}+1/2$) then the extremities are
the fillings of the boxes of $T_m(\lambda)$ at the end of a row of $T_m(\lambda)$
which are on or above the $0$ diagonal (resp. the upper $1/2$-diagonal)
and the boxes at the bottom of a column of $T_m(\lambda)$ which are on or below
the zero diagonal (resp. the lower $1/2$ diagonal). Here we agree to count $0$
twice if $0$ is both at the the end of a row and of a column.
\end{defn}

\begin{prop}\label{prop:sing}
We have $m\in M_\lambda^{reg}$ (the complement
of $M_\lambda^{sing}$, i.e. the values $m\in\mathbb{R}$ such that
$\xi_\lambda(k_1,mk_1)$ is residual if $k_1\not=0$)
if and only if $m\not\in\mathbb{Z}/2$ or $m\in \mathbb{Z}/2$
and the extremities of $T_m(\lambda)$ are all distinct.
If $m<1-n$ or of $m>n-1$ then $m$ is regular with respect
to any partition $\lambda\vdash n$.
\end{prop}

Let $\mc K^{reg}$ be the intersection of sets $\mc K^{reg}_\xi =
\mc K^{reg}_{W_0 \xi}$, where $\xi$ runs over $\mr{Res}^{lin}(R_1 ,V, k)$.
\index{Kre@$\mc K^{reg}$, regular parameters in $\mc K$}

\begin{cor}\label{cor:linresgenB}
We have
\begin{equation}
\mc{K}^{reg}=\mc{K}\backslash\bigcup_{m}L_m
\end{equation}
where $m$ runs over the half-integral values satisfying $1-n\leq m\leq n-1$.
In particular, if $k\not\in L_m$ for all half-integral $m$ satisfying
$1-n\leq m\leq n-1$ the evaluation map
\begin{equation}
\textup{ev}_k:W_0\backslash\textup{Res}^{lin}(R_1)\to\textup{Dyn}^{dist}(R_1,V,F_1,k)
\end{equation}
is bijective.
\end{cor}
Let $m\in\mathbb{Z}/2$ and $\lambda\vdash n$.
Suppose that $m\not\in M^{sing}_\lambda$ (in other words
$\xi_\lambda(k_1,mk_1)\in\textup{Res}^{lin}(R_1,V,F_1,(k_1,mk_1))$ if $k_1\not=0$).
Since $W_0$ contains sign changes and permutations the corresponding element
$D(k)\in\textup{Dyn}^{dist}(R_1,V,F_1,(k_1,mk_1))$ has coordinates which are all
of the form $p|k_1|$ with $p\geq 0$ and $p\in m+\mathbb{Z}$.
Conversely, any point $D(k)\in\textup{Dyn}^{dist}(R_1,V,F_1,k)$ is of this
form. In order to see this we recall the following result
(see \cite{HO}, \cite{Slooten}):
\begin{prop}\label{prop:multres}
Let $m\in\mathbb{Z}/2$ and let $k=(k_1,mk_1)$ with $k_1\not=0$.
Let $D\in\mathbb{R}^n$ be dominant with respect $F_1$. Then
$D\in\textup{Dyn}^{dist}(R_1,V,F_1,k)$ only if all coordinates
of $D$ are of the form $p|k_1|$ with $p\geq 0$. So let us suppose
that all coordinates of $D$ are of the above mentioned form.
Let $\mu_p=\mu_p(D)$ denote the multiplicity of $p|k_1|$ as a
coordinate of $D$. We distinguish the following cases:
\begin{enumerate}
\item If $m=0$ then
$D\in\textup{Dyn}^{dist}(R_1,V,F_1,k)$ iff
(i) $\mu_r=1$ if $r$ is maximal such that $\mu_r\not=0$,
(ii) $\mu_p\in\{\mu_{p+1},\mu_{p+1}+1\}
\text{\ for\ all\ }p>0$, and (iii) $\mu_0=\lfloor 1/2(\mu_1+1)\rfloor$.
\item If $m\in\mathbb{Z}\backslash\{0\}$ then
$D\in\textup{Dyn}^{dist}(R_1,V,F_1,k)$ iff
(i) $\mu_r=1$ if $r$ is maximal such that $\mu_r\not=0$,
(ii) $\mu_p\in\{\mu_{p+1},\mu_{p+1}+1\}\text{\ for\ all\ }p\geq|m|$,
(iii) $\mu_p\in\{\mu_{p+1}-1,\mu_{p+1}\}
\text{\ for\ }1\leq p\leq |m|-1$,
\text{\ and\ finally\ }
(iv) $\mu_0=\lfloor \mu_1/2\rfloor$.
\item If $m\in\mathbb{Z}+1/2$ then
$D\in\textup{Dyn}^{dist}(R_1,V,F_1,k)$ iff
(i) $\mu_r=1$ if $r$ is maximal such that $\mu_r\not=0$,
(ii) $\mu_p\in\{\mu_{p+1},\mu_{p+1}+1\}
\text{\ for\ all\ }p\geq|m|$,\text{\ and\ finally\ }
(iii) $\mu_p\in\{\mu_{p+1}-1,\mu_{p+1}\}
\text{\ for\ }1/2\leq p\leq |m|-1$.
\end{enumerate}
\end{prop}
\begin{defn}\label{defn:jump}
We keep the notations as in Proposition \ref{prop:multres}.
Assume that $D\in\textup{Dyn}^{dist}(R_1,V,F_1,k)$.
We call $p\in m+\mathbb{Z}$ a jump of $D$ if $p\geq|m|$ and
$\mu_p=\mu_{p+1}+1$
or if $0< p <|m|$ and $\mu_p=\mu_{p+1}$. Finally we add $0$ (if
$m\in\mathbb{Z}$) or $-1/2$ (if $m\in 1/2+\mathbb{Z}$) to the list
of jumps of $D$ in order to ensure that the number of jumps of
$D$ is equal to $\lceil |m|\rceil+2\nu$ for some
$\nu\in\mathbb{Z}_{\geq 0}$ (this is always possible, see \cite{Slooten}).
\end{defn}
\begin{rem}
It is a simple matter to reconstruct $D$ from its list of jumps
by computing the multiplicities $m_p$ of the entries of the form $p|k_1|$,
starting from the top $m_r=1$.
\end{rem}
This gives rise to a different classification of the set of $k$-weighted
distinguished Dynkin diagrams $\textup{Dyn}^{dist}(R_1,V,F_1,k)$ by the
introduction of a combinatorial analogue $\mc{U}_m(n)$ of the
corresponding set of ``distinguished $m$-unipotent classes'':
\begin{defn}\label{defn:unipotent} If $m\in\mathbb{Z}$ we define
\begin{equation}
\mc{U}_m^{dist}(n)=\{u\vdash 2n+m^2\mid
l(u)\geq |m|\text{\ and\ }u
\text{\ has\ odd,\ distinct\ parts}\}
\end{equation}
and if $m\in 1/2 +\mathbb{Z}$ we define
\begin{equation}
\mc{U}_m^{dist}(n)=\{u\vdash 2n+m^2-1/4\mid
l(u)\geq \lfloor |m|\rfloor\text{\ and\ }u
\text{\ has\ even,\ distinct\ parts}\}
\end{equation}
\end{defn}
\begin{prop}\label{prop:fBC}
Let $m\in\mathbb{Z}/2$ and let $u\in\mc{U}_m^{dist}(n)$.
Let $k=(k_1,mk_1)\in L_m$ with $k_1\not=0$.
If $m\in 1/2+\mathbb{Z}$ we add $0$ as a part of $u$ if necessary
to assure that the number of parts of $u$ is equal to $\lceil|m|\rceil+2\nu$
for some $\nu\in\mathbb{Z}_{\geq 0}$. The list $j=j(u)$
consisting of the numbers $(u_i-1)/2$ where $u_i$ runs over the parts
of $u$ (ordered in ascending order) is the list of jumps of a unique
distinguished $k$-weighted Dynkin diagram $D\in\textup{Dyn}^{dist}(R_1,V,F_1,k)$
(where $D$ is of the form as described in Proposition \ref{prop:multres}).
This sets up a bijection
\begin{equation}
f^{BC}_k:\mc{U}_m^{dist}(n)\to\textup{Dyn}^{dist}(R_1,V,F_1,k)
\end{equation}
Finally we remark that $\textup{Dyn}^{dist}(R_1,V,F_1,(0,0))=\emptyset$.
\end{prop}
This completes the classification of the set $\textup{Dyn}^{dist}(R_1,V,F_1,k)$
for all values of $k\in\mc{K}$.
It remains to describe for all special values $k\in L_m\backslash\{0\}$
and all $D\in\textup{Dyn}^{dist}(R_1,V,F_1,k)$
the fiber $\textup{ev}_k^{-1}(D)$ of the evaluation map
\begin{equation}
\textup{ev}_k:W_0\backslash\textup{Res}^{lin}_k(R_1)\to
\textup{Dyn}^{dist}(R_1,V,F_1,k)
\end{equation}
(where $W_0\backslash\textup{Res}^{lin}_k(R_1)$ is the
set of orbits of generic residual points which remain residual
upon evaluation at $k$ (note that this depends on $m=m(k)$ rather than $k$)).
Equivalently, we will describe for each $D\in\textup{Dyn}^{dist}(R_1,V,F_1,k)$
the set
\begin{equation}
\mc{P}_m(D):=\Lambda^{-1}(\textup{ev}_k^{-1}(D))\subset\mc{P}(n)
\end{equation}
of all partitions $\lambda$ of $n$ such that $W_0\xi_\lambda(k)=W_0D$.
\index{P2n@$\mc P (2,n)$, set of bipartitions of $n \in \mh N$}
\begin{defn} Let $m\in\mathbb{Z}/2$.
Given $u\in\mc{U}_m^{dist}(n)$ we define a bipartition
$\phi_m(u)\in\mc{P}(2,n)$ as follows. First assume that $m$ is nonnegative.
Let $j=j(u)$ be the sequence of jumps of length
$\lceil m\rceil+2\nu\in\mathbb{Z}_{\geq 0}$ associated to
$u$ as in proposition \ref{prop:fBC}.
Then we define $\phi_m(u)=(\xi_m(u),\eta_m(u))\in\mc{P}(2,n)$ where
\begin{align*}
\xi_m(u)&=(j_1,j_3,\dots,j_{2\nu-1},j_{2\nu+1},j_{2\nu+2}-1,
j_{2\nu+3}-2,\dots,j_{2\nu+m}-(m-1)),\\
\eta_m(u)&=(j_2+1,j_4+1,\dots,j_{2\nu}+1)
\end{align*}
if $m\in\mathbb{Z}$ and
\begin{align*}
\xi_m(u)&=(j_1+\frac{1}{2},j_3+\frac{1}{2},\dots,j_{2\nu+1}+\frac{1}{2},
j_{2\nu+2}-\frac{1}{2},j_{2\nu+3}-\frac{3}{2},
\dots,j_{2\nu+m+\frac{1}{2}}-(m-1)),\\
\eta_m(u)&=(j_2+\frac{1}{2},j_4+\frac{1}{2},\dots,j_{2\nu}+\frac{1}{2})
\end{align*}
if $m\in\frac{1}{2}+\mathbb{Z}$. If $m<0$ then we define
$\phi_m(u):=(\eta_{-m}(u),\xi_{-m}(u))\in\mc{P}(2,n)$.
\end{defn}
\begin{defn}\label{defn:sym} Let $(\xi,\eta)\in\mc{P}(2,n)$.
Recall the equivalence class of $m$-symbols of $(\xi,\eta)$ denoted by
$\bar{\Lambda}^m(\xi,\eta)$ (if $m=0$ we use the $+$-symbol)
(see \cite[Definition 3.6]{Slooten} for the definition of these symbols).
If $(\xi,\eta)\in\mc{P}(2,n)$ we denote by $[(\xi,\eta)]_m$ the
set of $(\xi^\prime,\eta^\prime)\in\mc{P}(2,n)$ such that
$\bar{\Lambda}^m(\xi,\eta)$ and $\bar{\Lambda}^m(\xi^\prime,\eta^\prime)$
have representatives which contain the same entries the same number of
times. For $u\in\mc{U}_m^{dist}(n)$ we define $\Sigma_m(u)\subset\mc{P}(2,n)$
by $\Sigma_m(u):=[\phi_m(u)]_m$.
\end{defn}
Finally the following result of Slooten gives the desired parametrization
of the set $\mc{P}_m(D)$ (and hence of the fiber
$\textup{ev}_k^{-1}(D)$ of the evaluation map):
\begin{thm}\label{thm:sym=conf}(see \cite[Theorem 5.27]{Slooten})
The joining map $\mc{J}_m$ (see \cite[Definition 5.18]{Slooten}) is
well defined on $\Sigma_m(u)$ and this yields a bijection
\begin{equation}
\mc{J}_m:\Sigma_m(u)\to\mc{P}_m(f^{BC}_k(u))
\end{equation}
whose inverse is given by the splitting map $\mc{S}_m$
(see \cite[Definition 5.16]{Slooten}).
\begin{cor} Let $m\in\mathbb{Z}/2$, $k=(k_1,mk_1)$ with $k_1\not=0$
and suppose that $D\in\textup{Dyn}^{dist}(R_1,V,F_1,k)$. Put
$u=(f^{BC}_k)^{-1}(D)\in\mc{U}_m(n)$. We can arrange that $u$ has
$\lceil m\rceil+2\nu$ parts (with $\nu\in\mathbb{Z}_{\geq 0}$).
Then
\begin{equation}
|\mc{P}_m(D)|=
\begin{cases}
\binom{\lceil m\rceil+2\nu}{\nu} & \text{if\ } u_1\not=0, \\
\binom{\lceil m\rceil+2\nu-1}{\nu} & \text{otherwise}.
  \end{cases}
  \end{equation}
\end{cor}
\end{thm}

\subsubsection{The case $k_1=0$}
\label{subsub:k10}

If $k=(0,0)$ then there are no linear residual points since
$k$ is singular for all generic linear residual points.

The situation with $k=(0,k_2)$ with $k_2\not=0$ is an important
special case. Its importance stems in part from the fact that although $k$
is highly nongeneric \emph{it is regular for all generic linear residual points}.
In fact, all generic linear residual orbits coalesce upon specialization
for $k_1=0$ to the unique orbit of residual points $W_0\xi(k)$ where $\xi$ is
defined by
$\xi_i(k)=k_2$ for all $i=1,\dots,n$. In other words, we have
\begin{equation}
\textup{Res}^{lin}_k(R_1)=\textup{Res}^{lin}(R_1)
\end{equation}
and (in the coordinates $e_1,\dots,e_n$ of $V$)
\begin{equation}
\textup{Dyn}^{dist}(R_1,V,F_1,k)=\{(|k_2|,\dots,|k_2|)\}
\end{equation}
The evaluation map $\textup{ev}_k$ is the unique map
from $\textup{Res}^{lin}(R_1)$ to $\textup{Dyn}^{dist}(R_1,V,F_1,k)$.

\subsection{The case $R_1=C_n,\ n\geq 3$}
\label{C:res}

Put $R_1=\{\pm e_i\pm e_j\mid 1\leq i\not=j\leq n\}\cup\{\pm 2e_i\mid 1\leq i\leq n\}$.
Choose as a basis $F_1=\{e_1-e_2,\dots,e_{n-1}-e_n,2e_n\}$. We put
$k(e_i\pm e_j)=k_1\in\mathbb{R}$ and $k(2e_i)=k_2\in\mathbb{R}$ and in this
way make the identification $\mc{K}=\mathbb{R}^2$. Clearly we have the following
equality for all $k=(k_1,k_2)$:
\begin{equation}
\textup{Res}^{lin}(C_n,(k_1,k_2))=\textup{Res}^{lin}(B_n,(k_1,k_2/2))
\end{equation}
Since $W_0(B_n)=W_0(C_n)$ we see that everything reduces to the
case $R_1=B_n$.

\subsection{The case $R_1=D_n,\ n\geq 4$}
\label{D:res}

We put $R_1=\{\pm e_i\pm e_j\mid 1\leq i\not=j\leq n\}$.
Choose as a basis $F_1=\{e_1-e_2,\dots,e_{n-1}-e_n,e_{n-1}+e_n\}$.
The case $R_1=D_n$ can be reduced to the discussion of subsection
\ref{B:res} as well in the following way, using the Clifford theory
discussion from \cite{RamRam}.

Let $F_1^b$ denote the basis for $B_n$ as in subsection \ref{B:res}. Let
\begin{equation}
\psi:\mathbf{H}(B_n,V,F_1^b,(k_1,k_2))\to
\mathbf{H}(B_n,V,F_1^b,(k_1,-k_2))
\end{equation}
be the unique algebra isomorphism such that $\psi(x)=x$ for all
$x\in V^*=\mathbb{R}\otimes X$, $\psi(s_{e_{i-1}-e_i})=s_{e_{i-1}-e_i}$
(for all $i=2,\dots,n$) and $\psi(s_{e_n})=-s_{e_n}$ (compare with
the isomorphisms $\psi_s$ discussed in paragraph
\ref{subsub:defn}). Then $\psi$ restricts to an
involutive automorphism of
$\mathbf{H}(B_n,V,F_1^b,(k_1,0))$. Let $\Psi=\{1,\psi\} \cong {\mathbb{Z}/2}$ be
the group of automorphims of $\mathbf{H}(B_n,V,F_1^b,(k_1,0))$
generated by $\psi$. Then it is easy to see that
\begin{equation}\label{eq:Dn=Bninv}
\mathbf{H}(D_n,V,F_1,(k_1,0)) \cong \mathbf{H}(B_n,V,F_1^b,(k_1,0))^\Psi
\end{equation}
(where the generator $s_{e_{n-1}+e_n}$ on the left hand side
corresponds to the element $s_{e_n}s_{e_{n-1}-e_{n}}s_{e_n}$ on the right
hand side).

Let $k=k(\pm e_i\pm e_j)\in\mc{K}(D_n)$.
We use $k$ as a
coordinate on the line $L_0\subset\mc{K}(B_n)$ by identifying
$k$ with the element $(k,0)\in L_0$.
Let us from now assume that
$k\in\mc{K}^{reg}(D_n)=\mc{K}(D_n)\backslash\{0\}$ (and in the
context of $R_1=B_n$ we identify $k$ with $(k,0)\in L_0$).
We have
$W_0(B_n)=W_0(D_n)\rtimes\Gamma$ where $\Gamma=\{e,\gamma\}\cong {\mathbb{Z}/2}$
and $\gamma$ is the diagram automorphism that exchanges
$e_{n-1}-e_n$ and $e_{n_1}+e_n$.
Hence the center equals (see Corollary \ref{cor:simple}):
\begin{equation}
\mathbf{Z}(B_n,F_1^b,(k,0))=\mathbf{Z}(D_n,F_1,k)^\Gamma
\end{equation}
It is easy to see that
for every $u\in\mc{U}_0^{dist}(n)$ (defined as in subsection \ref{B:res})
the orbit $W_0(B_n)f^{BC}_k(u)\in W_0(B_n)\backslash\textup{Res}(B_n,k)$ is in
fact a single $W_0(D_n)$-orbit of residual points for $R_1=D_n$. It follows that
\begin{equation}
f^{BC}_k:\mc{U}_0^{dist}(n)\to\textup{Dyn}^{dist}(D_n,F_1,k)
\end{equation}
is a bijection.

Observe that we have (using the notation of Theorem \ref{thm:genresBn})
the following relation:
\begin{equation}\label{eq:inv}
W_0\xi_{\lambda^\prime}(k_1,-k_2)=W_0\xi_\lambda(k_1,k_2)
\end{equation}
where $\lambda\to\lambda^\prime$ is the conjugation
involution of $\mc{P}(n)$. Thus the set
$W_0(B_n)\backslash\textup{Res}_0^{lin}(B_n)$ of orbits of generic residual
$B_n$-points which remain residual if we restrict $(k_1,k_2)$ to
a (nonzero) element $(k,0)\in L_0$ admits an involution $\iota$
given (via $\Lambda$) by the conjugation involution.
By Proposition \ref{prop:sing} this involution acts
\emph{in a fixed point free manner} on $W_0\backslash\textup{Res}_0^{lin}(B_n)$.
The involution is clearly compatible with the evaluation
map $\textup{ev}_0$. It follows from (\ref{eq:inv}) that for all
$\delta\in\Delta^{\mathbf{H}}(B_n,V,F_1^b,(k,0))$ we have
\begin{equation}\label{eq:Bn:iota}
gcc^{\mb H} (\delta\circ\psi) = \iota (gcc^{\mb H} (\delta))
\end{equation}
Accordingly we define
\begin{equation}\label{eq:Dn}
W_0(D_n)\backslash\textup{Res}^{lin}(D_n)_k^\sharp:=W_0(B_n)
\backslash\textup{Res}_0^{lin}(B_n)/\{e,\iota\}
\end{equation}
and we have a corresponding evaluation map
\begin{equation}\label{eq:ev}
W_0(D_n)\backslash\textup{Res}^{lin}(D_n)_k^\sharp\to\textup{Dyn}^{dist}(D_n,F_1,k)
\end{equation}
\begin{rem}\label{rem:KL}
The relation with the usual Kazhdan-Lusztig parameters for $D_n$ is
as follows. For all $u\in\mc{U}_0^{dist}(n)$ the involution $\iota$ acts without
fixed points on the set $\Sigma_0(u)$ by:
\begin{align*}
\iota:\Sigma_0(u)&\to\Sigma_0(u)\\
(\xi,\eta)&\to(\eta,\xi)
\end{align*}
The set $\Sigma^{D_n}(u)$ of Springer representations
of $W_0(D_n)$ associated with $u$ is the set of $\{1,\iota\}$-orbits
in $\Sigma_0(u)$. In particular, for all $D\in\textup{Dyn}^{dist}(D_n,F_1,k)$
we have a natural bijection between the fiber
$(\textup{ev}_{k}^\sharp)^{-1}(D)$ and the set of classical Kazhdan-Lusztig
parameters $\Sigma^{D_n}(u)$ associated to $u=u(D)$.
\end{rem}

\subsection{The case $R_1=E_n,\ n=6,7,8$}
\label{E:res}

In the simply laced cases we can classify the generic linear residual
orbits with the weighted Dynkin diagrams for the distinguished nilpotent
orbits (see \cite[Proposition B.1(i)]{Opd1}). Since the weighted Dynkin
diagrams characterize the nilpotent orbits completely by the Bala-Carter
theorem (see \cite{C}) we obtain for all $k\not=0$ a bijection
\begin{equation}
f^{BC}_k:\mc{U}^{dist}(R_1)\to\textup{Dyn}^{dist}(R_1,V,F_1,k)
\end{equation}
where $\mc{U}^{dist}(R_1)$ denotes the set of distinguished nilpotent
orbits of the simple complex Lie algebra with root system $R_1$. It is
well known that the values of the roots on the generic linear residual
points are integral linear combinations of the $k(\alpha)$ (corresponding
to the fact that the roots take even values on the distinguished weighted
Dynkin diagrams). We refer to \cite[pages 176-177]{C} for the tables of
the distinguished weighted Dynkin diagrams.

\subsection{The case $R_1=F_4$}
\label{F:res}

Let $(\alpha_1,\alpha_2,\alpha_3,\alpha_4)$ be a basis of simple roots of $R_1$
such that
$\alpha_1$ and $\alpha_2$ are long, $\alpha_3$ and $\alpha_4$ are short, and
$\alpha_2(\alpha^\vee_3)=-2$.

The set $W_0\backslash\textup{Res}^{lin}(F_4)$ was completely
classified in \cite[Table 4.10]{HO}, but unfortunately this table contains
an error (the coordinates of $f_7$ are incorrect).
We therefore include the corrected table (see Table 1) below.
There are eight orbits of generic linear
residual points for $F_4$, numbered $f_1,\dots,f_8$.
The orbits are generically regular with respect to the $W_0$-action,
except for $f_8$ which generically has an isotropy group of type
$A_1\times A_1$. In the table below we have specified
for each generic linear residual orbit $f_n=W_0\xi_n$
a generic linear residual point $\xi_n$  by means of
the vector of values $(\alpha_1(\xi_n),\dots,\alpha_4(\xi_n))$.
Here $k=(k_1,k_2)$ where $k_1$ is the parameter of the long
roots.
\begin{table}
\caption{$F_4$: Generic linear residual orbits}
\begin{equation*}
\begin{array}{|c|c|}
\hline\text{Orbits}\ f=W_0\xi&\xi\\\hline
f_1&\xi_1=(k_1,k_1,k_2,k_2)\\\hline
f_2&\xi_2=(k_1,k_1,k_2-k_1,k_2)\\\hline
f_3&\xi_3=(k_1,k_1,k_2-k_1,k_1)\\\hline
f_4&\xi_4=(k_1,k_1,k_2-2k_1,k_2)\\\hline
f_5&\xi_5=(k_1,k_1,k_2-2k_1,2k_1)\\\hline
f_6&\xi_6=(k_1,k_1,k_2-2k_1,k_1)\\\hline
f_7&\xi_7=(k_1,k_1,k_2-2k_1,-2k_2)\\\hline
f_8&\xi_8=(0,k_1,0,k_2-k_1)\\\hline
\end{array}
\end{equation*}
\end{table}
\begin{table}
\caption{$F_4$: Regular parameters}
\begin{equation*}
\begin{array}{|c|c|}
\hline\text{Orbit}&\mc{K}^{reg}_\xi\\\hline
f_1&(2k_1+3k_2)(3k_1+4k_2)(3k_1+5k_2)(5k_1+6k_2)\not=0\\\hline
f_2&(k_1^2-(6k_2)^2)k_2\not=0\\\hline
f_3&(3k_1+2k_2)(k_1+3k_2)(2k_1+3k_2)(3k_1+4k_2)\not=0\\\hline
f_4&(2k_1-3k_2)(3k_1-4k_2)(3k_1-5k_2)(5k_1-6k_2)\not=0\\\hline
f_5&((3k_1)^2-(2k_2)^2)(k_1^2-(3k_2)^2)\not=0\\\hline
f_6&(3k_1-2k_2)(k_1-3k_2)(2k_1-3k_2)(3k_1-4k_2)\not=0\\\hline
f_7&((3k_1)^2-k_2^2)k_1\not=0\\\hline
f_8&k_1k_2\not=0\\\hline
\end{array}
\end{equation*}
\end{table}
We list in Table 3 the non-generic values of $k$, together
with the set $\textup{Dyn}^{dist}(k):=\textup{Dyn}^{dist}(R_1,V,F_1,k)$
of $k$-weighted Dynkin diagrams
and for each $D\in\textup{Dyn}^{dist}(k)$ the inverse image
${\textup{ev}}_{k}^{-1}(D)$ of the map
\begin{equation}
{\textup{ev}}_{k}:W_0\backslash{\textup{Res}}_{k}^{lin}\to\textup{Dyn}^{dist}(k)
\end{equation}
\begin{rem}
In Table 3 we assume that $x>0$. Not all special parameters
are listed in table 3 but all other special values can be obtained
from the listed ones by applying the following symmetries.
First of all we have
$f_i(k_1,k_2)=f_i(-k_1,-k_2)$ (since $-\textup{id}\in W_0$) and
$f_i(k_1,k_2)=f_{\theta(i)}(k_1,-k_2)=f_{\theta(i)}(-k_1,k_2)$ with
$\theta=(14)(36)$.
With these transformations we can reach all quadrants of $\mc{K}$
from the positive quadrant. In addition we have used
the following symmetry (arising from interchanging the long and
short roots) to reduce the length of Table 3: Let
$\Psi(a,b,c,d)=(2d,2c,b,a)$. Then we can define $D_i(2k_2,k_1)$ by
$D_i(2k_2,k_1)=\Psi(D_i(k_1,k_2))$. The map $\Psi$ acts as follows on the
set of generic linear residual orbits: $\Psi(f_i(k_1,k_2))=f_{\sigma(i)}(2k_2,k_1)$
where $\sigma$ is the transposition $(27)$. Observe that
$\Psi^2(a,b,c,d)=(2a,2b,2c,2d)$, thus $\Psi^2$ corresponds to
replacing $x$ by $2x$.
\end{rem}
\begin{table}
\caption{$k$-weighted Dynkin diagrams and confluence data for $F_4$}
\begin{equation*}
\begin{array}{|c|l|c|}
\hline{k=(k_1,k_2)}&D\in\textup{Dyn}^{dist}(k)&{\textup{ev}}_{k}^{-1}(D)\\\hline
(0,x)&D_1=(0,0,x,x)&f_1,f_2,f_4\\\hline
     &D_2=(0,0,x,0)&f_3,f_5,f_6\\\hline
(x,x)&D_1=(x,x,x,x)&f_1\\\hline
     &D_2=(x,x,0,x)&f_2,f_3\\\hline
     &D_3=(0,x,0,x)&f_5,f_7\\\hline
     &D_4=(0,x,0,0)&f_4,f_6,f_8\\\hline
(x,2x)&D_1=(x,x,2x,2x)&f_1\\\hline
     &D_2=(x,x,x,2x)&f_2\\\hline
     &D_3=(x,x,x,x)&f_3\\\hline
     &D_4=(x,x,0,2x)&f_4,f_5\\\hline
     &D_5=(x,x,0,x)&f_6,f_7\\\hline
     &D_6=(0,x,0,x)&f_8\\\hline
(x,3x)&D_1=(x,x,3x,3x)&f_1\\\hline
      &D_2=(x,x,2x,3x)&f_2\\\hline
      &D_3=(x,x,x,3x)&f_4\\\hline
      &D_4=(x,x,2x,x)&f_3\\\hline
      &D_5=(x,x,x,2x)&f_5\\\hline
      &D_6=(x,x,x,x)&f_6\\\hline
      &D_7=(0,x,0,2x)&f_8\\\hline
(2x,3x)&D_1=(2x,2x,3x,3x)&f_1\\\hline
       &D_2=(2x,2x,x,3x)&f_2\\\hline
       &D_3=(2x,2x,x,2x)&f_3\\\hline
       &D_4=(2x,0,x,2x)&f_4,f_7\\\hline
       &D_5=(0,2x,0,x)&f_8\\\hline
(3x,2x)&D_1=(3x,3x,2x,2x)&f_1\\\hline
       &D_2=(3x,x,x,2x)&f_3\\\hline
       &D_3=(3x,x,x,x)&f_2\\\hline
       &D_4=(2x,x,x,2x)&f_7\\\hline
       &D_5=(2x,x,x,x)&f_5\\\hline
       &D_6=(0,x,x,0)&f_8\\\hline
(5x,3x)&D_1=(5x,5x,3x,3x)&f_1\\\hline
       &D_2=(5x,x,2x,3x)&f_3\\\hline
       &D_3=(5x,x,2x,x)&f_2\\\hline
       &D_4=(4x,x,2x,3x)&f_7\\\hline
       &D_5=(4x,x,2x,x)&f_5\\\hline
       &D_6=(x,x,x,x)&f_6\\\hline
       &D_7=(0,x,2x,0)&f_8\\\hline
\end{array}
\end{equation*}
\end{table}

\subsection{The case $R_1=G_2$}
\label{G:res}

See  \cite[Proposition 4.15]{HO}. There are three orbits of
generic linear residual points $W_0\xi_1$, $W_0\xi_2$ and $W_0\xi_3$.
which we will refer to as $g_1$, $g_2$, and $g_3$.
Let $\alpha_1$ be the simple long root and $\alpha_2$ the simple short
root. Let $k=(k_1,k_2)$ with $k_1$ the parameter of the long root.
The following table lists the $g_i=W_0\xi_i$ and the set
$\mc{K}_i^{reg}$ where $W_0\xi_i$ remains residual upon
specialization. We use similar conventions as in the case $F_4$.
\begin{table}
\caption{Generic linear residual orbits for $G_2$}
\begin{equation*}
\begin{array}{|c|c|c|}
\hline\text{Type}&\xi&\mc{K}^{reg}_\xi\\\hline
g_1&\xi_1=(k_1,k_2)&(k_1+2k_2)(2k_1+3k_2)\not=0\\\hline
g_2&\xi_2=(k_1,k_2-k_1)&(k_1-2k_2)(2k_1-3k_2)\not=0\\\hline
g_3&\xi_3=(k_1,1/2(k_2-k_1))&k_1k_2\not=0\\\hline
\end{array}
\end{equation*}
\end{table}
We list in Table 5 the non-generic values of $k$, together
with the set $\textup{Dyn}^{dist}(k)$ of $k$-weighted Dynkin diagrams
and for each $D\in\textup{Dyn}^{dist}(k)$ the inverse image
$\textup{ev}_{k}^{-1}(D)$ of the map
\begin{equation}
\textup{ev}_{k}:W_0\backslash\textup{Res}_{k}^{lin}\to\textup{Dyn}^{dist}(k)
\end{equation}
\begin{rem}
In Table 5 we assume that $x>0$. Not all special parameters
are listed in table 5 but all other special values can be obtained
from the listed ones by applying the following symmetries.
First of all we have
$g_i(k_1,k_2)=g_i(-k_1,-k_2)$ (since $-\textup{id}\in W_0$) and
$g_i(k_1,k_2)=g_{\theta(i)}(k_1,-k_2)=g_{\theta(i)}(-k_1,k_2)$ with
$\theta=(12)$.
With these transformations we can reach all quadrants of $\mc{K}$
from the positive quadrant. In addition we have used
the following symmetry (arising from interchanging the long and
short roots) to reduce the length of Table 5: Let
$\Psi(a,b)=(3b,a)$. Then we can define $D_i(3k_2,k_1)$ by
$D_i(3k_2,k_1)=\Psi(D_i(k_1,k_2))$. The map $\Psi$ acts as follows on the
set of generic linear residual orbits: $\Psi(f_i(k_1,k_2))=f_i(3k_2,k_1)$.
Observe that $\Psi^2(a,b)=(3a,3b)$, thus $\Psi^2$ corresponds to
replacing $x$ by $3x$.
\end{rem}
\begin{table}
\caption{$k$-weighted Dynkin diagrams and confluence for $G_2$}
\begin{equation*}
\begin{array}{|c|l|c|}
\hline{k=(k_1,k_2)}&D\in\textup{Dyn}^{dist}(k)&{\textup{ev}}_{k}^{-1}(D)\\\hline
(0,x)&D_1=(0,x)&g_1,g_2\\\hline
(x,x)&D_1=(x,x)&g_1\\\hline
     &D_2=(x,0)&g_2,g_3\\\hline
(2x,x)&D_1=(2x,x)&g_1\\\hline
      &D_2=(\frac{1}{2}x,\frac{1}{2}x)&g_3\\\hline
\end{array}
\end{equation*}
\end{table}

\section{The classification of the discrete series of ${\mathbf{H}}$}
\label{sec:dsH}

We formulate the main theorem of this paper.
\begin{thm}\label{thm:dsH}
Let $R_1\subset V^*$ be a non-simply laced irreducible root system or
$R_1=A_n$. Let $F_1$ be a basis of simple roots, and let $k\in\mc{K}$.
We denote by $\Delta^{\mathbf{H}}(R_1,V,F_1,k)$ the set of
irreducible discrete series characters of ${\mathbf{H}}(R_1,V,F_1,k)$.
The generic central character map induces a bijection
\index{gcc@$gcc^{\mathbf{H}}_k$, specialization of $gcc^{\mb H}$ at $k \in \mc K$}
\begin{equation}\label{eq:bijfinal}
gcc^{\mathbf{H}}_k : \Delta^{\mathbf{H}}(R_1,V,F_1,k)\stackrel{\simeq}
{\longrightarrow}W_0\backslash{\textup{Res}}_k^{lin}(R_1)
\end{equation}
which is compatible with the central character map in the sense that
$\textup{ev}_k({gcc}_k^{\mathbf{H}}(\delta))=cc(\delta)$ for all $k\in\mc{K}$
and for all $\delta\in\Delta^{\mathbf{H}}(R_1,V,F_1,k)$, \emph{except} when
$R_1=F_4$ and $k\in\mc{K}^{reg}_{f_8}$, in which case there are exactly
two elements $\delta_{f_8^\prime},\,\delta_{f_8^{\prime\prime}}\in
\Delta^{\mathbf{H}}(R_1,V,F_1,k)$ with generic central character $f_8$.
This statement is also true for $R_1=D_n$ (with $n\geq 4$) if we
replace $W_0(D_n)\backslash{\textup{Res}}_k^{lin}(D_n)$ by
$W_0(D_n)\backslash{\textup{Res}}_k^{lin}(D_n)^\sharp$ and
${gcc}^{\mathbf{H}}_k$ by the map ${gcc}^{\mathbf{H},\sharp}_k$ which
is equal to the map $gcc_{(k,0)}^{{\mathbf{H}},B_n}$ for type $B_n$,
composed with the induction map for characters of
$\mathbf{H}(D_n,V,F_1,k)$ to $\mathbf{H}(B_n,V,F_1^b,(k,0))$.
\end{thm}
\begin{proof}
We apply the reduction results Corollary \ref{cor:dsgradedc} and Corollary
\ref{cor:dsgraded} with $u=1$. In this situation we will denote the natural
map $\mc{Q}\to\mc{K}$ given by $q\to k_{u=1}=k$ by $k=2\log(q)$.

In view of Proposition \ref{lem:redrat}, Corollary \ref{cor:dsgraded} and
Corollary \ref{cor:mult} the result is equivalent to the statement that
for all $W_0\xi\in W_0\backslash{\textup{Res}}_k^{lin}(R_1)$ and all
connected components $U \subset \mc{K}^{reg}_{W_0 \xi}$ we have
$M_{\{W_0\exp(\xi)\}\times\exp(U)}=1$
except when $R_1=F_4$ and $W_0\xi=f_8$, in which case the value
should be $2$ (independent of the choice of $U$).

If $R_1=A_n$ (with $n\geq 1$) then
there is one generic residual orbit $W_0\xi$, with two
components $\mc{K}_{W_0\xi}=\{U_+,U_-\}$. It is of course well known
in this case that
$M_{\pm}:=M_{\{W_0\exp(\xi)\}\times\exp{(U_{\pm})}}=1$
and there are many possible
proofs for this fact, but we will explain the proof that is central to
the approach in this paper in order to illustrate the method in
this basic case.

The multiplicities $M_{\pm}$ are on the one hand at least $1$
(by Corollary \ref{cor:mult}) and on the other hand at most $1$
by Corollary \ref{cor:mult}, Corollary \ref{cor:dsgraded}, and
Corollary \ref{cor:upper}. This proves the required equality.

If $R_1=B_n$ (with $n\geq 2$) we argue in a similar way.
By Corollary \ref{cor:mult} and
Corollary \ref{cor:linresgenB} we see that for all generic $k\in\mc{K}$
the cardinality $|\Delta^{\mathbf{H}}(R_1,V,F_1,k)|\geq |\mc{P}(n)|$
with equality iff $M_{\{W_0\exp(\xi)\}\times\exp(U)}=1$ for all $U$
such that ${k}\in U$. On the other hand it is well known that the
set of elliptic conjugacy classes of $W_0(B_n)$ is naturally in bijection
with the set $\mc{P}(n)$. Hence Corollary \ref{cor:dsgraded} and
Corollary \ref{cor:upper} show that
$|\Delta^{\mathbf{H}}(R_1,V,F_1,k)|\leq |\mc{P}(n)|$. We conclude that
$|\Delta^{\mathbf{H}}(R_1,V,F_1,k)|= |\mc{P}(n)|$ and thus that
$M_{\{W_0\exp(\xi)\}\times \exp(U)}=1$ for all orbits $W_0\xi$ and
all connected components $U \subset \mc{K}^{reg}_{W_0\xi}$ such that $U\ni{k}$.
Since $k$ was chosen arbitrarily we see that
$M_{\{W_0\exp(\xi)\}\times\exp(U)}=1$ for all $W_0\xi$
and all $\mc{C}_{W_0\exp(\xi)}$, as desired.

If $R_1=C_n$ then the result follows easily from the case $R_1=B_n$ using
that fact that ${\mathbf{H}}(B_n,(k_1,k_2))\simeq{\mathbf{H}}(C_n,(k_1,k_2/2))$.

If $R_1=G_2$ the argument is completely analogous to the case $R_1=B_n$,
using the results of subsection \ref{G:res}.

In the case $R_1=F_4$ we need additional arguments. The Weyl group $W_0(F_4)$
has $9$ elliptic conjugacy classes, but by Subsection \ref{F:res} we see that
there are only $8$ generic linear residual
points $f_1,\dots,f_8$. The points $f_1,\dots,f_7$ are (generically) regular.
A generic residual orbit $W_0\exp(\xi(k))$ carries precisely
$1$ irreducible discrete series character (see \cite[Corollary 1.2.11]{SlootenThesis}),
proving that the multiplicities associated to these orbits are all precisely
equal to $1$. Now consider $f_8$.
By the above numerology we see that for any component
$U$ of $\mc{K}^{reg}_{f_8}$ the value of $M_{{f_8}\times U}$ can be either
$1$ or $2$ and in the rest of the proof we will show that it has to be always $2$.
From Table 2 we have $\mc{K}_{f_8}^{reg}=\{U_{\pm,\pm}\}$ with
$U_{\epsilon_1,\epsilon_2}=\{(k_1,k_2)\mid \epsilon_i k_i>0 (i=1,2)\}$.
This simple structure of $\mc{K}_{f_8}^{reg}$ is very helpful at this point.
There exist standard automorphisms (for $\epsilon_i=\pm 1$)
\begin{equation}
\psi_{\epsilon_1,\epsilon_2}:\mathbf{H}(R_1,V,F_1,(k_1,k_2))\to
\mathbf{H}(R_1,V,F_1,(\epsilon_1 k_1,\epsilon_2 k_2))
\end{equation}
such that $\psi_{\epsilon_1,\epsilon_2}(x)=x$ for all $x\in V^*$,
$\psi_{\epsilon_1,\epsilon_2}(s_i)=\epsilon_1 s_i$ (for $i=1,\,2$) and
$\psi_{\epsilon_1,\epsilon_2}(s_j)=\epsilon_2 s_j$ (for $j=3,\,4$).
Clearly  twisting by $\psi_{\epsilon_1,\epsilon_2}$
sends discrete series characters to discrete series characters and
thus that the multiplicities $M_{f_8\times U}$
are independent of $U$. It was shown
by Mark Reeder \cite{Re} that there exist $2$ irreducible discrete
series with central character $\textup{ev}_{(4x,x)}(f_8)$
for the (generic) parameters $(4x,x)$ (with $x>0$).
In Reeder's parametrization these characters are called
$[A_1E_7(a_5),-21]$ and $[A_1E_7(a_5),-3]$. Reeder's
result is based on the explicit computation of the weight
diagrams of the discrete series modules (alternatively
we could invoke here the standard Kazhdan-Lusztig classification
for the parameters $(x,x)$ (with $x>0$) to arrive at the same
conclusion).

Finally let us consider the case $R_1=D_n$. Of course this simply
laced case can be treated directly by the Kazhdan-Lusztig classification
(see Remark \ref{rem:KL})
but we want to show here how to adapt the deformation method to
so that the classification for $R_1=D_n$ is also treated by an
appropriate version of the generic central character map.
It was shown in Subsection \ref{D:res}
that the degenerated affine Hecke algebra $\mathbf{H}(D_n,V,F_1,k)$
is the fixed point algebra of $\mathbf{H}(B_n,V,F_1,(k,0))$ for the action
of the automorphism group $\Psi \cong {\mathbb{Z}/2}$. Our knowledge of
the case $R_1=B_n$ implies that the generic central
character map $gcc_{(k,0)}^{{\mathbf{H}},B_n}$ for type $B_n$
yields a bijection
between $\Delta^{\mathbf{H}}(B_n,V,F_1^b,(k,0))$ and
$W_0\backslash\textup{Res}^{lin}_0(B_n)$.
In Subsection \ref{D:res} we have seen that twisting by $\psi$ acts
freely on the set of generic linear residual orbits
$W_0\backslash\textup{Res}^{lin}_0(B_n)$. It follows that twisting by $\psi$
acts freely on $\Delta^{\mathbf{H}}(B_n,V,F_1^b,(k,0))$ as well.
Using \cite[Theorem A.6, Theorem A.13]{RamRam} we see that all characters in
$\Delta^{\mathbf{H}}(B_n,V,F_1^b,(k,0))$ remain irreducible when restricted
to $\mathbf{H}(D_n,V,F_1,k)=\mathbf{H}(B_n,V,F_1^b,(k,0))^\Psi$, that
all $\delta\in\Delta^{\mathbf{H}}(D_n,V,F_1,k)$ arise in this way and
that there always exist precisely two irreducible characters
$\delta_+,\,\delta_-\in\Delta^{\mathbf{H}}(B_n,V,F_1^b,(k,0))$
restricting to $\delta$, and these two characters are $\psi$-twists
of each other. This proves the required result.
\end{proof}
Let us look at an interesting special case:
\begin{ex}\label{ex:k10}
We have
$\mathbf{H}(B_n,V,F_1,(0,k_2))\simeq \mathbf{H}(A_1^n,V,F_1(A_1^n),k_2)\rtimes S_n$
with $F_1^{A}=\{e_1,\dots,e_n\}$. Using this it is easy to see that for $k_2\not=0$
\begin{equation}
\Delta^{\mathbf{H}}(B_n,V,F_1,(0,k_2))=\{\delta_\pi\mid \pi\in\widehat{S_n}\}
\end{equation}
with $\delta_\pi=\delta^{\otimes n}\otimes\pi$ and
where $\delta$ is the unique irreducible (one dimensional) discrete series
character of $\mathbf{H}(A_1,V(A_1),F_1(A_1),k_2)$. If $k_2>0$ then
\begin{equation}
\delta_{\pi(\lambda)}|_{W_0}=\chi(-,\lambda^\prime)
\end{equation}
and if $k_2<0$ then
\begin{equation}
\delta_{\pi(\lambda)}|_{W_0}=\chi(\lambda,-)
\end{equation}
where $\{\pi(\lambda)\}_{\lambda\in\mc{P}(n)}$ denotes the usual
parametrization of the irreducible characters of $S_n$ by partitions
of $n$ (see e.g. \cite{C}), and where
$\{\chi(\tau,\sigma)\}_{(\tau,\sigma)\in\mc{P}(2,n)}$
is the usual parametrization of the irreducible characters of $W_0=W(B_n)$
by bipartitions of $n$.

On the other hand we recall from subsection \ref{subsub:k10} that $k=(0,k_2)$
is a regular parameter for all generic linear residual orbits of
$\mathbf{H}(B_n,V,F_1,(k_1,k_2))$. Hence the map
\begin{equation}
gcc_{(0,k_2)}:\Delta^{\mathbf{H}}(B_n,V,F_1,(0,k_2))\to
W_0\backslash\textup{Res}^{lin}(B_n)
\end{equation}
is a bijection by Theorem \ref{thm:dsH}.
By continuity (see Theorem \ref{thm:main1} and Definition \ref{denf:gends})
it follows that for all $\lambda\in\mc{P}(n)$ the
generic irreducible discrete series character $\delta_{W_0\xi_\lambda\times U_{\pm\infty}}$
whose domain of definition is the unique connected component
$U_{\pm\infty}=U_{W_0\xi_\lambda,\pm\infty}$
of $\mc{K}^{reg}_{W_0\xi_\lambda}$ which contains $(0,k_2)$ for $\pm k_2>0$ restricts
to an
irreducible character of $S_n$, and this sets up a bijective correspondence between
the set of generic linear residual orbits and the set of irreducible characters of
$S_n$.
\end{ex}
\begin{rem}
Unfortunately we do not know how to compute the generic central character map
in this case. We conjecture that
\begin{equation}
gcc_{(0,k_2)}(\delta_{\pi(\lambda)})=W_0\xi_{\lambda^\prime}
\end{equation}
if $k_2>0$ and
\begin{equation}
gcc_{(0,k_2)}(\delta_{\pi(\lambda)})=W_0\xi_{\lambda}
\end{equation}
if $k_2<0$.
\end{rem}
The following corollary of Theorem \ref{thm:dsH} was known
for degenerate affine Hecke algebras with equal parameters by
the work of Reeder \cite{Ree}.

\begin{cor}\label{cor:ell}
Let $k\in\mc{K}^{reg}$ be a regular parameter.
The elliptic pairing (see page \pageref{cor:ON})  is positive definite
on $\textup{Ell}(\mathbf{H}(R_1,V,F_1,k))$ and the map
\begin{align*}
\textup{Ell}(\mathbf{H}(R_1,V,F_1,k))&\to\textup{Ell}(W_0)\\
[\pi]&\to [\pi|_{W_0}]
\end{align*}
yields an isometric isomorphism with respect to the elliptic pairing.
\end{cor}
\begin{proof}
We may assume that $R_1$ is irreducible.
If $R_1$ is non-simply laced we see from our results
above that (since $k\in\mc{K}^{reg}$) the images
in $\textup{Ell}(\mc{H}(R_1,V,F_1,k))$ of the irreducible
characters in $\Delta^{\mathbf{H}}(R_1,V,F_1,k)$ form a linear
basis of $\textup{Ell}(\mc{H}(R_1,V,F_1,k))$. We also know that
these even form an orthonormal basis with respect to the elliptic
pairing, hence the elliptic pairing is positive definite in this case.
Using results of \cite{OpdSol} it follows that the limits
of these characters for $x k$ (with $x\to 0$) from an orthonormal set
of elliptic characters of $W_0$ (actually, in order to see this
using the results of \cite{OpdSol} we need to lift the characters to
$\mc{H}(\mc{R},q)$ using the equivalence of Corollary \ref{cor:dsgraded},
then take the limit $q^x$ with $x\to 0$ to get a set of orthonormal
elliptic characters for $W$, and then use the formula for
the elliptic paring of \cite[Theorem 3.2]{OpdSol}).
Finally we already established in the previous theorem that the
cardinality of this set is equal to the dimension
$\textup{ell}(W_0)$ of the space $\textup{Ell}(W_0)$.
This yields the desired result for non-simply laced cases.
For simply laced cases (or more generally all cases with equal
parameters $k$ (i.e. such that $k_\alpha=x$ for all $\alpha\in R_1$)
the result is due to Reeder \cite{Ree} (based on the Kazhdan-Lusztig
model for the characters of $\mc{H}(\mc{R},q)$).
\end{proof}

It is natural to expect that the result of Corollary \ref{cor:ell}
holds for arbitrary $k$. We conjecture something stronger (see
\cite{ABP} for related conjectures):
\begin{conj} A generic family $\delta$ of
irreducible discrete series characters
$\delta\in\Delta^{\mathbf{H},gen}(R_1,V,F_1)$
with domain of definition $U\in\mc{K}^{reg}_{W_0\xi}$ say,
has weakly continuous limits to the points $k\in\overline{U}$
(the closure of $U$). In view of the above
results this would imply that the elliptic pairing is positive definite
on $\textup{Ell}(\mathbf{H}(R_1,V,F_1,k))$ for all semisimple root systems
$R_1$ and all $k\in\mc{K}$, and that this space is
isometric to $\textup{Ell}(W_0)$ for all $k\in\mc{K}$.
\end{conj}
\begin{rem} Using the $gcc^{\mathbf{H}}$ invariant is not difficult to check
that for all irreducible root systems $R_1$ the irreducible discrete series
characters are stable for twisting by diagram automorphisms
(a case-by-case verification).
\end{rem}

\section{The classification of the discrete series of $\mc{H}$}
\label{sec:ds}

Since a semisimple root datum is in general not isomorphic to a direct sums of
irreducible root data the classification of the irreducible discrete series
characters can not be reduced to the same problem for an irreducible root datum.
However, we have seen (Theorem \ref{thm:red1} and Theorem \ref{thm:red2}) how to
reduce the problem to the analogous problem for crossed products of semisimple
degenerate affine algebras by certain groups of diagram automorphisms. In
Section \ref{thm:dsH} we have covered the basic building blocks, the
simple degenerate affine Hecke algebras.

Even though the classification problem for semisimple affine Hecke algebras can
in general not be reduced to the simple cases it is instructive to give the classification
in certain basic situations. This is what we seek to do in the present section.
In particular we classify in this section the irreducible discrete series characters
for all the irreducible non-simply laced root data and all possible positive
root labels (using Theorem \ref{thm:red1} and Theorem \ref{thm:red2} to reduce the
problem to Theorem \ref{thm:dsH}).

Let $\mc{R}=(X,R_0,Y,R_0^\vee,F_0)$ be an \emph{irreducible} root datum,
and let $q\in\mc{Q}=\mc{Q}(\mc{R})$. Recall the maximal root datum
$\mc{R}^{max}$ (with $X^{max}=P(R_1)$, the weight lattice of $R_1$,
and $R_0^{max}=R_0$) with the natural isogeny $\psi:\mc{R}\to\mc{R}^{max}$
such that $\mc{Q}(\mc{R})=\mc{Q}(\mc{R}^{max})$. Let us define
\begin{equation}
\Gamma=Y/Q(R_1^\vee) \cong
\textup{Hom}(X^{max}/X,\mathbb{C}^\times)\subset T^{max}
\end{equation}
An element $\gamma\in\Gamma$ uniquely extends to a linear character
(also denoted $\gamma$) of $W^{max}=X^{max}\rtimes W_0$ which is trivial
on $W_0$. $\Gamma$ acts on the affine Hecke algebra
$\mc{H}^{max}=\mc{H}(\mc{R}^{max},q)$ by means of algebra isomorphisms
as follows: for $w\in W^{max}$ and $\gamma\in \Gamma$ we define
$\gamma(N_w)=\gamma(w)N_w$. With this action of $\Gamma$ we have
\begin{equation}
\mc{H}(\mc{R},q)=\mc{H}(\mc{R}^{max},q)^\Gamma
\end{equation}

We are interested in applying Theorem \ref{thm:red1} to central
characters which carry discrete series characters of
$\mc{H}$, in other words to orbits $W_0r\in\textup{Res}(\mc{R},q)$
of residual points in $T$. We know that $r\in T$ is of the form
$r=s\exp(\xi)$ with $s\in T_u$ such that
\begin{equation}
R_{s,1}=\{\alpha\in R_1\mid\alpha(s)=1\}
\end{equation}
is of maximal rank, and $\xi$ is a linear
$(R_{s,1},k_s)$-residual point.
Let us define $W^\vee=W_0\ltimes 2\pi i Y$, then
the action groupoid of the action of $W_0$ on $T$ is
equivalent to the action groupoid of $W^\vee$ acting on $iV$.
We have a splitting of the form
\begin{equation}
W^\vee=W^\vee(\mc{R}^{max})\rtimes\Gamma
\end{equation}
with $W^\vee(\mc{R}^{max})=W(R_1^{(1)})=W_0\ltimes 2\pi iQ(R_1^\vee)$
on $iV$, and where $\Gamma$ acts on $W(R_1^{(1)})$
via diagram automorphisms of $R_1^{(1)}$.
Hence we may assume that $s(e)=\exp(e)$ with $e\in E(C^\vee)$, the
set of extremal points of the closure of the fundamental alcove
$\overline{C^\vee}$ of $W(R_1^{(1)})$. It follows that
\begin{equation}
W_{s(e)} \cong W(R_{s(e),1})\rtimes \Gamma_{s(e)}
\end{equation}
with $\Gamma_{s(e)} \cong \{\gamma\in\Gamma\mid \gamma(e)=e\}$
(compare with Definition \ref{defn:iso} and Corollary \ref{cor:genreds}).

Let $F^\vee$ be the set of simple affine roots of $R_1^{(1)}$.
If $a^\vee\in F^\vee$ then there exists a unique extremal point
$e(a^\vee)\in E(C^\vee)$ such that $a^\vee(e(a^\vee))\not=0$.
This sets up a canonical bijection $F^\vee\longleftrightarrow E(C^\vee)$
which we denote by $e\to a^\vee(e)$ and $a^\vee\to e(a^\vee)$.

Let $D(a^\vee)\in V^*$ denote the gradient of $a^\vee$. By the above,
if $e\not=e(a^\vee)$ then $D(a^\vee)(s(e))=1$.
Hence if $D(a^\vee)$ can be written as $D(a^\vee)=2\beta$ with
$\beta\in R_0$ then $\beta(s(e))=\pm 1$ for all extremal points
$e\in\overline{C^\vee}$ with $e\not=e(a^\vee)$. In this situation
the value $\beta(s(e))\in\{\pm 1\}$ is independent of the choice
of $e\not=e(a^\vee)$ (namely, it equals $-1$ iff
$\{a^\vee\}=F^\vee\backslash F_1$). Thus the following definition
makes sense (in view of (\ref{eq:k})):
\begin{defn}
We define the spectral diagram $\Sigma$ associated with $(\mc{R},q)$ as the
affine Dynkin diagram of $W^\vee$ associated with the basis $F^\vee$
of $R_1^{(1)}$, where we give all the vertices $a^\vee\in F^\vee$ of $\Sigma$
a weight $k_{a^\vee}$ defined as follows.
We define $k_{a^\vee}=k_{s,D(a^\vee)}$ (as in (\ref{eq:k})) where
$s=s(e)$ for $e\in E(C^\vee)\backslash\{e(a^\vee)\}$ (an arbitrary choice).
Note that $\Sigma$ (labelled with these weights) is invariant for
the natural action of $\Gamma$ on $F^\vee$. We include the
action of $\Gamma$ on the diagram and the marking of the special
vertex (extending the diagram of $R_1$) in the spectral diagram.
\end{defn}
\begin{ex}
If $\mc{R}=\mc{R}^{max}$ we have $\Gamma=1$. These cases
are referred to as $R_1^{(1)}$.
\end{ex}
\begin{ex}\label{rem:cn2x}
It is possible that the generic affine Hecke algebra
of a root datum is a specialization of the generic affine Hecke
algebra of another root datum.
For example, $\mc{H}(C_n,P(C_n),B_n,Q(B_n),F_0(C_n))$ is
isomorphic to the specialization $v_{\beta^\vee}=1$ in the
generic algebra of the type $\mc{H}(B_n,Q(B_n),C_n,P(C_n),F_0(B_n))$
where $\beta\in R_0=B_n$ is such that $2\beta\in R_1$. This is
compatible with the previous remark in the sense
that both these cases are referred to as $C_n^{(1)}$.
A basic example in this class is the Iwahori Hecke algebra of
the Chevalley group of type $G=SO_{2n+1}(F)$, with
$q^2=|\mathcal{O}/\mathcal{P}|$, the cardinality of the
residue field. See Figure \ref{soodd} (with $k=2\log(q)$).
\input{soodd.TpX}
\end{ex}
\begin{ex}
The Iwahori Hecke algebra of the simply connected group
$Sp_{2n}(F)$ (where we put $q^2=|\mathcal{O}/\mathcal{P}|$)
has the spectral diagram displayed in Figure \ref{sp2n}
(where $k=2\log(q)$). It corresponds to the case
$R_0=B_n$, $X=Q(R_0)$ and therefore it is obviously also a
specialization of $C_n^{(1)}$ (namely this case corresponds to the
specialization  $v_{\alpha^\vee}=1$ for $\alpha=2\beta$
with $\beta\in R_0$).
\input{sp2n.TpX}
Indeed, the spectral diagram of Figure \ref{sp2n}
is equivalent to the diagram of
type $C_n^{(1)}$ displayed in Figure \ref{sp2nequiv}.
\input{sp2nequiv.TpX}
\end{ex}
\begin{ex}\label{ex:3}
More generally, let $\mc{R}$ be of type
$C_n^{(1)}$. Let $R_0=\{\pm e_i,\,\pm e_i\pm e_j\}$ and put
$X=Q(R_0)$.
Choose $F_0=\{e_1-e_2,\dots,e_{n-1}-e_n,e_n\}$ and put
$q_1=q(s_{x_i-x_{i+1}})$, $q_2=q(s_{2x_n})$ and
$q_0=q(s_{1-2x_1})$. Put $k=2\log(q_1)$ and define $m_{\pm}$ by
$m_{\pm}k=\pm\log(q_0)+\log(q_2)$. The corresponding spectral
diagram is displayed in Figure \ref{genericcn}. We refer
to \cite{Lu4}, \cite{Blo} for explicit examples of such affine
Hecke algebras as convolution algebras in the
representation theory of $p$-adic groups.
\input{genericcn.TpX}
\end{ex}
\begin{defn}
For each element $e\in E(C^\vee)$ we associate the semisimple
root system $R_{s(e),1}$ with basis $F_{s(e),1}$ (as in Definition
\ref{defn:iso}). Then $D(F^\vee\backslash\{a^\vee(e)\})$ is a basis
for $R_{s(e),1}$.
Let $k_e\in\mc{K}(R_{s(e),1})$ denote the unique parameter function
on $R_{s(e),1}$ which corresponds to the set of weights
of $\Sigma$ restricted to $F^\vee\backslash\{a^\vee(e)\}$.
Then we associate to $e$ the algebra
\begin{equation}
\mathbf{H}_e:=\mathbf{H}(R_{s(e),1},V,F_{s(e),1},k_e)\rtimes\Gamma_{s(e)}
\end{equation}
We denote by $\Delta(\mathbf{H}_e)$ the set of irreducible discrete series characters
of $\mathbf{H}_e$ (in the sense as explained in the text following
Corollary \ref{cor:Dss}).
\end{defn}
\index{He@$\mb H_e$, extended degenerate affine Hecke algebra}
\index{1dHe@$\Delta(\mathbf{H}_e)$, set of irreducible discrete series characters
of $\mathbf{H}_e$}
Let us finally formulate our classification theorem:
\begin{thm}\label{thm:clas}
Let $\mc{R}=(X,R_0,Y,R_0^\vee,F_0)$  be a root datum with
$R_0$ irreducible, and let $q\in \mc{Q}$. Let $\Delta(\mc{R},q)$
be the set of irreducible discrete series characters of
the Hecke algebra $\mc{H}(\mc{R},q)$ as usual.
There exists a natural bijection
\begin{equation}
\Delta(\mc{R},q) \longleftrightarrow\coprod_{e\in\Gamma\backslash E(C^\vee)}
\Delta^{s(e)}(\mc{R},q)
\end{equation}
where the disjoint union is taken over a set of
representatives for the $\Gamma$-action on $E(C^\vee)$.
For each $e\in E(C^\vee)$ there is a natural bijection
\begin{equation}\label{eq:redbij}
\Delta^{s(e)}(\mc{R},q)\simeq\Delta(\mathbf{H}_e)
\end{equation}
(where the right hand side denotes the set of irreducible discrete series
characters of $\mathbf{H}_e$). In particular, if $\Gamma_{s(e)}=1$ we have
\begin{equation}\label{eq:redbijna}
\Delta^{s(e)}(\mc{R},q)\simeq\Delta^{\mathbf{H}}(R_{s(e),1},V,F_{s(e),1},k_e)
\end{equation}
(which is completely described by Theorem \ref{thm:dsH}).
If $\delta^{\mathbf{H}}\in\Delta(\mathbf{H}_e)$ then its
restriction to ${\mathbf{H}}(R_{s(e),1},V,F_{s(e),1},k_e)$ is a finite sum
of irreducible discrete series characters $\delta_i^{\mathbf{H}}$ whose generic
central characters $gcc^{\mathbf{H}}(\delta_i^{\mathbf{H}})$ constitute
one $W_{s(e)}$-orbit of a generic linear $R_{s(e),1}$-residual
point $\xi$ (using Theorem \ref{thm:dsH}). We express this by writing
\begin{equation}
gcc^{\mathbf{H}}(\delta^{\mathbf{H}})=W_{s(e)}\xi
\end{equation}
With this notation the bijection above has the property that if
$\delta\in\Delta^{s(e)}(\mc{R},q)$ corresponds to
$\delta^{\mathbf{H}}\in\Delta(\mathbf{H}_e)$ with
$gcc^{\mathbf{H}}(\delta^{\mathbf{H}})=W_{s(e)}\xi$ then
\begin{equation}\label{eq:corrcc}
gcc(\delta)=W_0(s(e)\exp(\xi))
\end{equation}
\end{thm}
\begin{proof}
Use Theorem \ref{thm:red1} and Theorem \ref{thm:red2}.
\end{proof}
\begin{rem}
If $\mc{R}$ is of type $R_1^{(1)}$ then one has $\Gamma_{s(e)}=1$
for all $\in E(C^\vee)$.
In general one needs to apply Clifford theory in order to describe
the sets $\Delta(\mathbf{H}_e)$ in terms of the results of
Theorem \ref{thm:dsH}.
\end{rem}
The only non-simply laced classical case which is not of type
$R_1^{(1)}$ is the case $R_0=C_n$ and $X=Q(R_0)$ (as is clear from the
examples above).
In this case $\mc{R}^{max}$ is of $C_n^{(1)}$-type with the specialization
$v_{\beta^\vee}=v_{2x_n}=1$ (as in Example \ref{rem:cn2x}).
Using the notation of Example
\ref{ex:3} and (\ref{eq:shift}) we see that $q_0=q(v_{2x_n})=1$.
Hence we have $m=m_+=m_-$, and a group $\Gamma \cong \mathbb{Z}/2$
acting on the spectral diagram $\Sigma$ as shown in Figure \ref{sosym}.
\input{sosym.TpX}
In the application of Theorem \ref{thm:clas} everything is
straightforward except when $n=2a$ is even and $e=e_a$ corresponds to the
the middle node of $\Sigma$ (the unique node of $\Sigma$ with
nontrivial isotropy in $\Gamma$). In this case we need to describe the set
\begin{equation}\label{e}
\Delta(\mathbf{H}_{e_a})=
\Delta((\mathbf{H}(C_a,V_a,F_a,k_a)\otimes\mathbf{H}(C_a,V_a,F_a,k_a))\rtimes\Gamma)
\end{equation}
where the nontrivial element of $\Gamma$ acts by the flip
$\tau$ of the two tensor legs.
\begin{thm}
We have
\begin{equation}
\Delta(\mathbf{H}_{e_a})\simeq
\Gamma\backslash(\Delta^{\mathbf{H}}(C_a,V_a,F_a,k_a)\times
\Delta^{\mathbf{H}}(C_a,V_a,F_a,k_a))^{\bullet}
\end{equation}
where for any set $A$, $(A\times A)^\bullet$ denotes the Cartesian product of
$A$ with itself with the diagonal counted twice, and where the unique nontrivial
element $\gamma\in\Gamma$ acts by $\pi(\gamma)(\delta_1,\delta_2)=(\delta_2,\delta_1)$.
\end{thm}
\begin{proof}
By Clifford theory it is clear that
all irreducible discrete series representations of $\mathbf{H}_{e_a}$
are obtained by the following recipe. We start from an irreducible discrete
series character $\delta=\delta_1\otimes\delta_2$ of
$\mathbf{H}(C_a,V_a,F_a,k_a)\otimes{\mathbf{H}}(C_a,V_a,F_a,k_a)$.
Consider its inertia group for the action of $\Gamma$ on such characters
(by twisting). In this simple situation we see that we can choose
an explicit intertwining isomorphism
\begin{equation}
\pi(\gamma):\delta_1\otimes\delta_2\to(\delta_2\otimes\delta_1)\circ\tau
\end{equation}
given by $\pi(\gamma)(v\otimes w)=w\otimes v$. Hence the
inertia subgroup in $\Gamma$ of $\delta_1\otimes\delta_2$
is nontrivial iff $\delta_1$ and $\delta_2$ are equivalent
irreducible representations.
If the inertia is trivial then Clifford theory tells us that
the induction of $\delta_1\otimes\delta_2$ to $\mathbf{H}_{e_a}$
is irreducible, and otherwise Clifford theory tells us that the induced
representation splits up in two inequivalent irreducible parts (distinguished
from each other by the sign of the trace of $\gamma$).
This proves the result.
\end{proof}

\appendix

\section{Analytic properties of the Schwartz algebra}

The aim of this appendix is to provide proofs of Theorems \ref{thm:frealg} and
\ref{thm:qcont}, which concern
the embedding $\mc S (\mc R ,q) \to C^*_r (\mc H (\mc R ,q))$ and holomorphic functional
calculus with varying parameters $q$. Our approach is purely analytic and does not make
any use of the representation theory of $\mc H (\mc R ,q)$. The appendix is based on the
second author's thesis \cite[Section 5.2]{Sol}, where some proofs can be found in more
detail.

First we recall some generalities. A Fr\'echet algebra is a Fr\'echet space endowed with
a jointly continuous multiplication. We include in the definition that the topology can
be defined by a (countable) family of submultiplicative seminorms. The submultiplicativity
ensures that our Fr\'echet algebras can be written as projective limits of Banach algebras.

\begin{thm}\label{thm:A.contour}
Let $A$ be a unital Fr\'echet algebra and let $a \in A$. Suppose that $U \subset \mh C$ is
an open neighborhood of the spectrum $\textup{Sp}(a)$ of $a$, and let $C^{an}(U)$
be the algebra of holomorphic functions on $U$. There exists a unique continuous algebra
homomorphism, the holomorphic functional calculus
\[
C^{an}(U) \to A : f \mapsto f(a) ,
\]
such that $1 \mapsto 1$ and $\mr{id}_U \mapsto a$.
Moreover, if $\Gamma$ is a positively oriented smooth simple closed contour, which lies inside
$U$ and encircles $\textup{Sp}(a)$, then
\[
f(a) = \frac{1}{2 \pi i} \int_\Gamma f(z) (z-a)^{-1} \textup{d} z .
\]
\end{thm}
\begin{proof}
This is well-known for Banach algebras, see for example \cite[Proposition 2.7]{Tak}.
As noticed in \cite[Lemma 1.3]{Phi}, we can generalize the result to $A$, because $A$
is a projective limit of Banach algebras.
\end{proof}

\begin{rem}
If $U$ is disconnected then we may also use finitely many contours $\Gamma$, each one
lying in a different connected component of $U$. Notice however that in general Fr\'echet
algebras the spectrum of an element need not be compact, so it may not be possible
to find suitable contours for the holomorphic functional calculus.
\end{rem}

The next Theorem, which relies on a result of Lusztig, is essential to
control the multiplication in $\mc H (\mc R ,q)$. Let $u,v \in W$ and let
$u = \omega s_1 \cdots s_{l (u)}$ be a reduced expression, where $l (\omega ) = 0$
and $s_i \in S$. The $s_i$ need not all be different. For $I \subset \{ 1,2, \ldots ,l (u) \}$
we write $\eta_I = \prod_{i \in I} \big( q(s_i ) - q(s_i )^{-1} \big)$ and
\[
u_I = \omega \tilde s_1 \cdots \tilde s_{l (u)}
\qquad \mr{where} \qquad
\tilde s_i = \left\{ \begin{array}{ccc}
s_i & \mr{if} & i \notin I\\
e   & \mr{if} & i \in I
\end{array} \right..
\]
\begin{thm}\label{thm:A.multH}
\[
N_u \cdot N_v = \sum_{I \subset \{1,2, \ldots ,l (u)\}}
\eta_I D_v^u (I) N_{u_I v} \qquad \mr{where}
\]
\begin{itemize}
\item $D_v^u (I)$ is either 0 or 1,
\item $D_u^v (\emptyset) = 1$ and $D_v^u (I) = 0$ if $|I| > |R_0^+ |$,
\item $\sum\limits_{I \subset \{ 1,2, \ldots ,l (u) \} }
D_v^u (I) < 3 (l (u) + 1)^{|R_0^+ |}$.
\end{itemize}
\end{thm}
\begin{proof}
It follows from the multiplicaton rules in Definition \ref{def:Heckealg} that for $s \in S$
\begin{equation*}
N_s \cdot N_v = N_{s v} + D_v^s \big( q(s) - q(s )^{-1} \big) N_v
\text{  where  } D_v^s  =
\left\{ \begin{array}{ccc}
0 & \mr{if} & l (s v) > l (v) \\
1 & \mr{if} & l (s v) < l (v)
\end{array} \right..
\end{equation*}
The expression for $N_u \cdot N_v$, with $D_u^v (I)$ being 0 or 1 and
$D_u^v (\emptyset) = 1$, follows from this, with induction to $l (u)$.
By \cite[Theorem 7.2]{Lus1} for fixed $w \in W$ the sum
$\sum_{I : u_I = w} \eta_I D_v^u (I)$
is a polynomial of degree at most $|R_0^+ |$ in the variables $q(s_i ) - q(s_i )^{-1}$.
Therefore $D_v^u (I) = 0$ whenever $|I| > |R_0^+ |$ and
\begin{multline}\label{eq:A.1}
\sum\limits_{I \subset \{ 1,2, \ldots ,l (u) \}} D_v^u (I)
\;\; \leq \;\; \# \big\{ I \subset \{ 1,2, \ldots ,l (u) \}
: |I| \leq |R_0^+ | \big\} \;\; \leq \;\; \sum\limits_{j=0
}^{|R_0^+ |} \left( \ds{l (u) \atop j} \right) \\
\leq \;\; {\ds \frac{l (u)!}{\big( l (u) - |R_0^+ | \big) !}}
\sum\limits_{j=0}^{|R_0^+ |} {\ds \frac{1}{j!}} \;\; < \;\;
3 (l (u) + 1)^{|R_0^+ |} ,
\end{multline}
where we should interpret $\big( l (u) - |R_0^+ | \big) !$ as 1 if $|R_0^+ | \geq l (u)$.
\end{proof}

For the reader's convenience we repeat some notations from Section \ref{subsec:Schwartz}
and we add some new ones. The vector space $V^* = \mh R \otimes_{\mh Z} X$ decomposes as
\[
V^* = V_0^* \oplus V_Z^* = \mh R R_0 \oplus \mh R \otimes_{\mh Z} Z ,
\]
so that we can write unambigously $V^* \ni \phi = \phi_0 + \phi_Z \in V_0^* \oplus V_Z^*$.
The norm on $W$ is defined by
$\mc N (w) = l (w) + \norm{x_Z}$ if $w = x w_0$ with $x \in X , w_0 \in W_0$.
Since $X_Z := \{ x_Z : x \in X \}$ is a lattice in $V_Z^*$, we can adjust the norm on $V^*$
so that it takes integral values on $X_Z$. This is not necessary, but it assures that
$\mc N (W) \subset \mh Z_{\geq 0}$.

\begin{lem}\label{lem:A.growthW}
There exists a real number $C_{\mc N}$ such that for all $n \in \mh Z_{\geq 0}$ :
\[
\# \{ w \in W : \mc N (w) = n \} < C_{\mc N} (n+1)^{\mr{rk}(X) - 1} .
\]
\end{lem}
\begin{proof}
Recall that $W = W_0 \ltimes X$ with $W_0$ finite. It is easily seen that $X$ possesses the
required property, and taking the semidirect product with a finite group does not disturb this.
\end{proof}

For $n \in \mh R$ we have a norm $p_n$ on $\mc H (\mc R ,q)$, defined as
\[
p_n \Big( {\ts \sum_{w \in W}} h_w N_w \Big) = \sup_{w \in W} |h_w | (\mc N (w) + 1)^n .
\]
The Schwartz algebra $\mc S (\mc R,q)$ is defined as the completion of $\mc H (\mc R,q)$
with respect to the family of (semi-)norms $p_n ,\, n \in \mh Z_{\geq 0}$. It clearly is a
Fr\'echet space (even a Schwartz space), but it is not so obvious that the multiplication
extends continuously from $\mc H (\mc R ,q)$ to $\mc S (\mc R,q)$; we will prove this later
in the appendix.

Let $L^2 (W)$ be the Hilbert space of square-integrable functions $W \to \mh C$ and let
$\mc S (W)$ be the Fr\'echet space of rapidly decaying functions $W \to \mh C$. We regard
these as topological vector spaces without a specific multiplication. By means of the bases
$\{ N_w : w \in W\}$ we can identify $L^2 (W)$ with $L^2 (\mc H (\mc R ,q))$ and $\mc S (W)$
with $\mc S (\mc R ,q)$. We note that $* ,\tau$ and the $p_n$ do not depend on $q \in \mc Q$,
so they are well-defined on $L^2 (W)$.
For $q \in \mc Q$ and $x \in L^2 (W)$ we denote the corresponding element
of $L^2 (\mc H (\mc R ,q))$ by $(x,q)$. To distinguish the products for various parameters
we add a subscript $q$, thus $(x,q) \cdot (y,q) = (x \cdot_q y ,q)$.

We realize the left regular representation $\lambda$ of $C^* (\mc H (\mc R,q))$ on $L^2 (W)$
and we abbreviate $\norm{(x,q)}_o := \norm{\lambda (x,q)}_{B(L^2 (W))}$. Furthermore let
$\norm{?}_\tau$ be the norm of $L^2 (W)$, so that $\norm{x}_\tau^2 = \tau (x^* \cdot_q x)$
for all $x \in L^2 (W)$ and $q \in \mc Q$ such that $x^* \cdot_q x$ is well-defined. Since the
number $q(s) - q(s)^{-1}$ appears often in the multiplication table of $\mc H (\mc R ,q)$
we will use the following metric on $\mc Q$:
\[
\rho (q,q') := \max_{s \in S}
\big| \big( q(s) - q(s)^{-1} \big) - \big( q'(s) - q'(s)^{-1} \big) \big| .
\]
Put $b := \mr{rk}(X) + 1$. By Lemma \ref{lem:A.growthW} the following sum converges:
\[
\sum_{w \in W} (\mc N (w) + 1)^{-b} <
\sum_{n=0}^\infty C_{\mc N} (n+1)^{\mr{rk}(X) - 1}
(n+1)^{-\mr{rk}(X) - 1} = C_{\mc N} \sum_{n=0}^\infty (n+1)^{-2} < \infty .
\]
Hence we may write $C_b := \sum_{w \in W} (\mc N (w) + 1)^{-b} \in \mh R$.
For all $x = \sum_{u \in W} x_u N_u \in \mc S (W)$ and $n \in \mh Z_{\geq 0}$ we get
\begin{equation}\label{eq:A.2}
\sum_u |x_u| (\mc N (u) + 1)^n \leq \sum_u \sup_v \big\{ |x_v| (\mc N (v) + 1)^{n+b} \big\}
(\mc N (u) + 1 )^{-b} = C_b p_{n+b} (x) .
\end{equation}
Define the parameter function $q^0 \in \mc Q$ by $q^0 (s) = 1$ for all $s \in S$.
Fix $\eta > 0$ and let $B_\rho (q^0 ,\eta )$ be the corresponding closed ball in $\mc Q$.
To estimate some operator norms, we will use the number
$C_\eta := 3 C_b \max \big\{ 1,\eta^{|R_0^+ |} \big\}$ .

\begin{prop}\label{prop:A.estopnorm}
For all $q,q' \in B_\rho (q^0 ,\eta ) ,\, x \in \mc S (W)$ the following estimates hold:
\[
\begin{array}{lrr}
\norm{\lambda (x,q)}_{B (L^2 (W))} \;=\; \norm{(x,q)}_o &
\leq & C_\eta p_{b + |R_0^+ |} (x) , \\
\norm{\lambda (x,q) - \lambda (x,q')}_{B (L^2 (W))} &
\leq & \rho (q,q') C_\eta p_{b + |R_0^+ |} (x) .
\end{array}
\]
In particular $\mc S (\mc R, q)$ is continuously embedded in $C_r^* (\mc H(\mc R ,q))$.
\end{prop}
\begin{proof}
Let $y = \sum_v y_v N_v \in L^2 (W)$. By \eqref{eq:A.2} and Theorem \ref{thm:A.multH}
\begin{equation*}
\begin{aligned}
\norm{x \cdot_q y}_\tau & = \big\| \sum_{u,v} x_u y_v N_u
\cdot_q N_v \big\|_\tau \\
& = \big\| \sum_{u,v} x_u y_v \sum_I  \eta_I D_v^u (I)
N_{u_I v} \big\|_\tau \\
& \leq \sum_u |x_u | \sum_{I : |I| \leq |R_0^+ |} |\eta_I |
\big\| \sum_v |y_v | N_{u_I v} \big\|_\tau \\
& \leq \sum_u |x_u | (l (u) + 1)^{|R_0^+ |}
3 \max \big\{ 1,\eta^{|R_0^+ |} \big\} \norm{y}_\tau \\
& \leq C_\eta p_{b + |R_0^+ |} (x) \norm{y}_\tau .
\end{aligned}
\end{equation*}
By the very definition of the operator norm on $B(L^2 (W))$ this
yields the first estimate. That in turn proves that $\mc S (\mc R
,q)$ is contained in $C_r^* (\mc R ,q)$ and that the inclusion map
is continuous.
\begin{equation*}
\begin{aligned}
\norm{x \cdot_q y - x \cdot_{q'} y}_\tau & = \big\| \sum_{u,v}
x_u y_v (N_u \cdot_q N_v - N_u \cdot_{q'} N_v) \big\|_\tau \\
& = \big\| \sum_{u,v} x_u y_v \sum_I (\eta_I -\eta'_I ) D_v^u (I)
N_{u_I v} \big\|_\tau \\
& \leq \big\| \sum_{u,v} x_u y_v \sum_I \rho (q,q') |I|
\eta^{|I| - 1} D_v^u (I) N_{u_I v} \big\|_\tau \\
& \leq \rho (q,q') \sum_u |x_u | \sum_{I : |I| \leq |R_0^+ |}
|I| \eta^{|I| - 1} \big\| \sum_v |y_v | N_{u_I v} \big\|_\tau \\
& \leq \rho (q,q') \sum_u |x_u | (\mc N (u) + 1)^{|R_0^+ |}
\,3 \max \big\{ 1,\eta^{|R_0^+ |} \big\} \norm{y}_\tau \\
& \leq \rho (q,q') C_\eta p_{b + |R_0^+ |} (x) \norm{y}_\tau .
\end{aligned}
\end{equation*}
Between lines 4 and 5 we used a small calculation like \eqref{eq:A.1} :
\begin{multline*}
\sum_{I : |I| \leq |R_0^+ |} |I| \eta^{|I| - 1} \: \leq \:
\sum_{j=0}^{|R_0^+ |} \left( \ds{l (u) \atop j} \right) j \eta^{j-1} \: \leq \:
{\ds \frac{l (u)!}{\big( l (u) - |R_0^+ | \big) !}}
\sum\limits_{j=0}^{|R_0^+ |} {\ds \frac{j}{j!}}
\max \big\{ 1,\eta^{|R_0^+ | - 1} \big\} \\
< \:  (l (u) + 1)^{|R_0^+ |} \,3 \max \big\{ 1,\eta^{|R_0^+ |} \big\} ,
\end{multline*}
and in the last line we may replace $l(u)$ by $\mc N (u)$.
\end{proof}

Now we want to show that $\mc S (\mc R ,q)$ really is an algebra. To this end
we will reconstruct it with an alternative but equivalent family of seminorms,
which are closer to being submultiplicative. Let $\mh C [W]^*$ be the algebraic dual
of $\mh C [W]$ and identify it with the space of all formal sums
$\sum_{w \in W} h_w N_w$. The norm $\mc N$ on $W$ induces an endomorphism
$\lambda (\mc N)$ of $\mh C [W]^*$ by
\[
{\ts \sum_{w \in W} h_w N_w \mapsto \sum_{w \in W} \mc N (w) h_w N_w} .
\]
This operator is unbounded on $L^2 (W)$ but it restricts to a continuous endomorphism
of $\mc S (W)$. For $T \in B(L^2 (W))$ we define the (in general unbounded) operator
$D(T) := [\lambda (\mc N) , T]$. Inspired by the work of Vign\'eras \cite[Section 7]{Vig}
we consider the following family of seminorms on $\mc H (\mc R ,q)$:
\[
p'_n (x) := \norm{D^n (\lambda (x))}_{B(L^2 (W))} \qquad n \in \mh Z_{\geq 0} .
\]
\begin{lem}\label{lem:A.seminorms}
The space $\mc S (\mc R ,q)$ is the completion of $\mc H (\mc R ,q)$ with respect to
the family of seminorms $\{ p'_n : n \in \mh Z_{\geq 0} \}$.
\end{lem}
\begin{proof}
We have to show that the families of seminorms $\{ p_n : n \in \mh Z_{\geq 0} \}$ and
$\{ p'_n : n \in \mh Z_{\geq 0} \}$ are equivalent. Let $\eta = \rho (q ,q^0) ,\,
n \in \mh N ,\, w \in W$ and $y = \sum_{v \in W} y_v N_v \in L^2 (W)$.
From the proof of Proposition \ref{prop:A.estopnorm} we see that
\begin{align*}
\norm{D^n (\lambda (N_u )) y}_\tau & = \big\| \sum_v y_v
\sum_{i=0}^n (-1)^i \left( \ds{n \atop i} \right) \lambda
(\mc N )^{n-i} \lambda (N_u ) \lambda (\mc N )^i N_v \big\|_\tau \\
& = \big\| \sum_v y_v \sum_{i=0}^n (-1)^i \left( \ds{n \atop i}
\right) \sum_I \eta_I D_v^u (I) \mc N (u_I v)^{n-i} \mc N (v)^i
N_{u_I v} \big\|_\tau \\
& = \big\| \sum_v y_v \sum_I \eta_I D_v^u (I) (\mc N (u_I v) -
\mc N (v))^n N_{u_I v} \big\|_\tau \\
& \leq \mc N (u)^n \big\| \sum_v |y_v| \sum_I |\eta_I |
D_v^u (I) N_{u_I v} \big\|_\tau \\
& \leq \mc N (u)^n 3 (\mc N (u) + 1)^{|R_0^+ |} \max \big\{1,
\eta^{|R_0^+ |} \big\} \big\| \sum_v |y_v| N_{u_I v} \big\|_\tau \\
& = \mc N (u)^n 3 (\mc N (u) + 1)^{|R_0^+ |} \max \big\{1,
\eta^{|R_0^+ |} \big\} \norm{y}_\tau .
\end{align*}
Hence, for $x = \sum_u x_u N_u \in \mc H (\mc R ,q)$
\begin{multline*}
\norm{D^n (\lambda (x)) }_{B(L^2 (W))} \: = \:
\big\| {\ts \sum_u} x_u D^n (\lambda (N_u)) \big\|_{B(L^2 (W))} \\
\leq \: {\ts \sum_u} |x_u | 3 (\mc N (u) + 1)^{n+ |R_o^+ |} \max
\big\{1, \eta^{|R_0^+ |} \big\} \: \leq \: C_\eta p_{n + b + |R_0^+ |} (x) .
\end{multline*}
On the other hand, since $\Omega' = \{ \omega \in W :
\mc N (\omega ) = 0 \}$ is finite,
\begin{align*}
p_n (x)^2 & \leq {\ts \sum_{u \in W}} (\mc N (u) + 1)^{2n} |x_u |^2 \\
& \leq {\ts \sum_{\omega \in \Omega'} |x_\omega |^2 +
  4^n \sum_{u \in W}} \mc N (u)^{2n}  |x_u |^2 \\
& \leq | \Omega' | \norm{x}^2_\tau +
  4^n \norm{\lambda (N)^n x}_\tau^2 \\
& = |\Omega' | \norm{ \lambda (x) N_e }_\tau^2 +
  4^n \norm{ D^n (\lambda (x)) N_e }_\tau^2 \\
& \leq |\Omega' | \norm{ \lambda (x) }^2_{B(L^2 (W))} +
  4^n \norm{ D^n (\lambda (x)) }^2_{B(L^2 (W))} \\
& \leq \big( |\Omega' |^{1/2} \norm{ \lambda (x) }_{B(L^2 (W))} +
  2^n \norm{ D^n (\lambda (x)) }_{B(L^2 (W))} \big)^2 ,
\end{align*}
which shows that $p_n$ is dominated by a linear combination of $p'_0$ and $p'_n$.
\end{proof}

\begin{thm}\label{thm:A.frealg}
\begin{enumerate}
\item $\mc S (\mc R ,q)$ is a Fr\'echet algebra.
\item $\mc S (\mc R ,q)^\times$ is open in $\mc S (\mc R,q)$ and inverting is
a continuous map from this set to itself.
\item An element of $\mc S (\mc R ,q)$ is invertible if and only if it is
invertible in $C^*_r (\mc H (\mc R ,q))$.
\item The subalgebra $\mc S (\mc R,q) \subset C^*_r (\mc H (\mc R,q))$ is closed
under the holomorphic functional calculus of $C^*_r (\mc H (\mc R,q))$.
\end{enumerate}
\end{thm}
\begin{proof}
(1) We already observed that $\mc S (\mc R,q)$ is a Fr\'echet space. Because $D$
is a derivation, $\mc S (\mc R,q)$ is also a topological algebra with jointly
continuous multiplication. A short calculation shows that the norm
$\sum^m_{n=0} \frac{1}{n!} p'_n$ on $\mc S (\mc R,q)$ is submultiplicative
for any $m \in \mh Z_{\geq 0}$. The family
$\big\{ \sum^m_{n=0} \frac{1}{n!} p'_n : m \in \mh Z_{\geq 0} \big\}$
is equivalent to $\{ p'_n : n \in \mh Z_{\geq 0} \}$, so defines the same topology.

(2) and (3) See Lemmas 16 and 17 of \cite{Vig}.

(4) This is a consequence of part (3) and Theorem \ref{thm:A.contour}.
\end{proof}

Our next goal is to show that inverting in $\mc S (\mc R ,q)$ also depends continuously
on $q \in \mc Q$. For this we need two preparatory lemmas. Put
$b' = 2b + |R_0^+| = 2 \, \mr{rk}(X) + |R_0^+| + 2$.

\begin{lem}\label{lem:A.estproducts}
Let $n \in \mh N ,\, q,q' \in B_\rho (q^0 ,\eta )$ and
$x_i = \sum_{u \in W} x_{iu} N_u \in \mc S (W)$. Then
\[
\begin{array}{lcr}
p_n (x_1 \cdot_q \cdots \cdot_q x_m ) & \leq &
 \prod\limits_{i=1}^m C_\eta C_b p_{n + b'}(x_i ) , \\
p_n (x_1 \cdot_q \cdots \cdot_q x_m - x_1 \cdot_{q'} \cdots
  \cdot_{q'} x_m) & \leq &
 \rho(q,q') \prod\limits_{i=1}^m C_\eta C_b p_{n + b'}(x_i) .
\end{array}
\]
\end{lem}
\begin{proof}
This can be deduced with a piece of careful bookkeeping:
\[
\begin{array}{ll}
p_n (x_1 \cdot_q \cdots \cdot_q x_m ) & \leq \\
  p_n( \sum_{u_i \in W} x_{1 u_1} \cdots x_{m u_m}
  N_{u_1} \cdot_q \cdots \cdot_q N_{u_m}) & \leq \\
\sum_{u_i \in W} | x_{1 u_1} \cdots x_{m u_m} | (\mc N (u_1) +
  \cdots + \mc N (u_m) + 1)^n \prod_{i=1}^m
  \norm{(N_{u_i},q)}_o & \leq \\
\sum_{u_i \in W} | x_{1 u_1} \cdots x_{m u_m} | \prod_{i=1}^m
  C_\eta (\mc N (u_i) + 1)^{n + b + |R_0^+ |} & = \\
\prod_{i=1}^m C_\eta \sum_{u \in W} |x_{i u}|
  (\mc N(u) +1)^{n + b + |R_0^+ |} & \leq \\
\prod_{i=1}^m C_\eta C_b p_{n + b'}(x_i) \\
 & \\
p_n (N_{u_1} \cdot_q \cdots \cdot_q N_{u_m} - N_{u_1}
  \cdot_{q'} \cdots \cdot_{q'} N_{u_m}) & \leq \\
\sum_{j=1}^{m-1} p_n(N_{u_1} \cdot_q \cdots \cdot_q N_{u_j}
  \cdot_q N_{u_{j+1}} \cdot_{q'} \cdots \cdot_{q'} N_{u_m} -& \\
\qquad \qquad N_{u_1} \cdot_q \cdots \cdot_q N_{u_j} \cdot_{q'}
  N_{u_{j+1}} \cdot_{q'} \cdots \cdot_{q'} N_{u_m}) & \leq \\
\sum_{j=1}^{m-1} \rho(q,q') \prod_{i=1}^m C_\eta
  (\mc N (u_i) + 1)^{n + b + |R_0^+ |} & \leq \\
\rho(q,q') \prod_{i=1}^m C_\eta
  (\mc N (u_i) + 1)^{n + b + |R_0^+ |}  & \\
 & \\
p_n (x_1 \cdot_q \cdots \cdot_q x_m - x_1 \cdot_{q'} \cdots
  \cdot_{q'} x_m) & \leq \\
\sum_{u_i \in W} | x_{1 u_1} \cdots x_{m u_m} | p_n(N_{u_1}
  \cdot_q \cdots \cdot_q N_{u_m} - N_{u_1} \cdot_{q'}
  \cdots \cdot_{q'} N_{u_m}) & \leq \\
\sum_{u_i \in W} | x_{1 u_1} \cdots x_{m u_m} | \rho(q,q')
  \prod_{i=1}^m C_\eta (\mc N(u_i) + 1)^{n + b + |R_0^+ |}
 & = \\
\rho(q,q') \prod_{i=1}^m C_\eta \sum_{u \in W}
  |x_{i u}| (\mc N(u) +1)^ {n + b + |R_0^+ |} & \leq \\
\rho(q,q') \prod_{i=1}^m C_\eta p_{n + b'} (x_i) &
\end{array}
\]
In these calculations we used \eqref{eq:A.2} and Proposition
\ref{prop:A.estopnorm} several times.
\end{proof}

Knowing how to handle multiple products in $\mc S (\mc R ,q)$, we can
make some rough estimates for power series.
Let $f : z \mapsto \sum_{m=0}^\infty a_m z^m$ be a holomorphic
function on a neighborhood of $0 \in \mh C$ and define another
holomorphic function $\tilde f$ (with the same radius of
convergence) by $\tilde f (z) := \sum_{m=0}^\infty |a_m| z^m$.

\begin{lem}\label{lem:A.estseries}
Let $n \in \mh N ,\; x \in \mc S(W)$ and $q,q' \in
B_\rho (q^0,\eta)$ be such that $f(x,q)$ and $f(x,q')$
are well-defined. Then
\[
\begin{array}{lcr}
p_n (f(x,q)) & \leq &
  \tilde f \big( C_\eta C_b p_{n + b'} (x) \big) , \\
p_n (f(x,q) - f(x,q')) & \leq & \rho(q,q')
  \tilde f \big( C_\eta C_b p_{n + b'} (x) \big) .
\end{array}
\]
\end{lem}
\begin{proof}
By Proposition \ref{lem:A.estproducts} we have
\begin{align*}
p_n (f(x,q)) & =
  p_n \big( {\ts \sum_{m=0}^\infty} a_m (x,q)^m \big) \\
 & \leq {\ts \sum_{m=0}^\infty} |a_m| p_n ((x,q)^m ) \\
 & \leq {\ts \sum_{m=0}^\infty} |a_m| \big( C_\eta C_b
p_{n + b'}(x) \big)^m \\
 & = \tilde f \big( C_\eta C_b p_{n + b'}(x) \big) .
\end{align*}
\begin{align*}
p_n (f(x,q) - f(x,q')) & = p_n \big( {\ts \sum_{m=0}^\infty} a_m
  ((x,q)^m - (x,q')^m) \big) \\
 & \leq {\ts \sum_{m=0}^\infty} |a_m| p_n \big( (x,q)^m - (x,q')^m \big) \\
 & \leq {\ts \sum_{m=0}^\infty} |a_m| \rho(q,q')
  \big( C_\eta C_b p_{n + b'}(x) \big)^m \\
 & = \rho(q,q') \tilde f \big( C_\eta C_b p_{n + b'}(x) \big) .
\end{align*}
The right hand sides could be infinite, but that is no problem.
\end{proof}

\begin{prop}\label{prop:A.inversion}
The set of invertible elements $\bigcup_{q \in \mc Q} \mc S
(\mc R ,q)^\times \times \{q\}$ is open in $\mc S (W ) \times \mc Q$,
and inverting is a continuous map from this set to itself.
\end{prop}
\begin{proof}
First we recall that if $\norm{1-h}_o < 1$, then $h$ is
invertible in $C_r^* (\mc R ,q)$, with inverse
$\sum_{n=0}^\infty (1-h)^n$. Take $q,q' \in B_\rho (q^0, \eta)
,\, y \in \mc S (\mc R) ,\, x \in \mc S (\mc R ,q)^\times$ and
write $a = (x,q)^{-1}$. If the sum converges, then
\begin{equation}\label{eq:A.3}
a \cdot_{q'} {\ts \sum_{m=1}^\infty} (1 - (x+y) \cdot_{q'} a,q')^m
= a \cdot_{q'} ((x+y) \cdot_{q'} a,q')^{-1} - a \cdot_{q'} 1
= (x+y,q')^{-1} - a .
\end{equation}
By Lemma \ref{lem:A.estproducts}
\begin{multline}\label{eq:A.4}
p_n( (x+y) \cdot_{q'} a - 1 ) \: \leq \:
p_n( x \cdot_{q'} a - x \cdot_q a ) + p_n( y \cdot_{q'} a ) \\
\leq \: \rho(q,q') C_\eta^2 C_b^2 p_{n+b'}(x) p_{n+b'}(a) +
C_\eta^2 C_b^2 p_{n+b'}(y) p_{n+b'}(a) .
\end{multline}
Let $U$ be the open neighborhood of $(x,q)$ consisting of those
$(x+y,q') \in \mc S (W) \times B_\rho (q^0 ,\eta)$ for which
\begin{align*}
\rho(q,q') C_\eta^3 C_b^2 p_{3b + |R_0 |}(x) p_{3b + |R_0 |}(a)
& < 1/2 , \\
C_\eta^3 C_b^2 p_{3b + |R_0 |}(y) p_{3b + |R_0 |}(a) & < 1/2 .
\end{align*}
By \eqref{eq:A.4} and Proposition \ref{prop:A.estopnorm} we have
\[
\norm{((x+y) \cdot_{q'} a - 1, q')}_o < 1 \qquad
\text{for all } (x+y,q') \in U ,
\]
so every element of $U$ is invertible. To prove that inverting is
continuous we consider the holomorphic function
$f(z) = \sum_{m=1}^\infty z^m = z / (1-z)$ .
By \eqref{eq:A.3} and Lemma \ref{lem:A.estseries} we have
\begin{align*}
p_n( (x+y,q')^{-1} - a) & \leq C_b^2 C_\eta^2 p_{n+b'}(a)
p_{n+b'}\big( f(1 - (x+y) \cdot_{q'} a,q') \big) \\
 & \leq C_b^2 C_\eta^2 p_{n+b'}(a) f \big( C_b C_\eta
  p_{n+2b'} (1 - (x+y) \cdot_{q'} a) \big) .
\end{align*}
Since $f(0) = 0$ we deduce from \eqref{eq:A.4} that this
expression is small whenever $\rho (q,q')$ and $y$ are small.
\end{proof}

With Proposition \ref{prop:A.inversion} we can prove that the holomorphic
functional calculus in the various Schwartz algebras is continuous in the
most general sense. For $U \subset \mh C$ we write
\[
V_U := \{ (x,q) \in \mc S(W) \times \mc Q : \mr{Sp}(x,q) \subset U \} .
\]
\begin{thm}\label{thm:A.funcalc}
Let $U \subset \mh C$ be open. Then $V_U$ is open in $\mc S(W) \times \mc Q$
and the map
\[
C^{an}(U) \times V_U \to \mc S(W) : (f,x,q) \mapsto f(x,q)
\]
is continuous.
\end{thm}
\begin{proof}
By Theorem \ref{thm:A.frealg}.4 the spectrum of $(x,q)$ in $\mc S (\mc R,q)$
equals its spectrum in the unital $C^*$-algebra $C^*_r (\mc H (\mc R ,q))$.
By Proposition \ref{prop:A.estopnorm} $(x,q) \mapsto \norm{(x,q)}_o$ is
continuous, so Sp$(x,q)$ is uniformly bounded on bounded subsets of
$\mc S (W) \times \mc Q$. Together with Proposition \ref{prop:A.inversion} this shows
that Sp$(x,q)$ depends continuously on $(x,q)$, in the following sense. Given
$\epsilon > 0$, there exists a neighborhood $N$ of $(x,q)$ in $\mc S (W) \times \mc Q$,
such that for all $(x',q') \in N$:
\[
\mr{Sp}(x',q') \subset \{ z' \in \mh C : \exists z \in \mr{Sp}(x,q) : |z - z'| < \epsilon \} .
\]
Since Sp$(x,q)$ is compact, it follows that $V_U$ is open in $\mc S (W) \times \mc Q$.

For every connected component $U_j$ of $U$ that meets Sp$(x,q)$, let $\Gamma_j$
be a positively oriented smooth simple contour closed in $U_j$ that encircles
Sp$(x,q) \cap U_j$, as in Theorem \ref{thm:A.contour}. Since Sp$(x,q)$ is compact,
we need only finitely many components. The above shows that $\Gamma_j$ also
encircles Sp$(x',q') \cap U_j$ for $(x',q')$ in a small neighborhood of $(x,q)$ in
$\mc S (W) \times \mc Q$. Now Theorem \ref{thm:A.contour} tells us that
\[
 f(x',q') = \frac{1}{2 \pi i} \sum_j \int_{\Gamma_j} f(z) (z-x' ,q' )^{-1} \textup{d}z
\]
for all such $(x',q')$, so by \ref{prop:A.inversion} $(f,x',q') \mapsto f(x',q')$ is continuous.
\end{proof}

\printindex


\begin{thebibliography}{99}
\bibitem[ABP]{ABP} A.-M. Aubert, A.-M., Baum, P., Plymen, R.,
The Hecke algebra of a reductive p-adic group: a geometric
conjecture. Aspects of Mathematics 37, Vieweg
Verlag (2006) 1--34.

\bibitem[BM1]{BM0} Barbasch, D., Moy, A.,
A unitarity criterion for $p$-adic groups.
Invent. Math. 98 (1989), no. 1, 19--37.

\bibitem[BM2]{BM} Barbasch, D., Moy, A.,
Reduction to real infinitesimal character in affine Hecke
algebras. Journal of the AMS 6 (1993), no. 3, 611--630.

\bibitem[BZ]{BZ} Bernstein, J., Zelevinski, V.,
Induced representations on reductive $p$-adic groups.
Ann. Sci. Ec. Norm. Sup. 10 (1977), 441--472.

\bibitem[BKR]{BKR} Blackadar, B., Kumjian, A., R{\o}rdam, M.,
Approximately central matrix units and the structure of
noncommutative tori.
$K$-Theory 6 (1992), no. 3, 267--284.

\bibitem[Bl]{Blo} Blondel, C.,
Propagation de paires couvrantes dans les groupes symplectiques.
Representation Theory 10 (2006), 399--434.

\bibitem[Bo]{Bo} Borel, A.,
Admissible representations of a semisimple group over a local field
with vectors fixed under an Iwahori subgroup.
Invent. Math. 35 (1976), 233--259.

\bibitem[BHK]{BHK} Bushnell, C.J., Henniart, G., Kutzko, P.C.,
Towards an explicit Plancherel formula for reductive $p$-adic
groups. Preprint 2005.

\bibitem[BK]{BK1} Bushnell, C.J., Kutzko, P.C.,
Types in reductive $p$-adic groups: the Hecke algebra
of a cover.
Proc. Amer. Math. Soc. 129 (2001),
no. 2, 601--607.

\bibitem[Ca]{C} Carter, R.W.,
Finite groups of Lie type.
Wiley Classics Library, John Wiley
and sons, Chichester, UK, 1993.

\bibitem[Ch]{Ch2} Cherednik, I.V.,
Double affine Hecke algebras and Macdonald's conjectures.
Annals of Math. 141 (1995), 191--216.

\bibitem[DO]{DeOp1} Delorme, P.,  Opdam, E.M.,
The Schwartz algebra of an affine Hecke algebra.
J. reine angew. Math. 625 (2008), 59--114

\bibitem[EOS]{EOS} Emsiz, E., Opdam, E.M., Stokman, J.V.,
Periodic integrable systems with delta-potentials.
Comm. Math. Phys. 264 (2006), no. 1, 191--225.

\bibitem[Ge]{Geck} Geck, M.,
On the representation theory of Iwahori-Hecke
algebras of extended finite Weyl groups.
Representation Theory 4 (2000), 370--397.

\bibitem[HO1]{HO} Heckman, G.J., Opdam, E.M.,
Yang's system of particles and Hecke algebras.
Ann. of Math. 145 (1997), 139--173.

\bibitem[HO2]{HOH} Heckman, G.J.,  Opdam, E.M.,
Heckman, G.J., Opdam, E.M.,
Harmonic analysis for affine Hecke algebras, in
Current developments in mathematics (Cambridge, MA, 1996), pp. 37--60.
Int. Press, Boston, MA, 1997.

\bibitem[Hu]{Hum1} Humphreys, J.E.,
Introduction to Lie algebras and Representation Theory.
GTM 9, 3rd ed. Springer Verlag, 1980.

\bibitem[IM]{IwMa} Iwahori, N., Matsumoto, H.,
On some Bruhat decomposition and the structure
of the Hecke rings of the $p$-adic Chevalley groups.
Inst. Hautes \'Etudes Sci. Publ. Math. 25 (1965), 5--48.

\bibitem[Kat2]{Kat2} Kato, S.,
An exotic Deligne-Langlands correspondence for symplectic groups.
Preprint, 2006, To appear in Duke Math. J.

\bibitem[KL]{KL} Kazhdan, D., Lusztig, G.,
Proof of the Deligne-Langlands conjecture for affine Hecke algebras.
Invent. Math. 87 (1987), 153--215.

\bibitem[Lu1]{Lus1} Lusztig, G.,
Cells in affine Weyl groups,
in Algebraic groups and related topics (Kyoto/Nagoya, 1983), pp. 255--287.
Adv. Stud. Pure Math., 6, North-Holland, Amsterdam, 1985.

\bibitem[Lu2]{Lus2} Lusztig, G.,
Affine Hecke algebras and their graded version.
J. Amer. Math. Soc. 2 (1989), 599--635.

\bibitem[Lu3]{Lu4} Lusztig, G.,
Classification of unipotent representations of simple
$p$-adic groups.
Internat. Math. Res. Notices 11 (1995), 517--589.

\bibitem[Lu4]{LuCL} Lusztig, G.,
Cuspidal local systems and graded Hecke algebras II.
in Representations of groups (Banff, AB, 1994), pp. 217--275.
CMS Conf. Proc., 16,
Amer. Math. Soc., Providence, RI, 1995.

\bibitem[Lu5]{Lus3} Lusztig, G.,
Cuspidal local systems and graded Hecke algebras III.
Representation Theory 6 (2002), 202--242.

\bibitem[Lu6]{Lu5} Lusztig, G.,
Hecke algebras with Unequal Parameters.
CRM Monograph Series 18,
Amer. Math. Soc., Providence RI, 2003.

\bibitem[Mac1]{Ma2} Macdonald, I.G.,
Spherical functions on a group of $p$-adic type.
Publications of the Ramanujan Institute, No. 2.
Ramanujan Institute,
Centre for Advanced Study in Mathematics,
University of Madras, Madras, 1971.

\bibitem[Mat]{Mat} Matsumoto, H.,
Analyse harmonique dans les syst\'emes de Tits
bornologiques de type affine.
Springer Lecture Notes 590 (1977).

\bibitem[Mo1]{Mo1} Morris, L.,
Tamely ramified intertwining algebras.
Invent. Math. 114 (1993), 233--274.

\bibitem[Mo2]{Mo2} Morris, L.,
Level zero $G$-types.
Compositio Math. 118 (1999), no. 2, 135--157.

\bibitem[Op1]{Opd1} Opdam, E.M.,
On the spectral decomposition of affine Hecke algebras.
J. Inst. Math. Jussieu 3 (2004), no. 4, 531--648.

\bibitem[Op2]{Opd2} Opdam, E.M.,
Opdam, Eric M. Hecke algebras and harmonic analysis.
International Congress of Mathematicians. Vol. II,
1227--1259, Eur. Math. Soc., Z\"urich, 2006.

\bibitem[Op3]{Opd3} Opdam, E.M.,
The central support of the Plancherel measure of an
affine Hecke algebra.
Moscow Mathematical Journal 7 (2007), no. 4, 723--741.

\bibitem[OS]{OpdSol} Opdam E.M., and Solleveld, M.S.,
Homological algebra for affine Hecke algebras.
Adv. in Math. 220 (2009), 1549--1601.

\bibitem[Phi]{Phi} Phillips, N.C.,
$K$-theory for Fr\'echet algebras.
Internat. J. Math. 2 (1991), no. 1, 77--129.

\bibitem[RR]{RamRam} Ram, A., Ramagge, J.,
Affine Hecke algebras,
cyclotomic Hecke algebras and Clifford theory.
A tribute to C. S. Seshadri (Chennai, 2002), 428--466,
Trends Math., Birkh\"auser, Basel, 2003.

\bibitem[Re1]{Re} Reeder, M.,
Formal degrees and L-packets of unipotent
discrete series of exceptional $p$-adic groups.
with an appendix by Frank L\"ubeck
J. reine angew. Math. 520 (2000), 37--93.

\bibitem[Re2]{Ree} Reeder, M.,
Euler-Poincar\'e pairings and
elliptic representations of Weyl groups and $p$-adic groups.
Compositio Math. 129 (2001), 149--181.

\bibitem[SS]{ScSt} Schneider, P., Stuhler, U.,
Representation theory and sheaves on the Bruhat-Tits building.
Publ. Math. Inst. Hautes \'Etudes Sci. 85 (1997), 97--191.

\bibitem[Slo1]{SlootenThesis} Slooten, K.,
A combinatorial
generalization of the Springer correspondence for classical type.
PhD Thesis, Universiteit van Amsterdam (2003),
http://dare.uva.nl/en/record/119136.

\bibitem[Sl2]{Slooten} Slooten, K.,
Generalized Springer correspondence and Green functions for type B/C
graded Hecke algebras.
Advances in Mathematics 203 (2006), 34--108.

\bibitem[So]{Sol} Solleveld, M.S.,
Periodic cyclic homology of affine Hecke algebras.
PhD Thesis, Universiteit van Amsterdam (2007),
http://dare.uva.nl/en/record/217308.

\bibitem[Ta]{Tak} Takesaki, M.,
Theory of operator algebras I.
Springer-Verlag, New York, 1979

\bibitem[Vi]{Vig} Vign\'eras, M.-F.,
On formal dimensions for reductive $p$-adic groups,
in Festschrift in honor of I. I. Piatetski-Shapiro on the occasion of his
sixtieth birthday, Part I (Ramat Aviv, 1989), pp. 225--266.
Israel Math. Conf. Proc., 2,
Weizmann, Jerusalem, 1990.

\bibitem[Wa]{W} Waldspurger, J.-L.,
La formule de Plancherel pour les groupes p-adiques
(d\'\,apr\`es Harish-Chandra),
J. Inst. Math. Jussieu  2 (2003), no.2, 235--333.
\end{thebibliography}
\end{document}